\documentclass[amsfonts]{amsart}
\usepackage{esint}
\usepackage{tgpagella}
\usepackage[bottom]{footmisc}
\usepackage[most]{tcolorbox}
\usepackage[english]{babel}
\usepackage{inputenc}
\usepackage{mathabx}
\usepackage{amsmath,amsthm,amsfonts,amssymb,amscd}
\usepackage{mathpazo}
\usepackage{amssymb,amsfonts}
\usepackage[all,arc]{xy}
\usepackage{enumerate}
\usepackage{mathrsfs}
\usepackage{tgtermes}
\usepackage{marginnote}
\usepackage{mathtools}
\usepackage{amsthm,thmtools,xcolor}

\usepackage[all]{xy}
\usepackage{tikz-cd}
\usetikzlibrary{cd}
\theoremstyle{plain}
\newtheorem{thm}{Theorem}[section]
\newtheorem{cor}{Corollary}[section]
\newtheorem{prop}{Proposition}
\newtheorem{lem}{Lemma}[section]

\newtheorem{lemx}{Lemma}

\theoremstyle{definition}
\newtheorem{defn}{Definition}[section]

\theoremstyle{remark}
\newtheorem{rem}{\textit{Remark}}[section]

\numberwithin{equation}{section}
\usepackage{color}
\usepackage{chngcntr}
\usepackage{apptools}
\AtAppendix{counterwithin{Lemma}{section}}

\usepackage[colorlinks = true,
linkcolor = blue,
urlcolor  = blue,
citecolor = blue,
anchorcolor = blue]{hyperref}
\usepackage{hyperref}
\makeatletter
\let\c@equation\c@thm
\makeatother
\numberwithin{equation}{section}

\DeclareMathOperator{\dist}{\mathrm{dist}}

\DeclareMathOperator{\supp}{\mathrm{supp}}
\newcommand\underrel[3][]{\mathrel{\mathop{#3}\limits_{%
			\ifx c#1\relax\mathclap{#2}\else#2\fi}}}
\title[Fractional Zakharov-Kuznetsov]{On  Kato's smoothing effect   for a fractional  version of the Zakharov-Kuznetsov equation}

\author{Argenis. J. Mendez}

\address{Instituto de Matemáticas, Pontificia Universidad Católica de Valparaíso, Blanco Viel 596, Cerro Barón, Valparaíso.}

\email{argenis.mendez@pucv.cl}

\thanks{}

\thanks{}


\date{\today}
\commby{A.J.  Mendez}
\begin{document}
\begin{abstract}
In this work we study some regularity properties associated to the initial value problem (IVP) 
\begin{equation}\label{main1}
\left\{
\begin{array}{ll}
\partial_{t}u-\partial_{x_{1}}(-\Delta)^{\alpha/2} u+u\partial_{x_{1}}u=0, \quad  0< \alpha\leq 2,& \\
u(x,0)=u_{0}(x),\quad   x=(x_{1},x_{2},\dots,x_{n})\in \mathbb{R}^{n},\, n\geq 2,\quad t\in\mathbb{R},& \\
\end{array} 
\right.
\end{equation}
where $(-\Delta)^{\alpha/2}$ denotes the $n-$dimensional fractional Laplacian.

We show that solutions to the IVP \eqref{main1} with initial data in a suitable Sobolev space exhibit a local smoothing effect in the spatial variable of $\frac{\alpha}{2}$   derivatives, almost everywhere in time.  One of the main difficulties that emerge when trying to obtain this regularizing effect underlies that the operator in consideration is non-local, and the property we are trying to describe is local, so new ideas are required.
Nevertheless, to avoid these problems, we use a perturbation argument  replacing $(-\Delta)^{\frac{\alpha}{2}}$ by $(I-\Delta)^{\frac{\alpha}{2}},$ that through the use of pseudo-differential calculus allows us to show that solutions become  locally smoother by $\frac{\alpha}{2}$ of a derivative in all spatial directions.

As a by-product, we use this particular smoothing effect to show that the extra regularity of the initial data on some distinguished subsets of the Euclidean space is propagated by the flow solution with infinity speed.
\end{abstract}

\maketitle
%\tableofcontents
\section{Introduction}
In this work we study  the   \emph{Fractional Zakharov-Kuznetsov-(FZK) }equation
\begin{equation}\label{zk4}
	\left\{
	\begin{array}{ll}
		\partial_{t}u-\partial_{x_{1}}(-\Delta)^{\frac{\alpha}{2}}u+u\partial_{x_{1}}u=0, & 0<\alpha< 2, \\
		u(x,0)=u_{0}(x), \quad x=(x_{1},x_{2},\dots,x_{n})\in \mathbb{R}^{n},& t\in\mathbb{R}, \\
	\end{array} 
	\right.
\end{equation}
where  $u=u(x,t)$ is a real valued function  and $\left(-\Delta\right)^{\frac{\alpha}{2}}$ stands for the \emph{fractional Laplacian} whose  description in the Fourier space is given by 
\begin{equation*}
	\mathcal{F}\left((-\Delta)^{\alpha/2}f\right)(\xi):=|2\pi\xi|^{\alpha}\mathcal{F}(f)(\xi)\quad f\in\mathcal{S}(\mathbb{R}^{n}).
\end{equation*}
The FZK equation  formally satisfies the following conservation laws, at least for smooth solutions
\begin{equation*}
	\mathcal{I}u(t)=	\int_{\mathbb{R}^{n}}u(x,t)\, dx=\mathcal{I}(0),
\end{equation*}
\begin{equation*}
	\mathcal{M}(t)	=\int_{\mathbb{R}^{n}}u^{2}(x,t)\, dx=\mathcal{M}(0),
\end{equation*}
and the Hamiltonian
\begin{equation*}
	\mathcal{H}(t)=\frac{1}{2}\int_{\mathbb{R}^{n}}\left(\left(-\Delta\right)^{\frac{\alpha}{2}}u(x,t)\right)^{2} \, dx -\frac{1}{6}\int_{\mathbb{R}^{n}}u^{3}(x,t) \, dx=\mathcal{H}(0).
\end{equation*}
In the case $\alpha=1$  the equation   \eqref{zk4}   can be formally rewritten as
\begin{equation}\label{shrira}
	\partial_{t}u-\mathcal{R}_{1}\Delta u+u\partial_{x_{1}}u=0,
\end{equation}
where  $\mathcal{R}_{j}$ denotes the \emph{Riesz transform} in the $x_{1}-$variable, that is,
\begin{equation*}
	\mathcal{R}_{1}(f)(x):=\frac{\Gamma\left(\frac{n+1}{2}\right)}{\pi^{\frac{n+1}{2}}}\,\mathrm{p.v.}\int_{\mathbb{R}^{n}}\frac{x_{1}-y_{1}}{|x-y|^{n+1}}f(y)\,dy,
\end{equation*}
whenever $f$ belongs to a suitable class of functions.

The model \eqref{shrira} seems  to be     firstly deduced by Shrira \cite{Schira} 
  to describe  the  bi-dimensional  long-wave perturbations in a boundary-layer
type shear flow. More precisely,  \eqref{shrira}  represents the equation of the longitudinal velocity of fluid under certain conditions (see Shrira \cite{Schira} for a more detailed description). We shall also refer the works \cite{astep}, \cite{Dyachenko},\cite{Scrira pelynosvky} and \cite{PSTEP} where several variants  (either be extensions or  reductions) of the  model \eqref{shrira} has been studied. 

We shall also point out that  the equation in \eqref{shrira}  represents a higher dimensional  extension of the famous  \emph{Benjamin-Ono} equation 
\begin{equation*}
	\partial_{t}u-\mathcal{H}\partial_{x}^{2}u+u\partial_{x}u=0,
\end{equation*}
where $\mathcal{H}$ denotes the Hilbert transform.

The equation \eqref{shrira} has called the attention in  recent years and  several results concerning local well-posedness have been established.  In this sense, we    can mention the work of Hickman, Linares, Ria\~{n}o, Rogers, Wright \cite{HLKW}, who establish local well-posedness  of the IVP associated to \eqref{shrira} in the Sobolev space $H^{s}(\mathbb{R}^{n}),$ where $s>\frac{5}{3}$   in the bi-dimensional case  and $s>\frac{n}{2}+\frac{1}{2}$  whenever $n\geq 3.$ Also, Schippa \cite{RS}   improves the  work in \cite{HLKW},   by showing that \eqref{zk4} is locally well-posed in $H^{s}(\mathbb{R}^{n})$ for $s>\frac{n+3}{2}-\alpha$  and $1\leq \alpha<2.$  See also  Riaño \cite{oscarmari} where  some results   of local well-posedness in  weighted Sobolev spaces are established, as well as,  some unique continuation principles for equation \eqref{shrira}.

The case  $\alpha=2$ in  \eqref{zk4} corresponds to the \emph{Zakharov-Kuznetsov-(ZK) } equation
\begin{equation}\label{zkeq}
	\partial_{t}u+\partial_{x_{1}}\Delta u+u\partial_{x_{1}}u=0.
\end{equation}
Originally \eqref{zkeq}  was derived by Zakharov, Kuznetsov \cite{ZAKHARIV} in the three-dimensional case as a model to describe the propagation of ionic-acoustic waves in magnetized plasma. More recently,  Lannes, Linares, and Saut \cite{lls} justify that the ZK equation can be formally deduced as a long wave small-amplitude limit of the  Euler-Poisson system in the  \textquotedblleft cold plasma  \textquotedblright 
  \, approximation. From the physical point of view the  ZK model   is not only interesting  in the 3-dimensional  case but also in the 2-dimensional  since  \textit{e.g.}   it   describes under  certain conditions the  amplitude of long waves  on the free surface  of a thin film in a specific  fluid with  particular parameters of  viscosity (see Melkonian, Maslowe \cite{MM}  for more details)

The FZK equation has been little  studied  except  in the distinguish cases  we pointed out  above, that is, $\alpha=1$ and $\alpha=2.$  Nevertheless,   results of local and global well-posedness  of the  FZK equation  in the range   $\alpha\in(0,2)-\{1,2\}$ are  scarce. To our knowledge, we can only mention the work of Schippa  \cite{RS} where the well-posedness problem in $H^{s}(\mathbb{R}^{n})$ is addressed.     In  \cite{RS} is proved  that the IVP \eqref{zk4} is locally  well-posed in $H^{s}(\mathbb{R}^{n})$  for $s>\frac{n+3}{2}-\alpha$ whenever $n>2$ and $ \alpha\in [1,2).$    See also \cite{RS} for  additional results of local well-posedness on $H^{s}(\mathbb{T}^{n}).$

In  contrast, the ZK equation has been the object of intense study in the  recent years, this has lead  to a enormous  improvements concerning the local and global well posedness. In this sense  we could mention the work of Faminski\u{\i} \cite{FAMI1} who shows local well posedness in $H^{m}(\mathbb{R}^{2}),m\in\mathbb{N}.$ Also in the $2d$ case Linares and Pastor \cite{LIPAS} prove local well-posedness in $H^{s}(\mathbb{R}^{2})$ for $s>\frac{3}{4}.$ Later,  Molinet and Pilod \cite{Molipilo}  and Gr\"{u}nrock and Herr \cite{GRUHER} extend the local well-posedness to $H^{s}(\mathbb{R}^{2}), \, s>\frac{1}{2},$ by using  quite similar arguments  based on the Fourier restriction method.

Regarding  the $3d$ case,  Molinet, Pilod \cite{Molipilo} and also Ribaud and Vento \cite{RIBAUDVENT} prove local well posedness in $H^{s}(\mathbb{R}^{3}),\, s>1.$  In this  direction, the most recent  work of Kinoshita \cite{KINO} and Herr, Kinoshita \cite{herkino} establish well-posedness in the best possible Sobolev range where the  Picard iteration scheme can be applied, that is, $H^{s}(\mathbb{R}^{2}),\, s>-\frac{1}{4}$ and  $H^{s}(\mathbb{R}^{n}),\, s>\frac{n-4}{2}$ when $n>2.$

In this work, we do not pursuit to improve  results concerning local or global well-posedness, but  on the contrary we  need  to establish    local well-posedness  on  certain Sobolev space that allow us   to describe  particular properties of the solutions.  The energy method  proved  by Bona and Smith \cite{Bonasmith} yields local well-posedness in $H^{s}(\mathbb{R}^{n})$ for $s>\frac{n+2}{2}.$
More precisely,   the following result holds:
\begin{thm}\label{lwp}
	Let $s>s_{n}$ where $s_{n} := \frac{n+2}{2} $. Then, for any $u_{0} \in  H^{s}(\mathbb{R}^{n}),$ there exist $T = T(\|u_{0}\|_{H^{s}})>0$ and a unique solution $u$ to the IVP \eqref{zk4} such that 
	\begin{equation}\label{conditions}
		u\in C\left([0,T]: H^{s_{n}}(\mathbb{R}^{n})\right)\quad \mbox{and}\quad \nabla u\in L^{1}\left((0,T): L^{\infty}(\mathbb{R}^{n})\right).
	\end{equation}
	Moreover, the flow map $u_{0}\longmapsto u$  defines a continuous application from  $H^{s_{n}}(\mathbb{R}^{n})$ into $H^{s_{n}}(\mathbb{R}^{n})$. 
\end{thm}
The energy method in this case does not consider the effects of the dispersion  and it is mainly based on   \emph{a priori estimates}    for smooth solutions, that 
 combined with Kato-Ponce commutator estimate ( see Theorem \ref{KPDESI}) give 
\begin{equation*}
	\|u\|_{L^{\infty}_{T}H^{s}_{x}}\lesssim \|u_{0}\|_{H^{s}_{x}}e^{\|\nabla u\|_{L^{1}_{T}L^{\infty}_{x}}}.
\end{equation*}
The  condition   on the gradient in \eqref{conditions} as the reader  can see later is  fundamental in  the  solution of the problems we address in this  work.

A quite remarkable property that  dispersive equations satisfies is  \emph{Kato's smoothing effect}, this is a property  found by Kato \cite{KATO1} in the context of the  \emph{Korteweg-de Vries-(KdV)} equation
\begin{equation}\label{kdv}
	\partial_{t}u+\partial_{x}^{3}u+u\partial_{x}u=0.
\end{equation} 
% that is a model  to describe the unidirectional  propagation  on nonlinear dispersive  long waves \cite{KDV}.
In \cite{KATO1} Kato proves that solutions of the  KdV equation \eqref{kdv}  becomes more regular locally by one derivative with respect to the initial data, that  is, if $u$ is a solution of \eqref{kdv} on a suitable Sobolev space then:  for any $r>0,$%the solution $u$  verify $u\in L^{2}_{T}H^{1}_{\mathrm{loc}}
\begin{equation}\label{smoothingkdv}
	\int_{0}^{T}\int_{-r}^{r}\left(\partial_{x}u(x,t)\right)^{2}\, dx\, dt< c(T,r)\|u_{0}\|_{L^{2}_{x}}.
\end{equation}
Independently, Kruzhkov, Faminski\u{\i} \cite{KF} obtained a quite similar result to \eqref{smoothingkdv}.

%Inequality \eqref{kdv} is  by itself  interesting and as   is pointed out by    Constantine and Saut \cite{CONSAUT}:
%\begin{tcolorbox}[arc=0pt, colback=gray!10, boxrule=0pt]
%	\emph{ \textquotedblleft The solution of the initial value problem, is local, with one derivative smoother than the initial datum. Kato's proof uses, in a crucial way, the algebraic properties of the symbol for the Korteweg-de Vries equation and the fact that the underlying spatial dimension is one. Judging from the way  several integrations by parts and cancellations conspires to reveal a smoothing effect\textquotedblright}
%\end{tcolorbox}
As was shown later, and almost simultaneously by Constantine, Saut \cite{CONSAUT}, Sj\"{o}lin \cite{SJOLin} and Vega \cite{VEGA}, the local smoothing is an intrinsic property  of linear dispersive equations (see  \cite{Lipolibro}  Chapter 4 and the references therein)

A question that arise naturally is  to determine whether the solutions $u$ of IVP \eqref{zk4} have a  local smoothing effect similar to that one satisfied by the solutions of the KdV equation \eqref{smoothingkdv}. Certainly, this is not an easy question to answer in the full range $\alpha\in (0,2).$   In the case $\alpha=2,$ the operator that provides the dispersion in the linear part of the equation  \eqref{zk4} is a local operator  and   it is   possible to obtain by performing energy estimates that solutions of the ZK equation  in a suitable Sobolev space satisfy an inequality similar in spirit  to \eqref{smoothingkdv} with $\nabla$ instead of $\partial_{x},$ but   on another class of subsets of the  euclidean space.

However, when we turn our attention to the case $\alpha\in (0,2),$ the situation  is not  so easy to address, since the operator $\left(-\Delta\right)^{\alpha/2}$  is fully non-local. Nevertheless, as is indicated in the work of Constantine, Saut \cite{CONSAUT} we expect a local gain of $\alpha/2$ of a derivative either the operator be  local or not.

One of the  main goals of this work is to prove  \emph{à  la Kato}  that solutions of the IVP \eqref{zk4}  gain locally $\alpha/2$ of a spatial  derivative. Certainly this problem has been addressed previously in the one-dimensional case \emph{e.g.}  Ponce \cite{GP} and Ginibre, Velo \cite{Ginibrev1, Gnibrev2} for solutions of the Benjamin-Ono equation.  We shall also mention  the work of  Kenig, Ponce and Vega \cite{KPVOS} where is proved   that  solutions of the  IVP
\begin{equation}\label{zk4.111}
	\left\{
	\begin{array}{ll}
	\mathrm{i}	\partial_{t}u+P(D)u=0, &  \\
		u(x,0)=u_{0}(x), \quad x\in \mathbb{R}^{n},& t\in\mathbb{R}, \\
	\end{array} 
	\right.
\end{equation}
where $$P(D)f(x):=\int_{\mathbb{R}^{n}}e^{ix\cdot \xi}P(\xi)\mathcal{F}(f)(\xi)\, d\xi,$$ with $P$ satisfying  certain conditions, enjoy of local   smoothing effect. Also, in \cite{KPVOS} is  showed that  solutions of the IVP \eqref{zk4.111}   satisfy   a global smoothing effect (see sections 3 and 4 in\cite{KPVOS}). 
Their  proofs are  mainly based on estimates of oscillatory integrals, as well as, the use of the Fourier restriction method.

The results in \cite{Ginibrev1} and its extension in \cite{Gnibrev2} are quite versatile and  allow us to obtain the smoothing effect in the desired range  $\alpha\in(0,2).$ The main idea  behind these arguments  relies on obtaining a  pointwise decomposition for  the commutator 
\begin{equation}\label{commu1}
	[(-\partial_{x}^{2})^{\frac{\alpha}{2}}\partial_{x_{1}}; \varphi ],
\end{equation}  
where $\varphi$ is a real valued smooth function with certain decay at infinity.

Heuristically, the idea is to decouple  \eqref{commu1}   in  lower-order pieces plus some non-localized error term easy to handle.

However,  in higher dimensions  there is not known   pointwise  decomposition  formula for  the commutator 
\begin{equation}\label{formula3}
	\left[\left(-\Delta\right)^{\frac{\alpha}{2}}\partial_{x_{1}}; \varphi \right],
\end{equation}  
similar to that one  in \cite{Ginibrev1}, so that new ideas are required   to  obtain the desired smoothing.
  
  After obtaining a decomposition of commutator \eqref{formula3}  in pieces of lower order, we replace  
the operator $(-\Delta)^{\frac{\alpha}{2}}$ by $(I-\Delta)^{\frac{\alpha}{2}},$ to our proposes the main difference between  both operators  relies on  the fact that $(I-\Delta)^{\frac{\alpha}{2}}$ is a pseudo-differential operator, instead of $(-\Delta)^{\frac{\alpha}{2}}$ that is  not for $\alpha$ in the indicated range.

Few years ago, Bourgain and  Li \cite{BL} established the  pointwise formula 
\begin{equation}\label{formula2}	
	\left(I-\Delta \right)^{\frac{\alpha}{2}}=\left(-\Delta \right)^{\frac{\alpha}{2}}+ \mathcal{K}_{\alpha},\quad 0<\alpha\leq 2,
\end{equation}
where $\mathcal{K}_{\alpha}$ an  integral operator     that  maps $L^{p}(\mathbb{R}^{n})$ into $L^{p}(\mathbb{R}^{n})$ for all $p\in[1,\infty].$

The expression above allow us to obtain after replacing in \eqref{formula3}
\begin{equation}\label{formula4}
	\left[\left(-\Delta\right)^{\frac{\alpha}{2}}\partial_{x_{1}}; \varphi \right]=	\left[\left(I-\Delta\right)^{\frac{\alpha}{2}}\partial_{x_{1}}; \varphi \right]+	\left[\mathcal{K}_{\alpha}\partial_{x_{1}}; \varphi \right],
\end{equation}
At this point, the situation is easier to handle since for the first expression on the r.h.s we use Pseudo-differential calculus to decouple pointwise the commutator expression in terms of lower order in the spirit of Ginibre and Velo decomposition \cite{Ginibrev1}. Even when the situation is more manageable in comparison to \eqref{formula3},  the decomposition produces a considerable amount of error terms that are not easy to handle due to the interactions between terms of higher regularity vs lower regularity.

If somehow we had to summarize in a few words the arguments used to estimate the first term in \eqref{formula4} we could  refer to a maxim 
Roman emperor's    Julius Caesar 
\textquotedblleft	\emph{divide et vinces}\textquotedblright.
%whose translation from Latin to English is
%\textquotedblleft \textit{divide and conquer} \textquotedblright.

The situation is quite different for the term involving $\mathcal{K}_{\alpha},$  this one is by far the hardest to deal with.  It contains a sum that requires to know information about the behavior of Bessel potentials at origin  and  infinity (see Appendix A)  and several cases depending on the  dimension  have to be examined. In this process the Gamma function and its properties are fundamentals to guarantee  control in certain  Sobolev norm.

After decomposing the first term on the r.h.s of \eqref{formula4}  we clearly obtain that solutions of the IVP \eqref{zk4}  gain locally in space the expected $\frac{\alpha}{2}$ of a derivative in the full range $\alpha\in (0,2)$ without   restrictions, this constitutes our first main result whose statement we show below.
\begin{lem}\label{main2}
	Let $u\in   C ((0, T) : H^{s}(R^{n})),\,  s>\frac{n}{2}+1,$
	be a solution of (1.1) with $0< \alpha<2$ and $n\geq 2.$  
	
	If $\varphi:\mathbb{R}^{n}\longrightarrow \mathbb{R}$ is  a $C^{\infty}(\mathbb{R}^{n})$ function satisfying:
	\begin{itemize}
		\item[(i)] There exist a non-decreasing smooth function  $\phi:\mathbb{R}\longrightarrow\mathbb{R}, $ and  $\nu=(\nu_{1},\nu_{2},\dots,\nu_{n})\in\mathbb{R}^{n}$ such that  
		\begin{equation*}
			\varphi(x)=\phi\left(\nu\cdot x+\delta\right)\quad x\in\mathbb{R}^{n},
		\end{equation*}
		for some $\delta\in\mathbb{R}.$  The vector  $\nu$  is taken in such a way that it satisfies  one and only one  of the following  conditions:
		\begin{itemize}
			\item[\sc Case 1:] $\nu_{1}>0$ and $\nu_{2},\nu_{3},\dots,\nu_{n}=0.$
			\item[\sc Case 2:]  $\nu_{1}>0,$     $(\nu_{2},\nu_{3},\dots,\nu_{n})\neq 0,$  verify the inequality
			\begin{equation}\label{condi1}
				0<	\sqrt{\nu_{2}^{2}+\nu_{3}^{2}+\dots+\nu_{n}^{2}}<\min\left\{ \frac{2\nu_{1}}{C\sqrt{\alpha(n-1)}},\frac{\nu_{1}(1+\alpha)}{\alpha\epsilon\sqrt{n-1}}\right\},
			\end{equation}
			with  
			\begin{equation}\label{condi2}
				0<\epsilon<\frac{\nu_{1}}{|\overline{\nu}|\sqrt{n-1}}-\frac{\alpha\sqrt{n-1}|\overline{\nu}|}{4\nu_{1}}C^{2},
			\end{equation} 
			where   $$C:=\inf_{f\in L^{2}(\mathbb{R}^{n}), f\neq 0}\frac{\|J^{-1}\partial_{x_{j}}f\|_{L^{2}}}{\|f\|_{L^{2}}},\quad j=2,3,\dots,n,$$
			and 
			\begin{equation*}
				|\overline{\nu}|:=\sqrt{\nu_{2}^{2}+\nu_{3}^{2}+\dots+\nu_{n}^{2}}.
			\end{equation*}
		\end{itemize}
		\item[(ii)] The function $\phi$ satisfies: 
		\begin{equation*}
			\phi'\equiv 1\quad \mbox{on}\quad [0,1].
		\end{equation*}
		\item[(ii)] There exist a  positive constant $c$ such that
		\begin{equation*}
			\sup_{0\leq j\leq 4}\sup_{x\in\mathbb{R}}\left|(\partial_{x}^{j}\phi)(x)\right|\leq c.
		\end{equation*}
		\item[(iii)]For all $x\in\mathbb{R},$ 
		$$\phi'(x)\geq0.$$
		\item[(iv)] The function  $\phi^{1/2}$ satisfies
		\begin{equation*}
			\sup_{x\in\mathbb{R}}	\left|\partial_{x}^{j}\left(\sqrt{\phi(x)}\right)\right|\leq c\quad \mbox{for}\quad j=1,2.
		\end{equation*}
	\end{itemize}  
	Then
	\begin{equation*}
		\begin{split}
			&\int_{0}^{T}\int_{\mathbb{R}^{n}} \left(J^{s+\frac{\alpha}{2}}u(x,t)\right)^{2}\partial_{x_{1}}\varphi(x)\,dx\,dt+\int_{0}^{T}\int_{\mathbb{R}^{n}} \left(\partial_{x_{1}}J^{s+\frac{\alpha-2}{2}}u(x,t)\right)^{2}\partial_{x_{1}}\varphi(x)\,dx\,dt\\
			&\lesssim_{n,\alpha}\left(1+ T+ \left\|\nabla
			u\right\|_{L^{1}_{T}L^{\infty}_{x}}+ T\left\|u\right\|_{L^{\infty}_{T}H^{r}_{x}}\right)^{1/2}\|u\|_{L^{\infty}_{T}H^{s_{n}^{+}}_{x}},
		\end{split}
	\end{equation*}
	whenever $r>\frac{n}{2}.$
\end{lem}
The reader should keep in mind  that the   results Lemma \eqref{main2}  show  a strong dependence on the variable $x_{1},$  this  can be observed in  the condition on $\nu_{1},$ it never can be null unlike the  other coordinates that can be null but not  all of them simultaneously    as it  is pointed out  in the   case 2  in Lemma \ref{main2}.

After translating properly $\varphi,$ it is possible to   describe some  regions where the smoothing effect is valid. 
We   show  that  depending  on the dimension and the possible sign of the coordinates of vector $\nu$ in \eqref{condi1} the geometry of the regions might change. See figures \ref{fig:2}-\ref{fig:3}. 

A question that comes out  naturally from Lemma \ref{main2}   is to determine  whether a homogeneous version of smoothing also holds. Indeed, it holds and its proof    is  a consequence of combining Lemma \ref{main2} and  formula  \eqref{formula2}.
		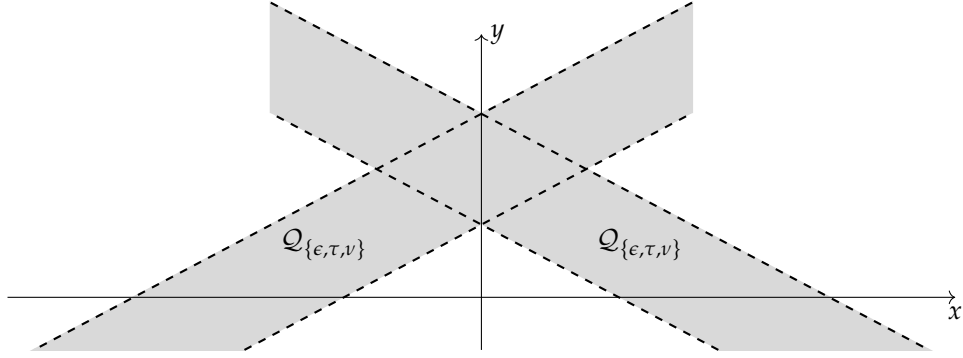
\begin{figure}[h!]
		\begin{center}
			\begin{tikzpicture}[scale=0.7]
				\filldraw[thick, color=gray!30](-4,3.5) -- (-4,5.6)--(8.5,-1) -- (4.5,-1); 
				\filldraw[thick, color=gray!30](4,3.5) -- (4,5.6)--(-8.5,-1) -- (-4.5,-1); 
				\draw[thick,dashed] (-4,5.6) -- (8.5,-1);
				\draw[thick,dashed] (4.5,-1)--(-4,3.5) ; 	
				\draw[thick,dashed] (4,5.6) -- (-8.5,-1);
				\draw[thick,dashed] (-4.5,-1)--(4,3.5) ; 	
				\draw[->] (-9,0) -- (9,0) node[below] {$x$};
				\draw[->] (0,-1) -- (0,5) node[right] {$y$};
				\node at (3,1){$\footnotesize{\mathcal{Q}_{\{ \epsilon,\tau,\nu \}}} $};
				\node at (-3,1){$\footnotesize{\mathcal{Q}_{\{ \epsilon,\tau,\nu  \}}} $};
			\end{tikzpicture}
			\qquad
		\end{center}
		\caption{\label{fig:2} \emph{Regions where the smoothing effect occurs. The region on the left  corresponds to $\nu_{1}>0, \nu_{2}<0.$ \\ The region on the right corresponds to the case  $\nu_{1}>0,\nu_{2}>0.$}}
	\end{figure}
	\begin{figure}[h!]
	\begin{center}
		\begin{tikzpicture}[scale=0.8]
			\filldraw[thick, color=gray!30] (2,-1) -- (2,5)  -- (5.8,5) -- (5.8,-1);
			\draw[thick,dashed] (2,-1) -- (2,5);
			\draw[thick,dashed] (5.8,-1) -- (5.8,5);
			\node at (4,1){$\footnotesize{\mathcal{Q}_{\{ \tau_{1},\tau_{2},\nu \}}} $};
			\node at (1.7,1.5){$\footnotesize{\frac{\tau_{1}}{\nu_{1}}} $};
			\node at (6.2,1.5){$\footnotesize{\frac{\tau_{2}}{\nu_{1}}} $};
			\draw[->] (-2,2) -- (9,2) node[below] {$x$};
			\draw[->] (1,-1) -- (1,5) node[right] {$y$};
		\end{tikzpicture}
		\qquad
	\end{center}
	\caption{Particular region where the smoothing effect  is valid:$ \mathcal{Q}_{\{\nu, \tau_{1},\tau_{2}\}}$ where $0<\tau_{1}<\tau_{2}$ and $\nu_{1}>0,\,\nu_{2}=0.$}\label{fig:1}
\end{figure}
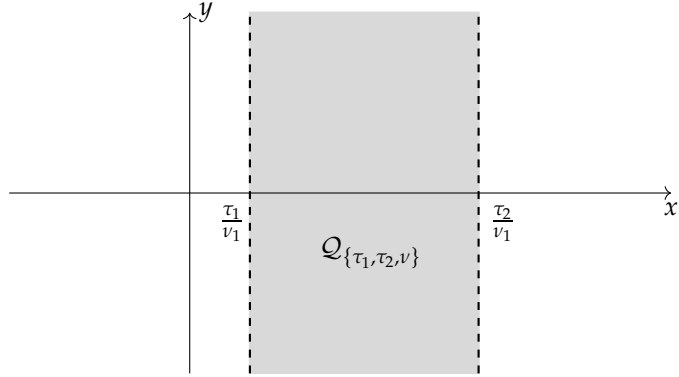
	\begin{center}
		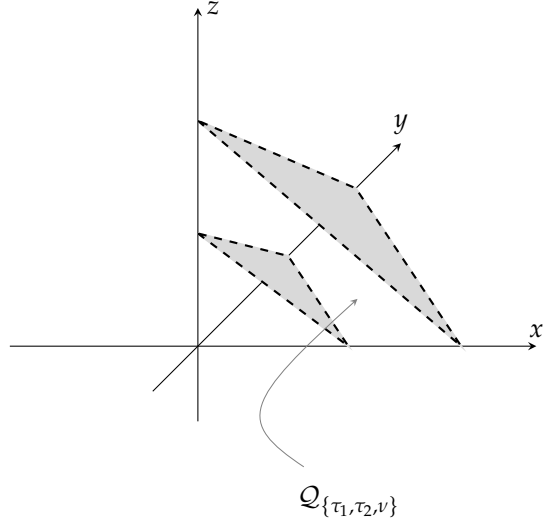
\begin{figure}[h]
			\begin{tikzpicture}[x=0.5cm,y=0.5cm,z=0.3cm,>=stealth]
				\draw[->] (xyz cs:x=-5) -- (xyz cs:x=9) node[above] {$x$};
				\draw[->] (xyz cs:y=-2) -- (xyz cs:y=9) node[right] {$z$};
				\draw[->] (xyz cs:z=-2) -- (xyz cs:z=9) node[above] {$y$};
				\filldraw[thick, color=gray!30](0,3,0) -- (4,0,0) -- (0,0,4)--(0,3,0); 
				\draw[thick,dashed] (0,3,0) -- (4,0,0)--(0,0,4)--(0,3,0);
				\filldraw[thick, color=gray!30](0,6,0) -- (7,0,0) -- (0,0,7)--(0,6,0); 
				\draw[thick,dashed] (0,6,0) -- (7,0,0)--(0,0,7)--(0,6,0);
				\node[align=center] at (4,-4) (ori) {\\$\mathcal{Q}_{\{\tau_{1},\tau_{2},\nu \}} $};
				\draw[->,help lines,shorten >=3pt] (ori) .. controls (1,-2) and (1.2,-1.5) .. (5,2,-1);
			\end{tikzpicture}
			\caption{\emph{Description of the region  where the smoothing takes place in dimension  3 with $0<\tau_{1}<\tau_{2}$ and   $\nu_{1},\nu_{2},\nu_{3}>0.$}} \label{fig:3}
		\end{figure}
	\end{center}

	\begin{cor}
		Under the hypothesis of Lemma \ref{main2}
		the solution of the IVP \eqref{zk4} satisfies 
		\begin{equation}\label{energyasympt}
			\begin{split}
				&\int_{0}^{T}\int_{\mathbb{R}^{n}} \left(\left(-\Delta\right)^{s+\frac{\alpha}{2}}u(x,t)\right)^{2}\partial_{x_{1}}\varphi\,dx\,dt\\
				&\lesssim_{n,\alpha
				}\left(1+ T+ \left\|\nabla u\right\|_{L^{1}_{T}L^{\infty}_{x}}+ T\left\|u\right\|_{L^{\infty}_{T}H^{s_{n}^{+}}_{x}}\right)^{1/2}\|u\|_{L^{\infty}_{T}H^{s_{n}^{+}}_{x}}.
			\end{split}
		\end{equation}
	\end{cor}
In the study of the asymptotic behavior of the solutions  of the Zakharov-Kuznetsov equation in  the energy space is    required to know the  behavior of the  function  and its derivatives on certain subsets of the  plane \emph{e.g.}  channels, squares. In this sense,  estimates such as  \eqref{energyasympt} are quite useful in the description of such behavior   (see Mendez, Mu\~{n}oz, Poblete  and Pozo  \cite{MMPP} for more details).%As a  consequence of the Lemma above we also are able to  obtain information on cubes  contained in $\mathbb{R}^{n}.$
\begin{cor}\label{cor1}
	Let $u\in   C ([0, T] : H^{s}(R^{n})),\,s>\frac{n}{2}+1$ with $n\geq 2,$   and $u$	be a solution of \eqref{zk4}.
	
	Let  $\vec{\kappa}=\left(\kappa_{1},\kappa_{2},\dots,\kappa_{n}\right)\in\mathbb{Z}^{n}.$  For  $\vec{\kappa}\neq 0$ we define 
	\begin{equation*}
		\mathfrak{P}_{\vec{\kappa}}:=\left\{x\in\mathbb{R}^{n}\,|\,\kappa_{j}< x_{j}\leq \kappa_{j}+1, j=1,2,\dots, n \right\}.
	\end{equation*}	
	Then
	\begin{equation}\label{eq1.1}
		\begin{split}
			&\int_{0}^{T}\int_{\mathfrak{P}_{\vec{\kappa}}} \left(J^{s+\frac{\alpha}{2}}u(x,t)\right)^{2}\,dx\,dt+\int_{0}^{T}\int_{\mathfrak{P}_{\vec{\kappa}}} \left(\partial_{x_{1}}J^{s+\frac{\alpha-2}{2}}u(x,t)\right)^{2}\,dx\,dt\\
			&\lesssim_{n,\alpha}\left(1+ T+ \left\|\nabla
			u\right\|_{L^{1}_{T}L^{\infty}_{x}}+ T\left\|u\right\|_{L^{\infty}_{T}H^{r}_{x}}\right)^{1/2}\|u\|_{L^{\infty}_{T}H^{s_{n}^{+}}_{x}},
		\end{split}
	\end{equation}
	whenever $r>\frac{n}{2}.$
\end{cor}
%\begin{proof}
%	For $\vec{\kappa}=\left(\kappa_{1},\kappa_{2},\dots,\kappa_{n}\right)\in\mathbb{Z}^{n}$ we define 
%	\begin{equation*}
%		\varphi_{\vec{\kappa}}(x_{1},x_{2},\dots,x_{n}):=\phi(x_{1}-\kappa_{1}),
%	\end{equation*}
%	with $\phi$ satisfying the conditions   indicated  on Lemma \ref{main2}.
%	
%	Notice that in this case we are taking $\nu_{1}=1$ and $\nu_{2}=\nu_{3}=\dots=\nu_{n}=0,$ thus 
%	\begin{equation*}
%		\partial_{x_{1}}\varphi_{\vec{\kappa}}(x)=\phi'(x_{1}-\kappa_{1})=1, \quad \mbox{whenever}\quad \kappa_{1}\leq x_{1} \leq \kappa_{1}+1.
%	\end{equation*}
%	Since
%	\begin{equation}\label{set}
%		\mathfrak{P}_{\vec{\kappa}}\subset \left\{x\in\mathbb{R}^{n}\,|\, x_{1}\geq\kappa_{1}\right\},
%	\end{equation}
%	then
%	\begin{equation*}
%		\begin{split}
%			\int_{0}^{T}\int_{\mathfrak{P}_{\vec{\kappa}}}\left(J^{s+\frac{\alpha}{2}}u(x,t)\right)^{2}\,dx\,dt\leq\int_{0}^{T}\int_{\mathbb{R}^{n}}\left(J^{s+\frac{\alpha}{2}}u(x,t)\right)^{2}\partial_{x_{1}}\varphi_{\vec{\kappa}}(x)\,dx\,dt<\infty.
%		\end{split}
%	\end{equation*}
%	The second term in the r.h.s   in \eqref{eq1.1} can be handled by using a  similar argument.
%\end{proof}
Also in the homogeneous case  is possible to describe  the smoothing effect in the same region as above.
\begin{cor}
	Let $u\in   C ([0, T] : H^{s}(R^{n})),\,s>\frac{n}{2}+1,$ with $n\geq 2,$   and $u$	be a solution of \eqref{zk4}.
	
	For $\vec{\kappa}$ and $\mathfrak{P}_{\vec{\kappa}}$ as in Corollary \ref{cor1}, the solution $u$ associated  to \eqref{zk4} satisfies:  
	\begin{equation*}
		\begin{split}
			&\int_{0}^{T}\int_{\mathfrak{P}_{\vec{\kappa}}} \left(\left(-\Delta\right)^{s+\frac{\alpha}{2}}u\right)^{2}(x,t)\,dx\,dt\\
			&\lesssim_{n,\alpha}\left(1+ T+ \left\|\nabla
			u\right\|_{L^{1}_{T}L^{\infty}_{x}}+ T\left\|u\right\|_{L^{\infty}_{T}H^{r}_{x}}\right)^{1/2}\|u\|_{L^{\infty}_{T}H^{s_{n}^{+}}_{x}},
		\end{split}
	\end{equation*}
	whenever $r>\frac{n}{2}.$
\end{cor}
Kato's smoothing effect has found diverse applications  in  the  field of dispersive equations.  Our intention in this part of the work is to  present  to the reader  an application of Kato's smoothing effect   for the solutions of the IVP \eqref{zk4}. 
 
 The question addressed is the following: \emph{If the initial data $u_{0}$ in the IVP \eqref{zk4}  is provided with extra regularity in the half space $\mathcal{H}_{\{\epsilon,\nu\}}$ where
 \begin{equation*}
 	\mathcal{H}_{\{\nu,\nu \}}:=\left\{x\in\mathbb{R}^{n}\, |\, \nu\cdot x>\epsilon\right\},
 \end{equation*}
  $\epsilon>0$ and $\nu$ is a non-null vector in $\mathbb{R}^{n}.$ Does the solution $u$  preserve the same regularity for almost all time $t>0$?}
 
 Surprisingly, that extra regularity is propagated by the flow solution  with infinity speed and this property  has been shown  to be true in several nonlinear dispersive models, in fact, this property is known  nowadays as \emph{principle of propagation of regularity}.
 
 The  description of such phenomena depends  strongly on Kato's smoothing effect. Indeed,  the  method to establish  this particular  property  is mainly based on weighted energy estimates  where it   is also possible to show that the  localized regularity entails the gain of extra  derivatives on the channel $\mathcal{Q}_{\{\epsilon,\tau,\nu\}}$ traveling in some specific direction, where 
 \begin{equation*}
 	\mathcal{Q}_{\{\epsilon,\tau,\nu\}}:=\left\{x\in\mathbb{R}^{n}\,|\, \epsilon<\nu\cdot x<\tau\right\},
 \end{equation*}
being $\tau>\epsilon.$  The figure \ref{fig:5}  below  intends to describe this particular phenomena  in dimension $2.$

 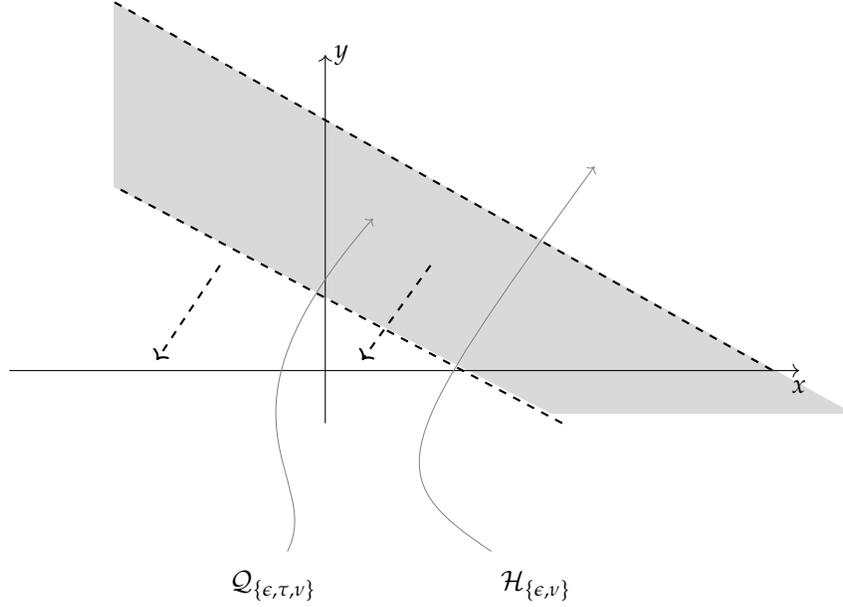
\begin{figure}[h!]
 	\begin{center}
 		\begin{tikzpicture}[scale=0.7]
 			\filldraw[thick, color=gray!30](-4,7) -- (-4,3.5)--(4.3,-0.8) -- (10,-0.8); 
 			%	\filldraw[thick, color=gray!30](4,3.5) -- (4,5.6)--(-8.5,-1) -- (-4.5,-1); 
 			\draw[thick,dashed] (-4,7) -- (8.5,0);
 			\draw[thick,dashed] (4.5,-1)--(-4,3.5) ; 	
 			%			\draw[thick,dashed] (4,5.6) -- (-8.5,-1);
 			%			\draw[thick,dashed] (-4.5,-1)--(4,3.5) ; 	
 			\draw[->] (-6,0) -- (9,0) node[below] {$x$};
 			\draw[->] (0,-1) -- (0,6) node[right] {$y$};
 			%	\node at (3,1){$\footnotesize{\mathcal{Q}_{\{\nu, \tau_{1},\tau_{2}\}}} $};
 			%	\node at (-3,1){$\footnotesize{\mathcal{Q}_{\{\nu, \tau_{1},\tau_{2}\}}} $};
 			\node[align=center] at (4,-4) (ori) {\\$\mathcal{H}_{\{\epsilon,\nu \}} $};
 			\draw[->,help lines,shorten >=3pt] (ori) .. controls (1,-2) and (1.2,-1.5) .. (5.2,4);
 			\node[align=center] at (-1,-4) (ori) {\\$\mathcal{Q}_{\{\epsilon,\tau,\nu \}} $};
 			\draw[->,help lines,shorten >=3pt] (ori) .. controls (0,-2) and (-2.5,-1) .. (1,3);
 			\draw[thick, dashed,->] (-2,2) -- (-3.2,0.2);
 			\draw[thick, dashed,->] (2,2) -- (0.7,0.2);
 		\end{tikzpicture}
 		\qquad
 	\end{center}
 	\caption{\label{fig:5} \emph{Propagation of regularity sense in the 2 dimensional case  with  $\nu_{1}>0,\nu_{2}>0$. The dashed arrows denotes the propagation sense.}}
 \end{figure}
\begin{center}
\end{center}
This question was originally address by Isaza, Linares and Ponce \cite{ILP1,ILP2,ILP3} for solutions of the KdV equation and  later   it was  studied   for solutions of the Benjamin-Ono equation and the Kadomset-Petviashvili equation  \emph{resp.} In the case of  one dimensional models  where the dispersion is weak this  was established in \cite{AM1,AM2} for solutions of the dispersive  generalized Benjamin-Ono and the fractional Korteweg-de Vries equation \emph{resp.} In the  quasi-linear type equations case Linares, Smith and Ponce  show under certain conditions that this principle also holds. In the case  of higher dispersion,  Segata and Smith show that for the fifth order KdV equation. The  extension to the  case where the regularity of the initial data  was fully addressed by Kenig, Linares, Ponce and Vega \cite{KLPV} for solutions of the KdV.

The  results in \cite{KLPV}  were  later extended  in  \cite{AMZK} to the  $n-$ dimensional case and subsequently these techniques were applied by Freire, Mendez  and Riaño \cite{FMR} for  solutions of the dispersive \emph{generalized Benjamin-Ono-Zakharov-Kuznetsov} equation, that is,
\begin{equation}\label{propaail}
	\partial_{t}u-\partial_{x_{1}}(-\partial_{x_{1}}^{2})^{\frac{\alpha+1}{2}}u+\partial_{x_{2}}^{2}\partial_{x_{1}}u+u\partial_{x_{1}}u=0\quad 0\leq \alpha\leq 1.
\end{equation}
The case $\alpha=0$ in \eqref{propaail}  was first addressed  by Nascimento in  \cite{ailton} in the spirit of \cite{ILP2}.

For the most recent compendium of propagation of regularity results we refer to  Linares and Ponce \cite{LPPROPA} and the references therein.

Our second main result is devoted  to  show that  solution of the FZK also satisfies the propagation of regularity principle  and it  is  summarized in the following theorem:
%The phenomena of propagation of regularity is summarized in the following theorem. The notation    used in the description of the results in the  following theorem  is   presented in section \ref{seccion2}.
 \begin{thm}\label{zk9}
	Let $u_{0}\in H^{s}(\mathbb{R}^{n})$ with $s>s_{n}.$  Let $\nu=(\nu_{1},\nu_{2},\dots,\nu_{n})\in \mathbb{R}^{n},\, n\geq 2$ with $\nu$ satisfying \eqref{condi1}-\eqref{condi2}.
	%\begin{itemize}
	%	\item[\sc Case 1:] $\nu_{1}>0$ and $\nu_{2},\nu_{3},\dots,\nu_{n}=0.$
	%	\item[\sc Case 2:]  $\nu_{1}>0,$     $(\nu_{2},\nu_{3},\dots,\nu_{n})\neq 0,$  verify the inequality
	%	\begin{equation*}
		%		0<	\sqrt{\nu_{2}^{2}+\nu_{3}^{2}+\dots+\nu_{n}^{2}}<\min\left\{ \frac{2\nu_{1}}{C\sqrt{\alpha(n-1)}},\frac{\nu_{1}(1+\alpha)}{\alpha\epsilon\sqrt{n-1}}\right\},
		%	\end{equation*}
	%	with  
	%	\begin{equation*}
		%		0<\epsilon<\frac{\nu_{1}}{|\overline{\nu}|\sqrt{n-1}}-\frac{\alpha\sqrt{n-1}|\overline{\nu}|}{4\nu_{1}}C^{2},
		%	\end{equation*} 
	%	and  $$C:=\sup_{f\in L^{2}(\mathbb{R}^{n})}\frac{\|J^{-1}\partial_{x_{j}}f\|_{L^{2}}}{\|f\|_{L^{2}}},\quad j=2,3,\dots,n.$$
	%\end{itemize}
	If the initial data $u_{0}$ satisfies 
	\begin{equation}\label{e1.1}
		\int_{\mathcal{H}_{\{\beta, \nu\}}} \left(J^{\widetilde{s}}u_{0}(x)\right)^{2}\,dx<\infty,
	\end{equation}	
	then the corresponding solution  $u=u(x,t)$ of the IVP  \eqref{zk4} with $1\leq \alpha <2$: satisfies  that for any $\omega > 0,\, \epsilon>0$ and $\tau\geq5 \epsilon,$ 
	\begin{equation}\label{g1}
		\sup_{0\leq t\leq T}\int_{\mathcal{H}_{\{\beta+\epsilon-\omega t,\nu\}}}\left(J^{r}u(x,t)\right)^{2}dx\leq c^{*},
	\end{equation} 	
	for any $r\in (0,\widetilde{s}]$ with $c^{*}=c^{*}\left(\epsilon; T; \omega ; \|u_{0}\|_{H^{s_{n}+}}; \|J^{s}u_{0}\|_{L^{2}\left(\mathcal{H}_{\{\beta,\nu \}}\right)}\right).$
	
	In addition, for any $\omega> 0,\, \epsilon>0$ and $\tau\geq 5\epsilon$
	\begin{equation}\label{g2.1}
		\int_{0}^{T}\int_{\mathcal{Q}_{\{\epsilon-\omega t+\beta ,\beta-\omega t+\tau,\nu }\}}\left(J^{\widetilde{s}+1}u\right)^{2}(x,t)\,dx\,dt\leq c^{*},
	\end{equation}
	with $c^{*}=c^{*}\left(\epsilon;\tau; T; \omega ; \|u_{0}\|_{H^{s_{n}{+}}}; \|J^{\widetilde{s}}u_{0}\|_{L^{2}(\mathcal{H}_{\{\beta,\nu \}})}\right).$
	
	If in addition to (\ref{e1.1}) there exists $\beta>0,$  such that 
	\begin{equation}\label{clave1}
		J^{\widetilde{s}+\frac{2-\alpha}{2}}u_{0}\in L^{2}\left(\mathcal{H}_{\{\beta,\nu\}}\right),
	\end{equation}
	then for any $\omega > 0,\,\epsilon>0$ and $\tau\geq 5\epsilon,$
	\begin{equation}\label{l1}
		\begin{split}
			&\sup_{0\leq t\leq  T}\int_{\mathcal{H}_{\{\beta+\epsilon-\omega t,\nu\}}}\left(J^{\widetilde{s}+\frac{1-\alpha}{2}}u\right)^{2}(x,t)\,dx\\
			&\qquad 
			+\int_{0}^{T}\int_{\mathcal{Q}_{\{\epsilon-\omega t+\beta ,\beta-\omega t+\tau,\nu }\}}\left(J^{\widetilde{s}+1}u\right)^{2}(x,t)\,dx\,dt\leq c,
		\end{split}
	\end{equation}
	with $c=c\left(T;\epsilon;\omega ;\alpha;\|u_{0}\|_{H^{\widetilde{s}}};\left\|J^{\widetilde{s}+\frac{1-\alpha}{2}}u_{0}\right\|_{L^{2}((x_{0},\infty))}\right)>0.$
\end{thm}
The proof of  of Theorem \ref{zk9} is based on weighted energy estimates combined with an inductive argument,  that due to the weak  effects of dispersion it has to be carried out in   two steps.
\begin{rem}
The result in Theorem \ref{zk9} is also true in  the case   where the dispersion is even weaker \emph{e.g.} $0<\alpha<1,$ the proof in this case  follows  by combining the ideas of  the proof of Theorem \ref{zk9} and the bi-inductive argument   applied in  \cite{AM2,AMTHESIS} for solutions of the fKdV.
\end{rem}
As a corollary   we obtain that in the case  that the extra regularity of the initial data  is  provided on a  integer scale the result also holds true.
\begin{cor}
	Let $u_{0}\in H^{s}(\mathbb{R}^{n})$ with $s>s_{n}.$  If for some $\nu=(\nu_{1},\nu_{2},\dots,\nu_{n})\in \mathbb{R}^{n},\, n\geq 2$ with $\nu$ satisfying \eqref{condi1}-\eqref{condi2}.
	
	%\begin{itemize}
	%	\item[\sc Case 1:] $\nu_{1}>0$ and $\nu_{2},\nu_{3},\dots,\nu_{n}=0.$
	%	\item[\sc Case 2:]  $\nu_{1}>0,$     $(\nu_{2},\nu_{3},\dots,\nu_{n})\neq 0,$  verify the inequality
	%	\begin{equation*}
		%		0<	\sqrt{\nu_{2}^{2}+\nu_{3}^{2}+\dots+\nu_{n}^{2}}<\min\left\{ \frac{2\nu_{1}}{C\sqrt{\alpha(n-1)}},\frac{\nu_{1}(1+\alpha)}{\alpha\epsilon\sqrt{n-1}}\right\},
		%	\end{equation*}
	%	with  
	%	\begin{equation*}
		%		0<\epsilon<\frac{\nu_{1}}{|\overline{\nu}|\sqrt{n-1}}-\frac{\alpha\sqrt{n-1}|\overline{\nu}|}{4\nu_{1}}C^{2},
		%	\end{equation*} 
	%	and  $$C:=\sup_{f\in L^{2}(\mathbb{R}^{n})}\frac{\|J^{-1}\partial_{x_{j}}f\|_{L^{2}}}{\|f\|_{L^{2}}},\quad j=2,3,\dots,n.$$
	%\end{itemize}
	If there exist $m\in\mathbb{N}, m> 1+ \left\lceil\frac{n}{2}\right\rceil$  such that, initial data $u_{0}$ satisfy
	\begin{equation}
		\partial_{x}^{\alpha}u_{0}\in L^{2}\left(\mathcal{H}_{\{\beta,\nu\}}\right)<\infty,\quad\mbox{such that}\quad |\alpha|=m,
	\end{equation}	
	then the corresponding solution  $u=u(x,t)$ of the IVP  \eqref{zk4} satisfies:  for any $\nu > 0,\, \epsilon>0$ and $\tau\geq5 \epsilon,$ 
	\begin{equation}
		\sup_{0\leq t\leq T}\int_{\mathcal{H}_{\{\beta+\epsilon-\omega t,\nu\}}}\left(J^{r}u(x,t)\right)^{2}dx\leq c^{*},
	\end{equation} 	
	for any $r\in(0,m]$ with $c^{*}=c^{*}\left(\epsilon; T; \nu ; \|u_{0}\|_{H^{s_{n}+}}; \|\partial_{x}^{\alpha}
	u_{0}\|_{L^{2}\left(\mathcal{H}_{\{\nu,\beta\}}\right)}\right).$
	
	In addition, for any $\omega > 0,\, \epsilon>0$ and $\tau\geq 5\epsilon$
	\begin{equation}
		\int_{0}^{T}\int_{\mathcal{Q}_{\{\epsilon-\omega t+\beta ,\beta-\omega t+\tau,\nu }\}}\left(J^{m+\frac{\alpha}{2}} u\right)^{2}(x,t)\,dx\,dt\leq c^{*},
	\end{equation}
	with $c^{*}=c^{*}\left(\epsilon;\tau; T; \nu ; \|u_{0}\|_{H^{s_{n}{+}}}; \|J^{\widetilde{s}}u_{0}\|_{L^{2}(\mathcal{H}_{\{\beta,\nu \}})}\right).$
	
	If in addition to (\ref{e1.1}) there exists $\beta>0,$  such that 
	\begin{equation}
		J^{\frac{2-\alpha}{2}}\partial_{x}^{\alpha}u_{0}\in L^{2}\left(\mathcal{H}_{\{\beta,\nu\}}\right),
	\end{equation}
	then for any $\omega> 0,\,\epsilon>0$ and $\tau\geq 5\epsilon,$
	\begin{equation}\label{l1}
		\begin{split}
			&\sup_{0\leq t\leq  T}\int_{\mathcal{H}_{\{\beta+\epsilon-\omega t,\nu\}}}\left(J^{r}u\right)^{2}(x,t)\,dx\\
			&\qquad 
			+\int_{0}^{T}\int_{\mathcal{Q}_{\{\epsilon-\omega t+\beta ,\beta-\omega t+\tau,\nu }\}}\left(J^{m+1}u\right)^{2}(x,t)\,dx\,dt\leq c,
		\end{split}
	\end{equation}
	with $r\in \left(0, m+\frac{1-\alpha}{2}\right]$  and  $c=c\left(T;\epsilon;\omega ;\alpha;\|u_{0}\|_{H^{s_{n}^{+}}};\left\|J^{m+\frac{1-\alpha}{2}}u_{0}\right\|_{L^{2}\left(\mathcal{H}_{\{\beta,\nu \}}\right)}\right)>0.$
\end{cor}
Also, the solutions associated to the IVP \eqref{zk4} satisfies the following transversal propagation of regularity, which  is summarized in the following theorem.
\begin{cor}\label{zk10}
	Let $u_{0}\in H^{s_{n}}(\mathbb{R}^{n}).$  If for some $\nu=(\nu_{1},\nu_{2},\dots,\nu_{n})\in \mathbb{R}^{n}$  with
	
	and for some $s\in \mathbb{R},s>s_{n}$
	\begin{equation*}
		\int_{\mathcal{H}_{\{\beta,\nu\}}}\left(J^{s}_{x_{1}}u_{0}(x)\right)^{2}\, dx<\infty,
	\end{equation*}	
	then the corresponding solution  $u=u(x,t)$ of the IVP  provided by Theorem \ref{lwp} satisfies  that for any $\omega>0,\, \epsilon>0$ and $\tau \geq5\epsilon$ 
	\begin{equation*}
		\sup_{0< t<T}\int_{\mathcal{H}_{\{\beta+\epsilon-\omega t,\nu \}}}\left(J^{r}_{x_{1}}u\right)^{2}(x,t)dx\leq c^{*},
	\end{equation*} 	
	for any $r\in (0,s]$ with $c^{*}=c^{*}\left(\epsilon; T; \omega; \|u_{0}\|_{H^{s_{n}}_{x_{1}}}; \|J^{r}u_{0}\|_{L^{2}\left(\mathcal{H}_{  \{\nu,\beta\}}\right)}\right).$
	
	In addition, for any $\omega > 0,\, \epsilon>0$ and $\tau\geq5\epsilon$
	\begin{equation*}
		\int_{0}^{T}\int_{\mathcal{Q}_{\{\beta-\omega t+\epsilon,\beta-\omega t+\tau,\nu }\}}\left(J^{s+1}_{x_{1}}u\right)^{2}(x,t)\,dx\,dt\leq c
	\end{equation*}
	with $c=c\left(\epsilon;\tau; T; \omega ; \|u_{0}\|_{H^{s_{n}+}}; \|J^{r}_{x_{1}}u_{0}\|_{L^{2}\left(\mathcal{H}_{\{\beta,\nu\}}\right)}\right)>0.$
\end{cor}
 \subsection{Organization of the paper}
 In section \ref{seccion2} we introduce the notation  to be used in this work.  In the section \ref{prooflema1} we provide in a detailed manner  the  arguments of the proof of   our first main result . Finally in section \ref{seccion5} we  provide an application of  Lemma \ref{main2}  by proving that  solutions of the IVP \eqref{zk4} satisfies the propagation of regularity principle. 
\subsection{Notation}\label{seccion2}
	For two quantities $A$ and $B$, we denote $A\lesssim B$  if $A\leq cB$ for some constant $c>0.$ Similarly, $A\gtrsim B$  if  $A\geq cB$ for some $c>0.$  Also for two positive quantities, $A$  $B$  we say that are \emph{ comparable}  if $A\lesssim B$ and $B\lesssim A,$ when such conditions are   satisfied we  indicate it by  writing $A\approx B.$   The dependence of the constant $c$  on other parameters or constants are usually clear from the context and we will often suppress this dependence whenever it is  possible.

For any pair of quantities  $X$ and $Y,$ we denote  $X\ll Y$ if $X\leq cY$ for some sufficiently small  positive constant $c.$ The smallness of such constant  is usually clear from the context.  The notation $X\gg Y$ is similarly defined.

For $f$ in a suitable class is defined the \emph{Fourier transform} of $f$ as 
\begin{equation*}
	(\mathcal{F}f)(\xi ):=\int_{\mathbb{R}^{n}}e^{2\pi \mathrm{i} x\cdot \xi} f(x)\, dx.
\end{equation*}
For $x\in  \mathbb{R}^{n}$ we denote $$\langle x\rangle :=\left(1+|x|^{2}\right)^{\frac{1}{2}}.$$ For  $s\in \mathbb{R}$ is  defined  
the \emph{ Bessel potential of order $-s$} as 
$J^{s}:=(1-\Delta)^{\frac{s}{2}},$  following this notation the operator $J^{s}$ admits   representation via Fourier transform as 
\begin{equation*}
	\mathcal{F}(J^{s}f)(\xi)=\langle 2\pi\xi\rangle^{s}\mathcal{F}(f)(\xi).
\end{equation*}
We denote by $\mathcal{S}(\mathbb{R}^{n})$ the \emph{Schwartz  functions space}  and the space of \emph{tempered distributions} by $\mathcal{S}'(\mathbb{R}^{n}).$
Additionally, for $s\in\mathbb{R}$ we  consider  the \emph{Sobolev spaces} $H^{s}(\mathbb{R}^{n})$ that are defined  as 
\begin{equation*}
H^{s}(\mathbb{R}^{n}):=J^{-s}L^{2}(\mathbb{R}^{n}).
\end{equation*}
For $p\in[1,\infty]$ we consider the classical Lebesgue spaces $L^{p}(\mathbb{R}^{n}).$  Also,  we shall often use   mixed-norm spaces notation. For example, for  $f:\mathbb{R}^{n}\times[0,T]\longrightarrow \mathbb{R},$ we will denote 
  \begin{equation*}
  	\|f\|_{L^{p}_{T}L^{q}_{x}}:=\left(\int_{0}^{T}\|f(\cdot, t)\|_{L^{q}_{x}}^{p}\, dt\right)^{\frac{1}{p}},
  \end{equation*}
with the obvious modifications in the cases $p=\infty$ or $q=\infty$. Additionally, the  \emph{mixed Sobolev spaces}
\begin{equation*}
	\|f\|_{L^{p}_{T}H^{s}_{x}}:=\left(\int_{0}^{T}\|f(\cdot, t)\|_{H^{s}_{x}}^{p}\, dt\right)^{\frac{1}{p}}.
\end{equation*} 
We recall for operators $A$ and $B$ we define the \emph{commutator }between the operator $A$ and $B$ as  $[A,B] = AB -BA.$ 

Let $\epsilon\in\mathbb{R}.$   For $\nu=\left(\nu_{1},\nu_{2},\dots,\nu_{n}\right)\in\mathbb{R}^{n}$  we define   the \emph{half space}
\begin{equation}\label{half}
	\mathcal{H}_{\{\nu,\nu \}}:=\left\{x\in\mathbb{R}^{n}\, |\, \nu\cdot x>\epsilon\right\},
\end{equation}
where $\cdot$ denotes the  canonical inner product in $\mathbb{R}^{n}.$

Let  $\tau>\epsilon.$  For $\nu=\left(\nu_{1},\nu_{2},\dots,\nu_{n}\right)\in\mathbb{R}^{n}$  we define the \emph{channel} as the set $\mathcal{Q}_{\{\nu,\epsilon,\tau\}}$ satisfying 
\begin{equation}\label{channel}
	\mathcal{Q}_{\{\epsilon,\tau,\nu\}}:=\left\{x\in\mathbb{R}^{n}\,|\, \epsilon<\nu\cdot x<\tau\right\}.
\end{equation}
%\section{Local well-posedness}\label{section3}
%In this section we describe the   space where the 

%\section{Consequences of  Lemma \ref{main2}}\label{seccion4}

%\section{Consequences of  Theorem \ref{zk9}}\label{section6}

\section{Proof of Lemma \ref{main2}}\label{prooflema1}
In this section we describe the details of the proof of the main lemma in this paper.
\begin{proof}
Let $\varphi:\mathbb{R}^{n}\longrightarrow\mathbb{R}$ be a smooth function such that 
\begin{equation}\label{weight}
\sup_{\gamma\in(\mathbb{N}_{0})^{n},|\gamma|\leq 4}\sup_{x\in\mathbb{R}^{n}}	|\partial_{x}^{\gamma}\varphi(x)|\leq c,
\end{equation}
for some positive constant $c.$ 

By standard arguments  we obtain
\begin{equation*}
\begin{split}
&\frac{1}{2}\frac{d}{dt}\int_{\mathbb{R}^{n}} \left(J^{s}u\right)^{2} \varphi\,dx\underbrace{-\int_{\mathbb{R}^{n}}J^{s}\partial_{x_{1}}(-\Delta)^{\alpha/2}uJ^{s}u\varphi(x)\,dx}_{\Theta_{1}(t)}\\
&+\underbrace{\int_{\mathbb{R}^{n}}J^{s}\left(u\partial_{x_{1}}u\right)J^{s}u \varphi\,dx}_{\Theta_{2}(t)}=0.
\end{split}
\end{equation*}
First,  we handle the term providing the dispersive part of the equation, by noticing that  after apply integration by parts we obtain 
\begin{equation*}
\Theta_{1}(t)=\frac{1}{2}\int_{\mathbb{R}^{n}} J^{s}u \left[(-\Delta)^{\alpha/2}\partial_{x_{1}}; \varphi\right]J^{s}u\,dx.
\end{equation*}
The next step is crucial in our argument mainly because we will replace the operator $(-\Delta)^{\frac{\alpha}{2}}$ by 
\begin{equation}\label{decomposition}
(-\Delta)^{\frac{\alpha}{2}}=\underbrace{(I-\Delta)^{\frac{\alpha}{2}}}_{J^{\alpha}}+\mathcal{K}_{\alpha},
\end{equation}
where  $\mathcal{K}_{\alpha}$ is an operator that satisfies the following properties: 
\begin{itemize}
	\item[(i)] There exists a kernel  $k_{\alpha} $  such that 
	\begin{equation}\label{bo1}
		(\mathcal{K}_{\alpha}f)(x):=\int_{\mathbb{R}^{n}} k_{\alpha}(x,x-y)f(y)\,dy,\quad f\in\mathcal{S}(\mathbb{R}^{n}).
	\end{equation} 
	\item[(ii)] For $1\leq p\leq \infty,$
	\begin{equation}\label{bo2}
		\mathcal{K}_{\alpha}: L^{p}(\mathbb{R}^{n})\longrightarrow L^{p}(\mathbb{R}^{n})
	\end{equation}
\end{itemize}
with 
\begin{equation}\label{bo3}
	\left\|\mathcal{K}_{\alpha}f\right\|_{L^{p}}\lesssim \left\|J^{\alpha-2}f\right\|_{L^{p}}.
\end{equation}
For more details on the decomposition \eqref{decomposition}  see Bourgain, Li \cite{BL}.
%\begin{flushleft}
%	{\sc Case  $\alpha\in(1,2):$}
%\end{flushleft}

Thus, after replacing \eqref{decomposition} yield
\begin{equation*}
\begin{split}
\Theta_{1}(t)&=\frac{1}{2}\int_{\mathbb{R}^{n}} J^{s}u \left[J^{\alpha}\partial_{x_{1}}; \varphi\right]J^{s}u\,dx+\frac{1}{2}\int_{\mathbb{R}^{n}} J^{s}u \left[\mathcal{K}_{\alpha}\partial_{x_{1}}; \varphi\right]J^{s}u\,dx\\
&=\Theta_{1,1}(t)+\Theta_{1,2}(t).
\end{split}
\end{equation*}
To handle the term $\Theta_{1,1}$ we define 
\begin{equation}\label{comono1}
\Psi_{c_{\alpha}}:=\left[J^{\alpha}\partial_{x_{1}}; \varphi\right].
\end{equation}
Clearly $\Psi_{c_{\alpha}}\in\mathrm{OP}\mathbb{S}^{\alpha}.$  In virtue of pseudo-differential calculus (see appendix \ref{apendice1}), its principal symbol has the following   decomposition
\begin{equation}
\begin{split}\label{comono1.1}
c_{\alpha}(x,\xi)&=\sum_{|\beta|=1}\frac{
1}{2\pi \mathrm{i}}\partial_{\xi}^{\beta}\left(2\pi \mathrm{i}\xi_{1} \langle 2\pi\xi\rangle^{\alpha}\right)\partial_{x}^{\beta}\varphi+ \sum_{|\beta|=2}\frac{1}{(2\pi \mathrm{i})^{2}}\partial_{\xi}^{\beta}\left(2\pi \mathrm{i}\xi_{1} \langle 2\pi\xi\rangle^{\alpha}\right)\partial_{x}^{\beta}\varphi\\
&\quad +r_{\alpha-2}(x,\xi)\\
&=p_{\alpha}(x,\xi)+p_{\alpha-1}(x,\xi)+r_{\alpha-2}(x,\xi),
\end{split}
\end{equation}
where $r_{\alpha-2}\in\mathbb{S}^{\alpha-2}\subset \mathbb{S}^{0}.$

After rearranging     the  expressions for $p_{\alpha}$ and $p_{\alpha-1}$ yield
\begin{equation}\label{comono2}
\begin{split}
p_{\alpha}(x,\xi)
&= \langle2\pi \xi\rangle^{\alpha} \partial_{x_{1}}\varphi- \alpha \sum_{|\beta|=1}\langle 2\pi\xi \rangle^{\alpha-2}(2\pi\mathrm{i}\xi)^{\beta+\mathrm{e}_{1}}\partial_{x}^{\beta}
\varphi
\end{split}
\end{equation}
and 
	\begin{equation}\label{comono3}
	\begin{split}
p_{\alpha-1}(x,\xi)
&=-\alpha\sum_{|\beta|=1}(2\pi \mathrm{i}\xi_{1})\langle 2\pi\xi\rangle^{\alpha-2}\partial_{x_{1}}\partial_{x}^{\beta}\varphi-\frac{\alpha}{2\pi}\sum_{|\beta|=1}(2\pi \mathrm{i}\xi)^{\beta}\langle 2\pi \xi\rangle ^{\alpha-2}\partial_{x}^{\beta}\partial_{x_{1}}\varphi\\
&\quad  -\frac{\alpha}{2\pi}\sum_{|\beta|=1}(2\pi\mathrm{i}\xi_{1})\langle2\pi \xi \rangle ^{\alpha-2}\partial_{x}^{2\beta}\varphi\\
&\quad +\frac{\alpha(\alpha-2)}{2\pi}\sum_{|\beta_{2}|=1}\sum_{|\beta_{1}|=1}(2\pi\mathrm{i}\xi_{1})(2\pi\mathrm{i}\xi)^{\beta_{1}}(2\pi\mathrm{i}\xi)^{\beta_{2}}\langle 2\pi\xi \rangle^{\alpha-4}\partial_{x}^{\beta_{2}}\partial_{x}^{\beta_{1}}
\varphi.
	\end{split}
	\end{equation}
Therefore,
\begin{equation}\label{comono4}
\Psi_{c_{\alpha}}=p_{\alpha}(x,D)+p_{\alpha-1}(x,D)+r_{\alpha-2}
(x,D),
\end{equation}
where   
\begin{equation}\label{energy}
	p_{\alpha}(x,D)=\partial_{x_{1}}\varphi J^{\alpha}-\alpha\partial_{x_{1}}\varphi J^{\alpha-2}\partial_{x_{1}}^{2}-\alpha\sum_{\mathclap{\substack{|\beta|=1\\\beta\neq \mathrm{e}_{1}}}}\partial_{x}^{\beta}\varphi J^{\alpha-2}\partial_{x}^{\beta}\partial_{x_{1}},
\end{equation}
\begin{equation}\label{energy2.1.1}
	\begin{split}
	p_{\alpha-1}(x,D)
	&=-\alpha\sum_{|\beta|=1}\partial_{x_{1}}\partial_{x}^{\beta}\varphi \partial_{x_{1}}J^{\alpha-2}-\frac{\alpha}{2\pi}\sum_{|\beta|=1}\partial_{x}^{\beta}\partial_{x_{1}}\varphi\partial_{x}^{\beta}J^{\alpha-2}\\
	&-\frac{\alpha}{2\pi}\sum_{|\beta|=1}\partial_{x}^{2\beta}\varphi\partial_{x_{1}}J^{\alpha-2}
	+\frac{\alpha(\alpha-2)}{2\pi}\sum_{|\beta_{2}|=1}\sum_{|\beta_{1}|=1}\partial_{x}^{\beta_{2}}\partial_{x}^{\beta_{1}}\varphi\partial_{x_{1}}\partial_{x}^{\beta_{2}}\partial_{x}^{\beta_{1}}J^{\alpha-4},
\end{split}
\end{equation}
and $r_{\alpha-2}(x,D)\in\mathrm{OP}\mathbb{S}^{\alpha-2}\subset \mathrm{OP}\mathbb{S}^{0}.$

	Thus, after replacing the operators $p_{\alpha}(x,D), p_{\alpha-1}(x,D)$ and $p_{\alpha-2}(x,D)$  in $\Theta_{1,1}$  we get 
	\begin{equation*}
	\begin{split}
	&\Theta_{1,1}(t)\\
	&=\frac{1}{2}\int_{\mathbb{R}^{n}}J^{s}uJ^{s+\alpha}u\partial_{x_{1}}\varphi\,dx-\frac{\alpha}{2}\int_{\mathbb{R}^{n}}J^{s}u J^{s+\alpha-2}\partial_{x_{1}}^{2}u\partial_{x_{1}}\varphi\,dx\\
	 &\quad -\frac{\alpha}{2}\sum_{\mathclap{\substack{|\beta|=1\\\beta\neq \mathrm{e}_{1}}}}\,\int_{\mathbb{R}^{n}}J^{s}uJ^{s+\alpha-2}\partial_{x_{1}}\partial_{x}^{\beta}u \partial_{x}^{\beta}\varphi\,dx-\frac{\alpha}{2}\sum_{\mathclap{\substack{|\beta|=1\\\beta\neq \mathrm{e}_{1}}}}\int_{\mathbb{R}^{n}}J^{s}uJ^{s+\alpha-2}\partial_{x_{1}}u \partial_{x}^{\beta}\partial_{x_{1}}\varphi\,dx\\
	 &\quad -\frac{\alpha}{4\pi}\sum_{|\beta|=1}\int_{\mathbb{R}^{n}} J^{s}u\partial_{x}^{\beta}J^{\alpha-2+s}u \partial_{x_{1}}\partial_{x}^{\beta}\varphi\,dx-\frac{\alpha}{4\pi}\sum_{|\beta|=1}\int_{\mathbb{R}^{n}}J^{s}u\partial_{x_{1}}J^{s+\alpha-2}u\partial_{x}^{2\beta}\varphi\,dx\\
	 &\quad \frac{\alpha(\alpha-2)}{4\pi}\sum_{|\beta_{2}|=1}\sum_{|\beta_{1}|=1}\int_{\mathbb{R}^{n}}J^{s}u\partial_{x_{1}}\partial_{x}^{\beta_{2}}\partial_{x}^{\beta_{1}}J^{\alpha-4+s}u \partial_{x}^{\beta_{2}}\partial_{x}^{\beta_{1}}\varphi\,dx\\
	 &\quad +\frac{1}{2}\int_{\mathbb{R}^{n}}J^{s}ur_{\alpha-2}(x,D)J^{s}u\,dx\\
	 &=\Theta_{1,1,1}(t)+\Theta_{1,1,2}(t)+\Theta_{1,1,3}(t)+\Theta_{1,1,4}(t)+\Theta_{1,1,5}(t)+\Theta_{1,1,6}(t)\\
	 &\quad +\Theta_{1,1,7}(t)+\Theta_{1,1,8}(t).
	\end{split}
	\end{equation*}
%	\begin{flushleft}
%		{\sc  \underline{Term $\Theta_{1,1,1}:$}}
%	\end{flushleft}
In the first place, we rewrite $\Theta_{1,1,1}$ as 
	\begin{equation}\label{r1}
	\begin{split}
	\Theta_{1,1,1}(t)&=\frac{1}{2}\int_{\mathbb{R}^{n}}\left(J^{s+\frac{\alpha}{2}}u\right)^{2}\partial_{x_{1}}\varphi\,dx+\frac{1}{2}\int_{\mathbb{R}^{n}}J^{s+\frac{\alpha}{2}}u\left[J^{\frac{\alpha}{2}};\partial_{x_{1}}\varphi\right]J^{s}u\,dx\\
	&=\Theta_{1,1,1,1}(t)+\Theta_{1,1,1,2}(t).
	\end{split}
	\end{equation}
	The first expression in the r.h.s above represents after integrating in time  the smoothing effect. 
	
	Nevertheless,   the price to pay  for such expression becomes reflected in  estimating $\Theta_{1,1,1,2},$  which is not an easy task to tackle down.
	
 For $s\in\mathbb{R}$ we set 	  
	\begin{equation}\label{eq1}
	a_{s}(x,D):=\left[J^{s}; \phi\right],\quad \phi\in C^{\infty}(\mathbb{R}^{n})
	\end{equation}  
with  $$\sup_{x}|\partial_{x}^{\gamma}\phi(x)|\leq c_{\gamma}\quad \mbox{ for all}\quad  \gamma\in (\mathbb{N}_{0})^{n}.$$

	  By the pseudo-differential calculus, its principal symbol admits the decomposition 
	\begin{equation}\label{eq2}
	a_{s}(x,D)=-s\sum_{|\beta|=1}\partial_{x}^{\beta}\phi \partial_{x}^{\beta}J^{s-2}	+r_{s-2}(x,D),
	\end{equation}
	where  $r_{s-2}(x,D)\in\mathrm{OP}\mathbb{S}^{s-2}.$
	
	Back to our case
	\begin{equation*}
	\begin{split}
&\Theta_{1,1,1,2}(t)\\
&=-\frac{\alpha}{4}\sum_{|\beta|=1}\int_{\mathbb{R}^{n}}J^{s}u\left[J^{\frac{\alpha}{2}}; \partial_{x}^{\beta}\partial_{x_{1}}\varphi\right]\partial_{x}^{\beta}J^{s+\frac{\alpha-4}{2}}u\,dx-\frac{\alpha}{4}\sum_{|\beta|=1}\int_{\mathbb{R}^{n}}J^{s}u\partial_{x}^{\beta}\partial_{x_{1}}\varphi\partial_{x}^{\beta}J^{s+\alpha-2}u\,dx\\
&\quad +\frac{1}{2}\int_{\mathbb{R}^{n}}J^{s}ur_{\frac{\alpha-4}{2}}(x,D)J^{s}u\,dx\\
&=-\frac{\alpha}{4}\sum_{|\beta|=1}\int_{\mathbb{R}^{n}}J^{s}u\left[J^{\frac{\alpha}{2}}; \partial_{x}^{\beta}\partial_{x_{1}}\varphi\right]\partial_{x}^{\beta}J^{s+\frac{\alpha-4}{2}}u\,dx\\
&\quad -\frac{\alpha}{4}\sum_{|\beta|=1}\int_{\mathbb{R}^{n}}J^{s+\frac{\alpha-2}{2}}u\left[J^{\frac{2-\alpha}{2}}; \partial_{x}^{\beta}\partial_{x_{1}}\varphi \right]\partial_{x}^{\beta}
J^{s+\frac{\alpha-2}{2}}u\,dx\\
&\quad +\frac{\alpha}{8}\sum_{|\beta|=1}\int_{\mathbb{R}^{n}}\left(J^{s+\frac{\alpha-2}{2}}u\right)^{2}\partial_{x}^{2\beta}\partial_{x_{1}}\varphi\, dx
 +\frac{1}{2}\int_{\mathbb{R}^{n}}J^{s}ur_{\frac{\alpha-4}{2}}(x,D)J^{s}u\,dx.
	\end{split}
	\end{equation*}
At this point  we have by  Theorem \ref{continuity} that 
\begin{equation}\label{contieq}
\int_{0}^{T}	|\Theta_{1,1,1,2}(t)|\, dt\lesssim_{\alpha} T\|u\|_{L^{\infty}_{T}H^{s}_{x}}^{2}.
\end{equation}
%	\begin{flushleft}
%	{\sc  \underline{Term $\Theta_{1,1,2}:$}}
%\end{flushleft}
On the other hand  we get after rearranging 
		\begin{equation*}
	\begin{split}
	\Theta_{1,1,2}(t)&=\frac{\alpha}{2}\int_{\mathbb{R}^{n}}\partial_{x_{1}}J^{s+\frac{\alpha-2}{2}}u\left[J^{\frac{\alpha-2}{2}};\partial_{x_{1}}\varphi\right]
\partial_{x_{1}}J^{s}u\,dx+\frac{\alpha}{2}\int_{\mathbb{R}^{n}}\left(\partial_{x_{1}}J^{s+\frac{\alpha-2}{2}}u\right)^{2}\partial_{x_{1}}\varphi\,dx\\
&\quad +\frac{\alpha}{2}\int_{\mathbb{R}^{n}}J^{s}u J^{s+\alpha-2}\partial_{x_{1}}u\partial_{x_{1}}^{2}\varphi\,dx\\
&=\frac{\alpha}{2}\int_{\mathbb{R}^{n}}\partial_{x_{1}}J^{s+\frac{\alpha-2}{2}}u\left[J^{\frac{\alpha-2}{2}};\partial_{x_{1}}\varphi\right]
\partial_{x_{1}}J^{s}u\,dx+\frac{\alpha}{2}\int_{\mathbb{R}^{n}}\left(\partial_{x_{1}}J^{s+\frac{\alpha-2}{2}}u\right)^{2}\partial_{x_{1}}\varphi\,dx\\
&\quad +\frac{\alpha}{2}\int_{\mathbb{R}^{n}}J^{s}u J^{s+\alpha-2}\partial_{x_{1}}u\partial_{x_{1}}^{2}\varphi\,dx +\frac{\alpha}{2}\int_{\mathbb{R}^{n}}\partial_{x_{1}}J^{s+\frac{\alpha-2}{2}}u\left[J^{\frac{\alpha-2}{2}};\partial_{x_{1}}^{2}\varphi\right]J^{s}u\,dx\\
&\quad -\frac{\alpha}{4}\int_{\mathbb{R}^{n}}\left(J^{s+\frac{\alpha-2}{2}}u\right)^{2}\partial_{x_{1}}^{3}\varphi\,dx\\
&=\Theta_{1,1,2,1}(t)+\Theta_{1,1,2,2}(t)+\Theta_{1,1,2,3}(t)+\Theta_{1,1,2,4}(t)+\Theta_{1,1,2,5}(t).
	\end{split}
	\end{equation*}
We shall point out that  an argument as the one used in \eqref{eq1}-\eqref{contieq} allow us  to estimate $\Theta_{1,1,2,1},\Theta_{1,1,2,3}$ and $\Theta_{1,1,2,4}$  with the bound
\begin{equation*}
\int_{0}^{T}\max\left\{	|\Theta_{1,1,2,1}(t)|, |\Theta_{1,1,2,3}(t)|,|\Theta_{1,1,2,4}(t)|\right\}\, dt\lesssim_{\alpha}T\|u\|_{L^{\infty}_{T}H^{s}_{x}}^{2}.
\end{equation*} 
Notice that $\Theta_{ 1,1,2,2}$ has the correct sing in front  and the correct regularity desired. More precisely,  it  provides the smoothing effect after integrating in time.

Finally, by the $L^{2}-$continuity (see Theorem \ref{continuity}) 
\begin{equation*}
\int_{0}^{T}	|\Theta_{1,1,2,5}(t)|\, dt \lesssim T\|u\|_{L^{\infty}_{T}H^{s}_{x}}^{2}\|\partial_{x_{1}}^{3}\varphi\|_{L^{\infty}_{x}}.
\end{equation*}   
%	\begin{flushleft}
%	{\sc  \underline{Term $\Theta_{1,1,3}:$}}
%\end{flushleft}
For the term $\Theta_{1,1,3}$ we have 
\begin{equation*}
	\begin{split}
		\Theta_{1,1,3}(t)
		&=	\frac{\alpha}{2}\sum_{\mathclap{\substack{|\beta|=1\\\beta\neq \mathrm{e}_{1}}}}\,\int_{\mathbb{R}^{n}}\partial_{x}^{\beta}J^{s}uJ^{s+\alpha-2}\partial_{x_{1}}u \partial_{x}^{\beta}\varphi\,dx+\frac{\alpha}{2}\sum_{\mathclap{\substack{|\beta|=1\\\beta\neq \mathrm{e}_{1}}}}\,\int_{\mathbb{R}^{n}}J^{s}uJ^{s+\alpha-2}\partial_{x_{1}}u \partial_{x}^{2\beta}\varphi\,dx\\
%		&=\frac{\alpha}{8\pi}\sum_{\mathclap{\substack{|\beta|=1\\\beta\neq \mathrm{e}_{1}}}}\int_{\mathbb{R}^{n}}\partial_{x}^{\beta}J^{s}uJ^{s+\alpha-2}\partial_{x_{1}}u \partial_{x}^{\beta}\varphi\,dx+\frac{\alpha}{8\pi}\sum_{\mathclap{\substack{|\beta|=1\\\beta\neq \mathrm{e}_{1}}}}\int_{\mathbb{R}^{n}}J^{s}uJ^{s+\alpha-2}\partial_{x_{1}}u \partial_{x}^{2\beta}\varphi\,dx\\
		&=\frac{\alpha}{2}\sum_{\mathclap{\substack{|\beta|=1\\\beta\neq \mathrm{e}_{1}}}}\,\int_{\mathbb{R}^{n}}\partial_{x}^{\beta}J^{s+\frac{\alpha-2}{2}}u\left[ J^{\frac{2-\alpha}{2}};\partial_{x}^{\beta}\varphi\right]\partial_{x_{1}}J^{s+\alpha-2}u\,dx\\
		&\quad +\frac{\alpha}{2}\sum_{\mathclap{\substack{|\beta|=1\\\beta\neq \mathrm{e}_{1}}}}\,\int_{\mathbb{R}^{n}}\partial_{x}^{\beta}J^{s+\frac{\alpha-2}{2}}u\partial_{x_{1}}J^{s+\frac{\alpha-2}{2}}u\partial_{x}^{\beta}\varphi\,dx\\
		&\quad +\frac{\alpha}{2}\sum_{\mathclap{\substack{|\beta|=1\\\beta\neq \mathrm{e}_{1}}}}\,\int_{\mathbb{R}^{n}}\partial_{x_{1}}J^{s+\frac{\alpha-2}{2}}u\left[J^{\frac{\alpha-2}{2}};\partial_{x}^{2\beta}\varphi\right]J^{s}u\,dx\\
		&\quad -\frac{\alpha}{4}\sum_{\mathclap{\substack{|\beta|=1\\\beta\neq \mathrm{e}_{1}}}}\int_{\mathbb{R}^{n}}\left(J^{s+\frac{\alpha-2}{2}}u\right)^{2}\partial_{x_{1}}\partial_{x}^{2\beta}\varphi\,dx\\
		&=\Theta_{1,1,3,1}(t)+\Theta_{1,1,3,2}(t)+\Theta_{1,1,3,3}(t)+\Theta_{1,1,3,4}(t).
	\end{split}
\end{equation*}
Thus, after  applying the decomposition \eqref{eq2}   is clear that 
\begin{equation*}
	\begin{split}
		\Theta_{1,1,3,1}(t)&=\frac{\alpha(\alpha-2)}{4}\sum_{\mathclap{\substack{|\beta|=1\\\beta\neq \mathrm{e}_{1}}}}\,\,\sum_{|\gamma|=1}\int_{\mathbb{R}^{n}}\partial_{x}^{\beta}J^{s+\frac{\alpha-2}{2}}u\partial_{x}^{\gamma}\partial_{x_{1}}J^{s+\frac{\alpha-6}{2}}u\partial_{x}^{\beta}\partial_{x}^{\gamma}\varphi\,dx\\
		&\quad +\frac{\alpha}{2}\sum_{\mathclap{\substack{|\beta|=1\\\beta\neq \mathrm{e}_{1}}}}\,\int_{\mathbb{R}^{n}}\partial_{x}^{\beta}J^{s+\frac{\alpha-2}{2}}u r_{\frac{\alpha-6}{2}}(x,D)\partial_{x_{1}}J^{s}u\,dx\\
		&=\frac{\alpha(\alpha-2)}{4}\sum_{\mathclap{\substack{|\beta|=1\\\beta\neq \mathrm{e}_{1}}}}\,\,\sum_{|\gamma|=1}\int_{\mathbb{R}^{n}}\partial_{x}^{\beta}J^{s+\frac{\alpha-2}{2}}u\partial_{x}^{\gamma}\partial_{x_{1}}J^{s+\frac{\alpha-6}{2}}u\partial_{x}^{\beta}\partial_{x}^{\gamma}\varphi\,dx\\
		&\quad -\frac{\alpha}{2}\sum_{\mathclap{\substack{|\beta|=1\\\beta\neq \mathrm{e}_{1}}}}\,\int_{\mathbb{R}^{n}}J^{s+\frac{\alpha-2}{2}}u \partial_{x}^{\beta}r_{\frac{\alpha-6}{2}}(x,D)\partial_{x_{1}}J^{s}u\,dx\\
		&=	\Theta_{1,1,3,1,1}(t)+	\Theta_{1,1,3,1,2}(t),
	\end{split}
\end{equation*} 
where $r_{\frac{\alpha-6}{2}}(x,D)\in \mathrm{OP}\mathbb{S}^{\frac{\alpha-6}{2}}\subset \mathrm{OP}\mathbb{S}^{0}.$

Since 
\begin{equation*}
	\begin{split}
		\Theta_{1,1,3,1,1}(t)&=	\frac{\alpha(\alpha-2)}{4}\sum_{\mathclap{\substack{|\beta|=1\\\beta\neq \mathrm{e}_{1}}}}\quad\sum_{\mathclap{\substack{|\gamma|=1\\\beta\neq \mathrm{e}_{1}}}}\,\int_{\mathbb{R}^{n}}\partial_{x}^{\beta}J^{s+\frac{\alpha-2}{2}}u\partial_{x}^{\gamma}\partial_{x_{1}}J^{s+\frac{\alpha-6}{2}}u\partial_{x}^{\beta}\partial_{x}^{\gamma}\varphi\,dx\\
		&\quad +\frac{\alpha(\alpha-2)}{4}\sum_{\mathclap{\substack{|\beta|=1\\\beta\neq \mathrm{e}_{1}}}}\int_{\mathbb{R}^{n}}\partial_{x}^{\beta}J^{s+\frac{\alpha-2}{2}}u\partial_{x_{1}}\partial_{x_{1}}J^{s+\frac{\alpha-6}{2}}u\partial_{x}^{\beta}\partial_{x_{1}}\varphi\,dx\\
		&=\Xi_{1}(t)+\Xi_{2}(t).
	\end{split}
\end{equation*}
%\begin{flushleft}
%		{\sc  \underline{Localization}}
%\end{flushleft}
The next terms are quite important since these ones  determine the kind of sets where the smoothing can take place. More precisely,   it force us to impose conditions on the weighted function $\varphi$  to decouple certain  terms.

To handle $\Xi_{1}$ notice that there exists a skew symmetric operator $\Psi_{-1}\in \mathrm{OP}\mathbb{S}^{-1},$    such that    
\begin{equation*}
\begin{split}
	\Xi_{1}(t)=	\frac{\alpha(\alpha-2)}{8}\int_{\mathbb{R}^{n}}\partial_{x}^{\beta}\partial_{x}^{\gamma}\varphi f\Psi_{-1}f\, \,dx,
\end{split}
\end{equation*}
where
\begin{equation*}
	f:=\sum_{\mathclap{\substack{|\beta|=1\\\beta\neq \mathrm{e}_{1}}}}\partial_{x}^{\beta}J^{s+\frac{\alpha-2}{2}}u.
\end{equation*}
If we  assume that $\varphi$ has the following representation 
\begin{equation}\label{function1}
	\varphi(x)=\phi\left(\nu\cdot x+\delta\right)\quad x\in\mathbb{R}^{n},\delta\in\mathbb{R},
\end{equation}
where $\nu=(\nu_{1},\nu_{2},\dots, \nu_{n}) \in\mathbb{R}^{n}$ is a non-null vector and $\phi:\mathbb{R}\longrightarrow\mathbb{R}$ satisfies  \eqref{weight},
this assumption  allow us to  say   
\begin{equation*}
	\begin{split}
			\Xi_{1}(t)&=	\frac{\alpha(2-\alpha)}{16}\sum_{\mathclap{\substack{|\beta|=1\\\beta\neq \mathrm{e}_{1}}}}\quad \sum_{\mathclap{\substack{|\gamma|=1\\\gamma\neq \mathrm{e}_{1}}}}\,\,\int_{\mathbb{R}^{n}}\partial_{x}^{\beta}J^{s+\frac{\alpha-2}{2}}u\left[\Psi_{-1};\phi\right]\partial_{x}^{\gamma}J^{s+\frac{\alpha-2}{2}}u \,dx\\
		&=\frac{\alpha(\alpha-2)}{16}\sum_{\mathclap{\substack{|\beta|=1\\\beta\neq \mathrm{e}_{1}}}}\quad \sum_{\mathclap{\substack{|\gamma|=1\\\gamma\neq \mathrm{e}_{1}}}}\,\,\int_{\mathbb{R}^{n}}J^{s+\frac{\alpha-2}{2}}u\,\partial_{x}^{\beta}\left[\Psi_{-1};\phi\right]\partial_{x}^{\gamma}J^{s+\frac{\alpha-2}{2}}u \,dx.
	\end{split}
\end{equation*}
Hence,  combining pseudo-differential calculus  with 	Theorem \eqref{continuity} imply that 
\begin{equation*}
	\begin{split}
	\int_{0}^{T}	|\Xi_{1}(t)|\, dt \lesssim_{n}T\left(\frac{\alpha|\alpha-2|}{16}+1\right)\|u\|_{L^{\infty}_{T}H^{s}_{x}}^{2},
	\end{split}
\end{equation*}
and 
\begin{equation*}
	\begin{split}
		\int_{0}^{T}|	\Theta_{1,1,3,1,2}(t)|\,dt \lesssim_{\alpha}T(n-1)\|u\|_{L^{\infty}_{T}H^{s}_{x}}^{2}.
	\end{split}
\end{equation*}
To provide some control over $	\Theta_{1,1,3,1,1}$ is not an easy task at all since  several interactions between the variables have to be taken into consideration.  %Nevertheless,  if  we  impose some conditions on the weighted function $\varphi$ such as:
%\textcolor{blue}{to be included}

Thus,  after  using  \eqref{eq2} yield
\begin{equation*}
	\begin{split}
		\Theta_{1,1,3,3}(t)&=\frac{\alpha(2-\alpha)}{8}\sum_{\mathclap{\substack{|\beta|=1\\\beta\neq \mathrm{e}_{1}}}}\,\sum_{|\gamma|=1}\int_{\mathbb{R}^{n}}\partial_{x}^{2\beta}\partial_{x}^{\gamma}\varphi  \partial_{x_{1}}J^{s+\frac{\alpha-2}{2}}u\partial_{x}^{\gamma}J^{s+\frac{\alpha-6}{2}}u\,dx\\
		&\quad -\frac{\alpha}{8\pi}\sum_{\mathclap{\substack{|\beta|=1\\\beta\neq \mathrm{e}_{1}}}}\,\int_{\mathbb{R}^{n}}\partial_{x_{1}}J^{s+\frac{\alpha-2}{2}}ur_{\frac{\alpha-6}{2}}(x,D)J^{s}u\,dx\\
		&=\frac{\alpha(2-\alpha)}{8}\sum_{\mathclap{\substack{|\beta|=1\\\beta\neq \mathrm{e}_{1}}}}\,\sum_{|\gamma|=1}\int_{\mathbb{R}^{n}}\partial_{x}^{2\beta}\partial_{x}^{\gamma}\varphi  \partial_{x_{1}}J^{s+\frac{\alpha-2}{2}}u\partial_{x}^{\gamma}J^{s+\frac{\alpha-6}{2}}u\,dx\\
		&\quad +\frac{\alpha}{8\pi}\sum_{\mathclap{\substack{|\beta|=1\\\beta\neq \mathrm{e}_{1}}}}\,\int_{\mathbb{R}^{n}}J^{s+\frac{\alpha-2}{2}}u  \partial_{x_{1}}r_{\frac{\alpha-6}{2}}(x,D)J^{s}u\,dx,
	\end{split}
\end{equation*}
where $\partial_{x_{1}}r_{\frac{\alpha-6}{2}}(x,D)\in \mathrm{OP}\mathbb{S}^{\frac{\alpha-4}{2}}\subset \mathrm{OP}\mathbb{S}^{0}.$

Since $0<\alpha<2,$  we obtain by the $L^{2}-$continuity (Theorem \ref{continuity})
\begin{equation*}
	\begin{split}
	\int_{0}^{T}	|\Theta_{1,1,3,3}(t)|\,dt 
		&\lesssim T\|u\|_{L^{\infty}_{T}H^{s}_{x}}^{2}\left\{\alpha(2-\alpha)\sum_{\mathclap{\substack{|\beta|=1\\\beta\neq \mathrm{e}_{1}}}}\,\sum_{|\gamma|=1}\left\|\partial_{x}^{2\beta}\partial_{x}^{\gamma}\varphi\right\|_{L^{\infty}_{x}}+\alpha(n-1)\right\}
	\end{split}
\end{equation*}
and 
\begin{equation*}
\int_{0}^{T}	|\Theta_{1,1,3,4}(t)|\, dt\lesssim_{\alpha}T\|u\|_{L^{\infty}_{T}H^{s}_{x}}^{2}\,\sum_{\mathclap{\substack{|\beta|=1\\\beta\neq \mathrm{e}_{1}}}}\left\|\partial_{x}^{2\beta}\partial_{x_{1}}\varphi\right\|_{L^{\infty}_{x}}.
\end{equation*}
The term  $\Theta_{1,1,3,2}$ will be  crucial in our analysis, since from it we shall extract the desired smoothing effect after integrating in time. Nevertheless, we postpone this tedious task for the next section where we unify all the estimates. % terms with such regularity properties.
%	\begin{flushleft}
%	{\sc  \underline{Term $\Theta_{1,1,4}:$}}
%\end{flushleft}

First, we rewrite the term as follows
\begin{equation*}
	\begin{split}
		\Theta_{1,1,4}(t)&=	-\frac{\alpha}{2}\,\sum_{\mathclap{\substack{|\beta|=1\\\beta\neq \mathrm{e}_{1}}}}\, \int_{\mathbb{R}^{n}} J^{s+\frac{\alpha-2}{2}}u\left[J^{\frac{2-\alpha}{2}}; \partial_{x_{1}}\partial_{x}^{\beta}\varphi\right]\partial_{x_{1}}J^{\alpha-2+s}u\,dx\\
		&\quad +\frac{\alpha}{4}\sum_{\mathclap{\substack{|\beta|=1\\\beta\neq \mathrm{e}_{1}}}}\, \int_{\mathbb{R}^{n}}\left(J^{s+\frac{\alpha-2}{2}}u\right)^{2}\partial_{x_{1}}^{2}\partial_{x}^{\beta}
		\varphi\,dx\\
		&=\Theta_{1,1,4,1}(t)+\Theta_{1,1,4,2}(t).
	\end{split}
\end{equation*}
By Theorem \ref{continuity} it follows
\begin{equation*}
\int_{0}^{T}	|\Theta_{1,1,4,2}(t)|\,dt\lesssim_{\alpha}T\|u\|_{L^{\infty}_{T}H^{s}_{x}}^{2}\sum_{\mathclap{\substack{|\beta|=1\\\beta\neq \mathrm{e}_{1}}}}\left\|\partial_{x_{1}}^{2}\partial_{x}^{\beta}\varphi\right\|_{L^{\infty}_{x}},
\end{equation*}
and for the remainder term we use decomposition \eqref{eq1}-\eqref{eq2} to obtain
\begin{equation*}
	\begin{split}
		\Theta_{1,1,4,1}(t)&=\frac{\pi\alpha(2-\alpha)}{2}\sum_{|\beta|=1}\sum_{|\gamma|=1}\int_{\mathbb{R}^{n}}J^{s+\frac{\alpha-2}{2}}u \partial_{x}^{\gamma}\partial_{x_{1}}\partial_{x}^{\beta} \varphi\partial_{x}^{\gamma}\partial_{x_{1}}J^{s+\frac{\alpha-6}{2}}\,dx\\
		&\quad-\frac{\alpha}{2}\sum_{\mathclap{\substack{|\beta|=1\\\beta\neq \mathrm{e}_{1}}}}\,\int_{\mathbb{R}^{n}} J^{s+\frac{\alpha-2}{2}}ur_{-\left(\frac{\alpha+2}{2}\right)}(x,D)\partial_{x}^{\beta}J^{\alpha-2+s}u\,dx,
	\end{split}
\end{equation*}
where $r_{-\left(\frac{\alpha+2}{2}\right)}(x,D)\in \mathrm{OP}\mathbb{S}^{-\left(\frac{\alpha+2}{2}\right)}\subset \mathrm{OP}\mathbb{S}^{0}.$

In virtue of Theorem \ref{continuity}
\begin{equation*}
	\begin{split}
		\int_{0}^{T}|\Theta_{1,1,4,1}(t)|\, dt
		&\lesssim T \|u\|_{L^{\infty}_{T}H^{s}_{x}}^{2} \left\{\frac{\alpha(2-\alpha)}{2}\sum_{|\beta|=1}\sum_{|\gamma|=1} \left\|\partial_{x_{1}}\partial_{x}^{\beta}\partial_{x}^{\gamma}\varphi\right\|_{L^{\infty}_{x}}   +\frac{(n-1)\alpha}{4\pi}\right\}.
	\end{split}
\end{equation*}
%	\begin{flushleft}
%	{\sc  \underline{Term $\Theta_{1,1,5}:$}}
%\end{flushleft}
On the other hand 
\begin{equation*}
	\begin{split}
	\Theta_{1,1,5}(t)&=	-\frac{\alpha}{4\pi}\sum_{|\beta|=1}\int_{\mathbb{R}^{n}} J^{s+\frac{\alpha-2}{2}}u\left[J^{\frac{2-\alpha}{2}}; \partial_{x_{1}}\partial_{x}^{\beta}\varphi\right]\partial_{x}^{\beta}J^{\alpha-2+s}u\,dx\\
	&\quad +\frac{\alpha}{8\pi}\sum_{|\beta|=1}\int_{\mathbb{R}^{n}}\left(J^{s+\frac{\alpha-2}{2}}u\right)^{2}\partial_{x_{1}}\partial_{x}^{2\beta}
	\varphi\,dx\\
	&=\Theta_{1,1,5,1}(t)+\Theta_{1,1,5,2}(t).
	\end{split}
\end{equation*}
For the first term above  we have by Theorem \ref{continuity} 
\begin{equation*}
\int_{0}^{T}|\Theta_{1,1,5,2}(t)|\,dt \lesssim T\|u\|_{L^{\infty}_{T}H^{s}_{x}}^{2}\left(\frac{\alpha}{8\pi}\sum_{|\beta|=1}\left\|\partial_{x_{1}}\partial_{x}^{2\beta}\varphi\right\|_{L^{\infty}_{x}}\right).
\end{equation*}
Instead, for  $\Theta_{1,1,5,2}$ we require an extra work  that can be   handled by using  \eqref{eq1}-\eqref{eq2}. More precisely,
\begin{equation}\label{eq3}
	\begin{split}
		\Theta_{1,1,5,1}(t)&=\frac{\alpha(2-\alpha)}{4}\sum_{|\beta|=1}\sum_{|\gamma|=1}\int_{\mathbb{R}^{n}}J^{s+\frac{\alpha-2}{2}}u \partial_{x}^{\gamma}\partial_{x_{1}}\partial_{x}^{\beta} \varphi\partial_{x}^{\gamma}\partial_{x}^{\beta}J^{s+\frac{\alpha-6}{2}}\,dx\\
		&\quad -\frac{\alpha}{4\pi}\sum_{|\beta|=1}\int_{\mathbb{R}^{n}} J^{s+\frac{\alpha-2}{2}}ur_{-\left(\frac{\alpha+2}{2}\right)}(x,D)\partial_{x}^{\beta}J^{\alpha-2+s}u\,dx,
	\end{split}
\end{equation}
where $r_{-\left(\frac{\alpha+2}{2}\right)}(x,D)\in \mathrm{OP}\mathbb{S}^{-\left(\frac{\alpha+2}{2}\right)}\subset \mathrm{OP}\mathbb{S}^{0}.$

At this point  we obtain by Theorem \ref{continuity}
\begin{equation}\label{eq4}
	\begin{split}
		\int_{0}^{T}|\Theta_{1,1,5,1}(t)|\,dt 
		&\lesssim T \|u\|_{L^{\infty}_{T}H^{s}_{x}}^{2} \left\{\frac{\alpha(2-\alpha)}{4}\sum_{|\beta|=1}\sum_{|\gamma|=1} \left\|\partial_{x_{1}}\partial_{x}^{\beta}\partial_{x}^{\gamma}\varphi\right\|_{L^{\infty}_{x}}   +\frac{n\alpha}{4\pi}\right\}.
	\end{split}
\end{equation}
%	\begin{flushleft}
%	{\sc  \underline{Term $\Theta_{1,1,6}:$}}
%\end{flushleft}
Notice that  this term is quite similar to the previous one and the way to  bound it follows the same. Indeed,
\begin{equation*}
	\begin{split}
		\Theta_{1,1,6}(t)&=-\frac{\alpha}{4\pi}\sum_{|\beta|=1}\int_{\mathbb{R}^{n}} J^{s+\frac{\alpha-2}{2}}u\left[J^{\frac{2-\alpha}{2}}; \partial_{x}^{2\beta}\varphi\right]\partial_{x_{1}}J^{\alpha-2+s}u\,dx\\
		&\quad +\frac{\alpha}{8\pi}\sum_{|\beta|=1}\int_{\mathbb{R}^{n}}\left(J^{s+\frac{\alpha-2}{2}}u\right)^{2}\partial_{x_{1}}\partial_{x}^{2\beta}
		\varphi\,dx\\
		&=\Theta_{1,1,6,1}(t)+\Theta_{1,1,6,2}(t).
	\end{split}
\end{equation*}
Applying an argument similar to the one in \eqref{eq3}-\eqref{eq4} \textit{mutatis mutandis} yield
\begin{equation*}
\int_{0}^{T}	|\Theta_{1,1,6,2}(t)|\, dt \lesssim T\|u\|_{L^{\infty}_{T}H^{s}_{x}}^{2}\left(\frac{\alpha}{8\pi}\sum_{|\beta|=1}\left\|\partial_{x_{1}}\partial_{x}^{2\beta}\varphi\right\|_{L^{\infty}_{x}}\right)
\end{equation*}
and 
\begin{equation*}
	\begin{split}
	\int_{0}^{T}	|\Theta_{1,1,6,1}(t)|\, dt 
		&\lesssim T \|u\|_{L^{\infty}_{T}H^{s}_{x}}^{2} \left\{\frac{\alpha(2-\alpha)}{4}\sum_{|\beta|=1}\sum_{|\gamma|=1} \left\|\partial_{x}^{2\beta}\partial_{x}^{\gamma}\varphi\right\|_{L^{\infty}_{x}}   +\frac{n\alpha}{4\pi}\right\}.
	\end{split}
\end{equation*}
%\begin{flushleft}
%	{\sc  \underline{Term $\Theta_{1,1,7}:$}}
%\end{flushleft}
After rewriting 
\begin{equation*}
	\begin{split}
		\Theta_{1,1,7}(t)&=	\frac{\alpha(\alpha-2)}{4\pi}\sum_{|\beta_{2}|=1}\sum_{|\beta_{1}|=1}\int_{\mathbb{R}^{n}}J^{s}u\partial_{x_{1}}\partial_{x}^{\beta_{2}}\partial_{x}^{\beta_{1}}J^{\alpha-4+s}u \partial_{x}^{\beta_{2}}\partial_{x}^{\beta_{1}}\varphi\,dx\\
		&=-\frac{\alpha(\alpha-2)}{4\pi}\sum_{|\beta_{2}|=1}\sum_{|\beta_{1}|=1}\int_{\mathbb{R}^{n}}\partial_{x}^{\beta_{1}}J^{s}u\partial_{x}^{\beta_{2}}\partial_{x_{1}}J^{\alpha-4+s}u \partial_{x}^{\beta_{2}}\partial_{x}^{\beta_{1}}\varphi\,dx\\
		&\quad -\frac{\alpha(\alpha-2)}{4\pi}\sum_{|\beta_{2}|=1}\sum_{|\beta_{1}|=1}\int_{\mathbb{R}^{n}}J^{s}u\partial_{x}^{\beta_{2}}\partial_{x_{1}}J^{\alpha-4+s}u \partial_{x}^{\beta_{2}}\partial_{x}^{2\beta_{1}}\varphi\,dx\\
		&=-\frac{\alpha(\alpha-2)}{4\pi}\sum_{|\beta_{2}|=1}\sum_{|\beta_{1}|=1}\int_{\mathbb{R}^{n}}J^{s+\frac{\alpha-4}{2}}\partial_{x}^{\beta_{1}}u\partial_{x_{1}}\partial_{x}^{\beta_{2}}J^{s+\frac{\alpha-4}{2}}u\partial_{x}^{\beta_{1}}\partial_{x}^{\beta_{2}}\varphi\,dx\\
		&\quad-\frac{\alpha(\alpha-2)}{4\pi}\sum_{|\beta_{2}|=1}\sum_{|\beta_{1}|=1}\int_{\mathbb{R}^{n}}J^{s+\frac{\alpha-4}{2}}\partial_{x}^{\beta_{1}}u\left[J^{\frac{4-\alpha}{2}};\partial_{x}^{\beta_{1}}\partial_{x}^{\beta_{2}}\varphi\right]\partial_{x_{1}}J^{s+\alpha-4}\partial_{x}^{\beta_{2}}u\,dx\\
		&\quad-\frac{\alpha(\alpha-2)}{4\pi}\sum_{|\beta_{2}|=1}\sum_{|\beta_{1}|=1}\int_{\mathbb{R}^{n}}J^{s}u\partial_{x}^{\beta_{2}}\partial_{x_{1}}J^{\alpha-4+s}u \partial_{x}^{\beta_{2}}\partial_{x}^{2\beta_{1}}\varphi\,dx\\
		&=\Theta_{1,1,7,1}(t)+\Theta_{1,1,7,2}(t)+\Theta_{1,1,7,3}(t).
	\end{split}
\end{equation*}
In virtue of \eqref{function1} 
\begin{equation*}
	\begin{split}
		\Theta_{1,1,7,1}(t)&=-\frac{\alpha(\alpha-2)}{4\pi}\sum_{|\beta_{2}|=1}\sum_{|\beta_{1}|=1}\int_{\mathbb{R}^{n}}J^{s+\frac{\alpha-4}{2}}\partial_{x}^{\beta_{1}}u\partial_{x_{1}}\partial_{x}^{\beta_{2}}J^{s+\frac{\alpha-4}{2}}u \nu^{\beta_{1}}\nu^{\beta_{2}}\phi''\, dx\\
		&=-\frac{\alpha(\alpha-2)}{4\pi}\sum_{|\beta_{2}|=1}\sum_{|\beta_{1}|=1}\int_{\mathbb{R}^{n}}J^{s+\frac{\alpha-4}{2}}\partial_{x}^{\beta_{1}}u\partial_{x_{1}}\partial_{x}^{\beta_{2}}J^{s+\frac{\alpha-4}{2}}u \nu^{\beta_{1}}\nu^{\beta_{2}}\phi''\, dx\\
		&=-\frac{\alpha(\alpha-2)}{4\pi}\int_{\mathbb{R}^{n}}\left(\sum_{|\beta_{1}|=1}\partial_{x}^{\beta_{1}}J^{s+\frac{\alpha-4}{2}}u\, \nu^{\beta_{1}}\right)\partial_{x_{1}}\left(\sum_{|\beta_{2}|=1}\partial_{x}^{\beta_{2}}J^{s+\frac{\alpha-4}{2}}u\, \nu^{\beta_{2}}\right)\,\phi''\, dx\\
		&=\frac{\alpha(2-\alpha)\nu_{1}}{8\pi}\int_{\mathbb{R}^{n}}\left(\sum_{|\beta|=1}\partial_{x}^{\beta}J^{s+\frac{\alpha-4}{2}}u\, \nu^{\beta}\right)^{2}\,\phi'''\, dx,
	\end{split}
\end{equation*}
which  by continuity implies 
\begin{equation*}
	\begin{split}
\int_{0}^{T}|\Theta_{1,1,7,1}(t)|\, dt &\lesssim_{\alpha,n}\int_{0}^{T}\sum_{|\beta|=1}\int_{\mathbb{R}^{n}}\nu^{2\beta}\left(\partial_{x}^{\beta}J^{s+\frac{\alpha-4}{2}}u\right) ^{2}|\phi'''|\,dx\, dt\\
&\lesssim_{\alpha,n,\nu}T\|u\|_{L^{\infty}_{T}H^{s}_{x}}^{2}\|\phi'''\|_{L^{\infty}_{x}},
	\end{split}
\end{equation*}
and 
\begin{equation*}
\int_{0}^{T}	|\Theta_{1,1,7,3}(t)|\, dt \lesssim T\frac{\alpha(2-\alpha)}{4\pi}\|u\|_{L^{\infty}_{T}H^{s}_{x}}^{2}\sum_{|\beta_{2}|=1}\sum_{|\beta_{1}|=1}\left\|\partial_{x}^{2\beta_{1}}\partial_{x}^{\beta_{2}}\varphi\right\|_{L^{\infty}_{x}}.
\end{equation*}
After adapting the decomposition  \eqref{eq1}-\eqref{eq2}  to our case  we can rewrite   the term above as 
\begin{equation*}
	\begin{split}
		&\Theta_{1,1,7,2}(t)\\
		&=\frac{\alpha(\alpha-2)(4-\alpha)}{4}\sum_{|\gamma|=1}\sum_{|\beta_{2}|=1}\sum_{|\beta_{1}|=1}\int_{\mathbb{R}^{n}}\partial_{x}^{\gamma}\partial_{x}^{\beta_{1}}\partial_{x}^{\beta_{2}}\varphi J^{s+\frac{\alpha-4}{2}}\partial_{x}^{\beta_{1}}uJ^{s+\frac{\alpha-8}{2}}\partial_{x_{1}}\partial_{x}^{\beta_{2}}\partial_{x}^{\beta_{1}}u\,dx\\
		&\quad -\frac{\alpha(\alpha-2)}{4\pi}\sum_{|\gamma|=1}\sum_{|\beta_{2}|=1}\sum_{|\beta_{1}|=1}\int_{\mathbb{R}^{n}}J^{s+\frac{\alpha-4}{2}}\partial_{x}^{\beta_{1}}u\, r_{-\frac{\alpha}{2}}(x,D)\partial_{x_{1}}J^{s+\alpha-4}\partial_{x}^{\beta_{2}}u\,dx.
	\end{split}
\end{equation*}
Hence, by continuity 
\begin{equation*}
	\begin{split}
		&\int_{0}^{T}	|\Theta_{1,1,7,2}(t)|\,dt\\
		&\lesssim T \|u\|_{L^{\infty}_{T}H^{s}_{x}}^{2}\left(\frac{\alpha(2-\alpha)(4-\alpha)}{4}\sum_{|\gamma|=1}\sum_{|\beta_{2}|=1}\sum_{|\beta_{1}|=1}\left\|\partial_{x}^{\gamma}\partial_{x}^{\beta_{1}}\partial_{x}^{\beta_{2}}\varphi\right\|_{L^{\infty}_{x}}+\frac{\alpha(2-\alpha)}{4\pi}n^{3}\right).
	\end{split}
\end{equation*}
%	\begin{flushleft}
%	{\sc  \underline{Term $\Theta_{1,1,8}:$}}
%\end{flushleft}
Since $r_{\alpha-2}(x,D)\in \mathrm{OP}\mathbb{S}^{\alpha-2}\subset  \mathrm{OP}\mathbb{S}^{0},$   the $L^{2}-$continuity of the  order zero pseudo-differential operators  implies
\begin{equation*}
	\int_{0}^{T}|\Theta_{1,1,8}(t)|\,dt\lesssim T\|u\|_{L^{\infty}_{T}H^{s}_{x}}^{2}.
\end{equation*}
\begin{flushleft}
	{\sc  \underline{Claim 1:}}
\end{flushleft}
There exist a constant $\lambda=\lambda(n,\nu,\alpha)>0,$ such that 
\begin{equation}\label{claim1}
	\begin{split}
	&\lambda\left(\int_{\mathbb{R}^{n}}\left(J^{s+\frac{\alpha}{2}}u\right)^{2}\partial_{x_{1}}\varphi\,dx+\int_{\mathbb{R}^{n}}\left(\partial_{x_{1}}J^{s+\frac{\alpha-2}{2}}u\right)^{2}\partial_{x_{1}}\varphi\,dx\right)\\
	&\leq \Theta_{1,1,1,1}(t)+\Theta_{1,1,2,2}(t)+\Theta_{1,1,3,2}(t).
	\end{split}
\end{equation}
We shall remind that  from \eqref{function1}
	\begin{equation*}
		\varphi(x)=\phi\left(\nu\cdot x+\delta\right)\quad x\in\mathbb{R}^{n}.
	\end{equation*}
	Therefore,
	\begin{equation}\label{ineq1}
		\begin{split}
		|	\Theta_{1,1,3,2}(t)|
			&\leq\frac{\alpha}{2}\sum_{\mathclap{\substack{|\beta|=1\\\beta\neq \mathrm{e}_{1}}}}\,\nu^{\beta}\int_{\mathbb{R}^{n}}
		\left|	\partial_{x}^{\beta}J^{s+\frac{\alpha-2}{2}}u\partial_{x_{1}}J^{s+\frac{\alpha-2}{2}}u\right|
		\,\phi' \,dx\\
			&\leq\frac{\alpha}{2}\sum_{\mathclap{\substack{|\beta|=1\\\beta\neq \mathrm{e}_{1}}}}\,\nu^{\beta}\left\|\eta\Psi_{\beta}J^{s+\frac{\alpha}{2}}u\right\|_{L^{2}_{x}}\left\|\eta\partial_{x_{1}}J^{s+\frac{\alpha-2}{2}}u\right\|_{L^{2}_{x}}\\
			&=\frac{\alpha}{2}\sum_{\mathclap{\substack{|\beta|=1\\\beta\neq \mathrm{e}_{1}}}}\,\nu^{\beta}\left(\left\|\Psi_{\beta}\left(J^{s+\frac{\alpha}{2}}u\eta\right)\right\|_{L^{2}_{x}}+\left\|\left[\Psi_{\beta};\eta\right]J^{s+\frac{\alpha}{2}}u\right\|_{L^{2}_{x}}\right)\left\|\eta\partial_{x_{1}}J^{s+\frac{\alpha-2}{2}}u\right\|_{L^{2}_{x}},
		\end{split}
	\end{equation}
where $\eta:=\left(\phi'\right)^{\frac{1}{2}}$ and $\Psi_{\beta}:=\partial_{x}^{\beta}J^{-1}.$

Additionally,
\begin{equation}\label{sharp}
	\left\|\Psi_{\beta}\left(J^{s+\frac{\alpha}{2}}u\eta\right)\ \right\|_{L^{2}_{x}}\leq C \left\|J^{s+\frac{\alpha}{2}}u\eta\right\|_{L^{2}_{x}},
\end{equation}
where  $$C:=\inf_{f\in L^{2}(\mathbb{R}^{n}),f\neq 0}\frac{\|J^{-1}\partial_{x_{j}}f\|_{L^{2}}}{\|f\|_{L^{2}}},\quad j=2,3,\dots,n.$$

Also,
\begin{equation*}
	\left\|\left[\Psi_{\beta};\eta\right]J^{s+\frac{\alpha}{2}}u\right\|_{L^{2}_{x}}\leq c(\alpha,\beta)\|J^{s}u(t)\|_{L^{2}_{x}}.
\end{equation*}
Thus, for $\epsilon>0,$
	\begin{equation}\label{ineq2}
	\begin{split}
		&\frac{\alpha}{2}\sum_{\mathclap{\substack{|\beta|=1\\\beta\neq \mathrm{e}_{1}}}}\,\nu^{\beta}\left(\left\|\Psi_{\beta}\left(J^{s+\frac{\alpha}{2}}u\eta\right)\right\|_{L^{2}_{x}}+\left\|\left[\Psi_{\beta};\eta\right]J^{s+\frac{\alpha}{2}}u\right\|_{L^{2}_{x}}\right)\left\|\eta\partial_{x_{1}}J^{s+\frac{\alpha-2}{2}}u\right\|_{L^{2}_{x}}\\
		&\leq \frac{\alpha \sqrt{n-1}|\overline{\nu}|C}{2}\left\|J^{s+\frac{\alpha}{2}}u\eta\right\|_{L^{2}_{x}}\left\|\eta\partial_{x_{1}}J^{s+\frac{\alpha-2}{2}}u\right\|_{L^{2}_{x}}+\frac{\alpha}{8\epsilon}\sum_{\substack{|\beta|=1\\\beta\neq \mathrm{e}_{1}}}\,\nu^{\beta}	\left\|\left[\Psi_{\beta};\eta\right]J^{s+\frac{\alpha}{2}}u\right\|_{L^{2}_{x}}^{2}
		\\
		&\quad +\frac{\alpha\epsilon\sqrt{n-1} |\overline{\nu}|}{2}\left\|\eta\partial_{x_{1}}J^{s+\frac{\alpha-2}{2}}u\right\|_{L^{2}_{x}},
	\end{split}
\end{equation}
where $|\overline{\nu}|:=\sqrt{\nu_{2}^{2}+\nu_{3}^{2}+\dots+\nu_{n}^{2}}.$

Hence,
%\begin{equation*}
%	\begin{split}
%		&\Theta_{1,1,1,1}(t)+\Theta_{1,1,2,2}(t)+\Theta_{1,1,3,2}(t)\\
%		&\geq \frac{\nu_{1}}{2}\left\|J^{s+\frac{\alpha}{2}}u\eta\right\|_{L^{2}_{x}}^{2}+ \frac{\alpha\nu_{1}}{2}\left\|\eta\partial_{x_{1}}J^{s+\frac{\alpha-2}{2}}u\right\|_{L^{2}_{x}}^{2}-\frac{\alpha}{2}\sum_{j=2}^{n}\nu_{j}\left\|J^{s+\frac{\alpha}{2}}u\eta\right\|_{L^{2}_{x}}\left\|\eta\partial_{x_{1}}J^{s+\frac{\alpha-2}{2}}u\right\|_{L^{2}_{x}}\\
%		&\quad -\frac{\alpha}{8\epsilon}\sum_{\substack{|\beta|=1\\\beta\neq \mathrm{e}_{1}}}\,\nu^{\beta}	\left\|\left[\Psi_{\beta};\eta\right]J^{s+\frac{\alpha}{2}}u\right\|_{L^{2}_{x}}^{2} -\frac{\alpha\epsilon\sqrt{n-1} |\overline{\nu}|}{2}\left\|\eta\partial_{x_{1}}J^{s+\frac{\alpha-2}{2}}u\right\|_{L^{2}_{x}},
%	\end{split}
%\end{equation*}
%which is equivalent to  the inequality
\begin{equation}\label{ineq1}
	\begin{split}
		& \frac{\nu_{1}}{2}\left\|J^{s+\frac{\alpha}{2}}u\eta\right\|_{L^{2}_{x}}^{2}+ \frac{\alpha\nu_{1}}{2}\left\|\eta\partial_{x_{1}}J^{s+\frac{\alpha-2}{2}}u\right\|_{L^{2}_{x}}^{2}-\frac{\alpha |\overline{\nu}|\sqrt{n-1}C}{2}\left\|J^{s+\frac{\alpha}{2}}u\eta\right\|_{L^{2}_{x}}\left\|\eta\partial_{x_{1}}J^{s+\frac{\alpha-2}{2}}u\right\|_{L^{2}_{x}}\\
		&\quad  -\frac{\alpha\epsilon\sqrt{n-1} |\overline{\nu}|}{2}\left\|\eta\partial_{x_{1}}J^{s+\frac{\alpha-2}{2}}u\right\|_{L^{2}_{x}}^{2}\\
		&\leq \frac{\alpha}{8\epsilon}\sum_{\substack{|\beta|=1\\\beta\neq \mathrm{e}_{1}}}\,\nu^{\beta}	c(\alpha,\beta)^{2}\|J^{s}u(t)\|_{L^{2}_{x}}^{2}+\Theta_{1,1,1,1}(t)+\Theta_{1,1,2,2}(t)+\Theta_{1,1,3,2}(t).
	\end{split}
\end{equation}
At this point we  shall make emphasis in   two possible situations that we proceed to  discuss below.
\begin{itemize}
\item[(i)] 	If $\nu_{1}>0$ and $\nu_{2}=\nu_{3}=\dots=\nu_{n}=0,$ then  \eqref{claim1} holds with  $$\lambda(\alpha,\nu,n)=\frac{\alpha\nu_{1}}{2}>0.$$

We call the reader attention on the dependence on the dispersion, as well as, on the direction $x_{1}.$
	\item[(ii)] If $\nu_{1}>0$ and 
	\begin{equation*}
		0<	\sqrt{\nu_{2}^{2}+\nu_{3}^{2}+\dots+\nu_{n}^{2}}<\min\left\{ \frac{2\nu_{1}}{C\sqrt{\alpha(n-1)}},\frac{\nu_{1}(1+\alpha)}{\alpha\epsilon\sqrt{n-1}}\right\},
	\end{equation*}
with   $\epsilon$ satisfying
\begin{equation*}
	0<\epsilon<\frac{\nu_{1}}{|\overline{\nu}|\sqrt{n-1}}-\frac{\alpha\sqrt{n-1}|\overline{\nu}|}{4\nu_{1}}C^{2},
\end{equation*} 
implies that 
\begin{equation*}
	\begin{split}
\lambda&=\lambda(\alpha,\nu,n,\epsilon)\\
&=\frac{\nu_{1}(1+\alpha)}{4}-\frac{\alpha\epsilon\sqrt{n-1}|\overline{\nu}|}{4}\\
&\quad -\frac{1}{2}\sqrt{\nu_{1}^{2}\frac{(1-\alpha)^{2}}{4}+|\overline{\nu}|\nu_{1}\frac{\alpha(1-\alpha)\epsilon\sqrt{n-1}}{2}+|\overline{\nu}|^{2}\frac{\alpha^{2}(n-1)(\epsilon^{2}+C^{2})}{4}},
	\end{split}
\end{equation*}
in  \eqref{claim1}.
\end{itemize} 
 In some sense we recover the results obtained by Linares \& Ponce  in  \cite{LPZK} for the Zakharov- Kuznetsov equation in the $2d$ and $3d$ cases. More precisely, in  \cite{LPZK}    is shown    that  inequality \eqref{claim1}  holds  true whenever 
\begin{equation*}
	\sqrt{3}\nu_{1}>\sqrt{\nu_{2}^{2}+\nu_{3}^{2}+\dots+\nu_{n}^{2}}>0,\quad\nu_{1}>0,\nu_{2},\nu_{3},\dots,\nu_{n}\geq0,
\end{equation*}
for dimension $n=2$ and $n=3.$  

Nevertheless,  a quick inspection shows that  for 
 $\alpha=2$ the value $\sqrt{3}$ is not obtained directly from our   calculations. We shall point that this  particular number is quite related   to an specific  cone where the radiation  part of the solutions falls into   (see \cite{CMPS} and   also the recent  work  in \cite{RRY}, that  provides  certain numerical simulations for solutions of \eqref{zk4} where this situation is described).  %We think that this   can be reconciled if   the sharp constant   $C=C_{\mathrm{opt}}$  in  inequality \eqref{sharp} can be calculated.  

Now we show how to deal with the term that  presents the major  difficulties.
\begin{equation}\label{commutatrot1}
\begin{split}
\Theta_{1,2}(t)&=\frac{1}{2}\int_{\mathbb{R}^{n}} J^{s}u \left[\mathcal{K}_{\alpha}\partial_{x_{1}}; \varphi\right]J^{s}u\,dx.
\end{split}
\end{equation}
As is  pointed out  by Bourgain \& Li  \cite{BL}, the operator $\mathcal{K}_{\alpha}$ can be rewritten as 
\begin{equation}\label{decompo}
\begin{split}
\widehat{\mathcal{K}_{\alpha}f}(\xi)&=\left(\langle 2\pi\xi \rangle^{\alpha}-|2\pi\xi|^{\alpha}\right)\widehat{f}(\xi)\\
%&=\langle2\pi \xi\rangle^{\alpha}\left(1-\frac{|2\pi\xi|^{\alpha}}{\langle 2\pi\xi\rangle^{\alpha}}\right)\widehat{f}(\xi)\\
%&=\langle \xi\rangle^{\alpha}\left(1-\left(1-\frac{1}{\langle \xi\rangle^{2}}\right)^{\frac{\alpha}{2}}\right)\widehat{f}(\xi)\\
%&=\left(\langle \xi\rangle^{\alpha}\sum_{j=1}^{\infty} {\alpha/2 \choose j}\frac{1}{\langle \xi\rangle ^{2j}}\right)\widehat{f}(\xi)\\
%&=\langle \xi\rangle^{\alpha-2}\left(\sum_{j=1}^{\infty} {\alpha/2 \choose j}\frac{1}{\langle \xi\rangle ^{2j-2}}\right)\widehat{f}(\xi)\\
&=\langle2\pi \xi\rangle^{\alpha-2}\psi(\xi)\widehat{f}(\xi),
\end{split}
\end{equation}
where 
\begin{equation}\label{a1}
	\psi(\xi):=\sum_{j=1}^{\infty} {\alpha/2 \choose j}\langle 2\pi\xi\rangle ^{2-2j}.
\end{equation}
For $\beta>0,$ the  binomial coefficient  has the following  asymptotic equivalence 
\begin{equation}\label{asym1}
	{\beta \choose k}=\frac{(-1)^{k}}{\Gamma(-\beta)k^{1+\beta}}\left(1+o(1)\right)\quad\mbox{as}\quad k\rightarrow \infty.
\end{equation}
More precisely,
\begin{equation}\label{asym2}
	\left|	{\beta \choose k}\right|\approx\frac{1}{k^{\beta+1}}\quad \mbox{for}\quad  k\gg 1.
\end{equation}
From \eqref{asym1}-\eqref{asym2}  is clear that 
\begin{equation*}
	|\psi(\xi)|<\infty,\, \forall\, \xi \in \mathbb{R}^{n}.
\end{equation*}
The decomposition \eqref{decompo} allow us to write 
\begin{equation}\label{expre1}
\mathcal{K}_{\alpha}f(x)=\left(\mathcal{T}_{\psi}J^{\alpha-2}\right)f(x),
\end{equation}
where $\left(\mathcal{T}_{\psi}f\right)^{\widehat{}}(\xi):=\psi(\xi)\widehat{f}(\xi),\, f\in\mathcal{S}(\mathbb{R}^{n}).$

From  \eqref{expre1} is clear that
\begin{equation*}
\begin{split}
	\Theta_{1,2}(t)&=\frac{1}{2}\int_{\mathbb{R}^{n}} J^{s}u \left[\mathcal{K}_{\alpha}\partial_{x_{1}}; \varphi\right]J^{s}u\,dx\\
	&=	\frac{1}{2}\int_{\mathbb{R}^{n}} J^{s}u \mathcal{T}_{\psi}\left[J^{\alpha-2};\varphi\right]\partial_{x_{1}}J^{s}u\,dx+	\frac{1}{2}\int_{\mathbb{R}^{n}} J^{s}u\left[\mathcal{T}_{\psi}; \varphi\right]J^{\alpha-2}\partial_{x_{1}}J^{s}u\, dx\\
	&=\Theta_{1,2,1}(t)+\Theta_{1,2,2}(t).
\end{split}
\end{equation*}
After combining Plancherel's Theorem, Theorem   \ref{continuity} and \eqref{eq1}-\eqref{eq2}  
\begin{equation*}
	\begin{split}
		\int_{0}^{T}|\Theta_{1,2,1}(t)|\,dt&\lesssim T \|u\|_{L^{\infty}_{T}H^{s}_{x}}\left\|\left[J^{\alpha-2};\varphi\right]\partial_{x_{1}}J^{s}u\right\|_{L^{\infty}_{T}L^{2}_{x}}\\
		&\lesssim_{\alpha,T} \|u\|_{L^{\infty}_{T}H^{s}_{x}}^{2}.
	\end{split}
\end{equation*}
At this point   only reminds to estimate $\Theta_{1,2,2}.$
In this sense, we    consider for $b>0,$ the function 
\begin{equation*}
\mathcal{B}_{b}(y):=\frac{1}{(4\pi)^{\frac{b}{2}}\Gamma\left(\frac{b}{2}\right)}\int_{0}^{\infty} e^{-\frac{\delta}{4\pi}} e^{-\frac{\pi|y|^{2}}{\delta}}\delta^{\frac{b-n}{2}}\,\frac{d\delta}{\delta},\quad y\in\mathbb{R}^{n},\quad b>0.
\end{equation*} 
It is well  known  that for any $b>0$ the following   properties are true:
\begin{enumerate}
	\item[(i)]  $\mathcal{B}_{b}\in L^{1}(\mathbb{R}^{n})$ with $\|\mathcal{B}_{b}\|_{L^{1}}=1.$
	\item[(ii)]  $\mathcal{B}_{b}$ has Fourier transform and it is given by the formula
	\begin{equation*}
		\widehat{\mathcal{B}_{b}}(\xi)=\frac{1}{\langle 2\pi\xi\rangle^{b}}.
	\end{equation*}
\item[(iii)] $\mathcal{B}_{b}$ is smooth in $\mathbb{R}^{n}- \{0\}.$
\item[(iv)] $\mathcal{B}_{b}$  is strictly positive.
\end{enumerate}
For the proof of  these and other properties see Stein \cite{stein1}, chapter V.

The consideration of this function allow us to write 
\begin{equation*}
	\begin{split}
		&\left[\mathcal{T}_{\psi}; \varphi\right]J^{\alpha-2}\partial_{x_{1}}J^{s}u\\
		&\quad =\sum_{j=2}^{\infty}{\alpha/2 \choose j}\left(\mathcal{B}_{2j-2}*(\varphi J^{\alpha-2}\partial_{x_{1}}J^{s}u)-\varphi\mathcal{B}_{2j-2}* J^{\alpha-2}\partial_{x_{1}}J^{s}u\right).
	\end{split}
\end{equation*}
Thus, we restrict our attention to estimate the commutator in the expression above.  More precisely,
for $x\in\mathbb{R}^{n},$
\begin{equation*}
	\begin{split}
		&\left(\mathcal{B}_{2j-2}*(\varphi J^{\alpha-2}\partial_{x_{1}}J^{s}u)-\varphi\mathcal{B}_{2j-2}* J^{\alpha-2}\partial_{x_{1}}J^{s}u\right)(x)\\
		&=-\int_{\mathbb{R}^{n}}\mathcal{B}_{2j-2}(y)\left(\varphi(x-y)-\varphi(x)\right)\left(\partial_{y_{1}}J^{s+\alpha-2}u\right)(x-y)\,dy.
		\end{split}
\end{equation*}
Now, to our propose we require  make use of the following smooth partition of unity:
let $\rho\in C^{\infty}_{0}(\mathbb{R}^{n})$ such that $\rho(y)=1$ on $\{|y|\leq \frac{1}{2}\}$ and $\rho(y)=0$ on $\{|y|\geq 1\}.$ 

Setting  $\chi(y)=\rho(y/2)-\rho(y)$ then,
 for all $y\in \mathbb{R}^{n}$
\begin{equation}\label{sum1}
	1=\rho\left(\frac{y}{2}\right)+\sum_{p\geq 0}\chi_{p}(y),
\end{equation}
where $\chi_{p}(y):=\chi\left(\frac{y}{2^{p+1}}\right),y\in\mathbb{R}^{n}.$

Thus, 
\begin{equation}\label{i1}
	\begin{split}
		&\left(\mathcal{B}_{2j-2}*(\varphi J^{\alpha-2}\partial_{x_{1}}J^{s}u)-\varphi\mathcal{B}_{2j-2}* J^{\alpha-2}\partial_{x_{1}}J^{s}u\right)(x)\\
		&=-\int_{\mathbb{R}^{n}}\rho\left(\frac{y}{2}\right)\mathcal{B}_{2j-2}(y)\left(\varphi(x-y)-\varphi(x)\right)\left(\partial_{y_{1}}J^{s+\alpha-2}u\right)(x-y)\,dy\\
		&\quad -\sum_{p\geq 0}\int_{\mathbb{R}^{n}}\chi_{p}(y)\mathcal{B}_{2j-2}(y)\left(\varphi(x-y)-\varphi(x)\right)\left(\partial_{y_{1}}J^{s+\alpha-2}u\right)(x-y)\,dy\\
		&=I+II.
	\end{split}
\end{equation}
At this stage  we shall take into consideration the  different behaviors of $\mathcal{B}_{2j-2}$ depending on  $j.$ 
\begin{flushleft}
	{\sc  \fbox{Case  $j=2:$}}
\end{flushleft}
%Hereafter   we will make use of  some specific operators. We start by considering  the  \emph{Hardy-Littlewood maximal function} denoted by  $\mathcal{M}$  and defined as 
%\begin{equation}\label{maximal}
%	(\mathcal{M}f)(x):=\sup_{r>0}\frac{1}{|B(x;r)|}\int_{B(x;r)}|f(x-y)|\,dy,\quad f\in L^{1}_{\mathrm{loc}}(\mathbb{R}^{n}).
%\end{equation}
% Also, we  consider  for $0<\kappa<n,$   the  \emph{Riesz potential of order $\kappa$}, denoted as  $\mathcal{I}_{\kappa}$ and defined as  
% \begin{equation}\label{rieszpot}
% 	(\mathcal{I}_{\kappa}f)(x):=\frac{\Gamma\left(\frac{n-\kappa}{2}\right)} {\pi^{\frac{n}{2}}2^{\kappa}\Gamma\left(\frac{\kappa}{2}\right)}\int_{\mathbb{R}^{n}}\frac{f(y)}{|x-y|^{n-\kappa}}\,dy.
% \end{equation}
%It is   well known that  for $0<\kappa<n,\, 1\leq p<q<\infty,$ with $\frac{1}{q}=\frac{1}{p}-\frac{\kappa}{n},$ the \emph{Hardy-Littlewood-Sobolev } Theorem  ensures that 
% for  $p>1,$ the operator $\mathcal{I}_{\kappa}$  satisfies
% \begin{equation}\label{hls}
% 	\|\mathcal{I}_{\kappa}f\|_{L^{q}}\leq c_{p,q,n,\kappa}\|f\|_{L^{p}}.
% \end{equation}
%An interesting fact  relating the  Hardy-Littlewood maximal operator and the Riesz potentials is given in its proof. More precisely,
%\begin{equation}\label{pointwise}
%	|(\mathcal{I}_{\kappa}f)(x)|\leq c\|f\|_{L^{p}}\left((\mathcal{M}f)(x)\right)^{1-\frac{\kappa p}{n}},\quad x\in\mathbb{R}^{n}.
%\end{equation} 
%Nevertheless, to our case is required to know the dependence on  the parameters indicated.
% for more details in this subject see Stein \cite{stein1}.
 \begin{flushleft}
 	{\sc  \underline{Sub case: $n=2:$}}
 \end{flushleft}
By Lemma \ref{b1}, the function  $\mathcal{B}_{2}$ can be represented as  
\begin{equation*}
	\mathcal{B}_{2}(y)=\frac{e^{-|y|}}{2\sqrt{2}\pi}\int_{0}^{\infty}e^{-|y|s}\left(s+\frac{s^{2}}{2}\right)^{-\frac{1}{2}}\,ds.
\end{equation*}
In addition,  Lemma \ref{b1} provides the following  estimate: 
for $\beta$ multi-index with $|\beta|=1,$  
\begin{equation}\label{e1}
	\left|\partial_{y}^{\beta}\mathcal{B}_{2}(y)\right|\leq c_{\beta} e^{-|y|}\left(1+|y|^{-1}\right), \quad y\in\mathbb{R}^{2}-\{0\}.
\end{equation}
Hence
\begin{equation*}
	\begin{split}
		I&=	\lim_{\epsilon\downarrow 0}\int_{B(0;2)\setminus B(0;\epsilon)}\rho\left(\frac{y}{2}\right)\mathcal{B}_{2}(y)\left(\varphi(x-y)-\varphi(x)\right)\left(\partial_{y_{1}}J^{s+\alpha-2}u\right)(x-y)\,dy\\
		&=-\frac{1}{2}\int_{B(0;2)}(\partial_{y_{1}}\rho)\left(\frac{y}{2}\right)\mathcal{B}_{2}(y)\left(\varphi(x-y)-\varphi(x)\right)\left(J^{s+\alpha-2}u\right)(x-y)\,dy\\
		&\quad -	\lim_{\epsilon\downarrow 0}\int_{B(0;2)\setminus B(0;\epsilon)}\rho\left(\frac{y}{2}\right)(\partial_{y_{1}}\mathcal{B}_{2})(y)\left(\varphi(x-y)-\varphi(x)\right)\left(J^{s+\alpha-2}u\right)(x-y)\,dy\\
		&\quad +\int_{B(0;2)}\rho\left(\frac{y}{2}\right)\mathcal{B}_{2}(y)(\partial_{y_{1}}\varphi)(x-y)\left(J^{s+\alpha-2}u\right)(x-y)\,dy\\
		&\quad +	\lim_{\epsilon\downarrow 0}\int_{\partial B(0;\epsilon)} \vartheta_{1}\rho\left(\frac{y}{2}\right)\mathcal{B}_{2}(y)\left(\varphi(x-y)-\varphi(x)\right)\left(J^{s+\alpha-2}u\right)(x-y)\,dS_{y}\\\
		&=I_{1}+I_{2}+I_{3}+I_{4},
	\end{split}
	\end{equation*}
where $\vartheta=(\vartheta_{1},\vartheta_{2},\dots,\vartheta_{n})$  denotes the inward pointing unit  normal along  $\partial B(0;\epsilon).$

Thus, 
by  the mean value Theorem 
\begin{equation}\label{mvt}
	|\varphi(x-y)-\varphi(x)|\leq \int_{0}^{1}\left|\nabla\varphi(x+(\theta-1)y)\cdot y\right|\,d\theta \leq \left\|\nabla \varphi\right\|_{L^{\infty}}|y|,\,x,y\in\mathbb{R}^{n}.
\end{equation}
Hence combining \eqref{e1} and \eqref{mvt} we get 
\begin{equation*}
	\begin{split}
		I_{1}&=-\frac{1}{2}\int_{B(0;2)}(\partial_{y_{1}}\rho)\left(\frac{y}{2}\right)\mathcal{B}_{2}(y)\left(\varphi(x-y)-\varphi(x)\right)\left|J^{s+\alpha-2}u(x-y)\right|\,dy\\
		&\lesssim\|\nabla\varphi\|_{L^{\infty}_{x}} \int_{B(0;2)}\left|\partial_{y_{1}}\rho \left(\frac{y}{2}\right)\right|\mathcal{B}_{2}(y)\left|J^{s+\alpha-2}u(x-y)\right|\,dy,
	\end{split}
\end{equation*}
which by Young's inequality   allow us to obtain
\begin{equation*}
\int_{0}^{T}	\|I_{1}(t)\|_{L^{2}_{x}}\,dt\lesssim_{T} \|\nabla\varphi\|_{L^{\infty}_{x}}\|u\|_{L^{\infty}_{T}H^{s}_{x}}.
\end{equation*}
For the second term, that  is, $I_{2}$ we have in virtue of \eqref{e1} that 
\begin{equation*}
	\begin{split}
		|I_{2}(t)|&\lesssim  \|\nabla\varphi\|_{L^{\infty}_{x}}\lim_{\epsilon\downarrow 0}\int_{B(0;2)\setminus B(0;\epsilon)}\rho\left(\frac{y}{2}\right)e^{-|y|}\left(|y|+1\right) \left(J^{s+\alpha-2}u\right)(x-y)\,dy\\
		&\lesssim \|\nabla\varphi\|_{L^{\infty}_{x}}\int_{B(0;2)}\rho\left(\frac{y}{2}\right)e^{-|y|}\left(|y|+1\right) \left(J^{s+\alpha-2}u\right)(x-y)\,dy.
	\end{split}
\end{equation*}
Then by Young's inequality 
\begin{equation*}
\int_{0}^{T}	\|I_{2}(t)\|_{L^{2}_{x}}\,dt\lesssim T  \|\nabla\varphi\|_{L^{\infty}_{x}}\|u\|_{L^{\infty}_{T}H^{s}_{x}}.
\end{equation*}
The term $I_{3}$ is quite straightforward to handle since  a direct application of young inequality yield
\begin{equation*}
	\int_{0}^{T}\|I_{3}(t)\|_{L^{2}_{x}}\,dt \lesssim T \|\partial_{x_{1}}\varphi\|_{L^{\infty}_{x}}\|u\|_{L^{\infty}_{T}H^{s}_{x}}.
\end{equation*}
On the other hand  Lemma \ref{b2}  implies  
\begin{equation*}
	\begin{split}
		|I_{4}|&=\left|	\lim_{\epsilon\downarrow 0}\int_{\partial B(0;\epsilon)}\vartheta_{1}\rho\left(\frac{y}{2}\right)\mathcal{B}_{2}(y)\left(\varphi(x-y)-\varphi(x)\right)\left(J^{s+\alpha-2}u\right)(x-y)\,dS_{y}\right|\\
		&\lesssim \|\nabla\varphi\|_{L^{\infty}_{x}}\lim_{\epsilon\downarrow 0}\epsilon^{2} e^{-\epsilon}\left(1+\log^{+}\left(\frac{1}{\epsilon}
		\right)\right) \int_{\partial B(0;\epsilon)}\frac{\left|J^{s+\alpha-2}u(x-y)\right|}{|y|}\,dS_{y}
	\end{split}
\end{equation*}
and since 
\begin{equation*}
\lim_{\epsilon\downarrow 0}	\fint_{\partial B(x;\epsilon)}\left|J^{s+\alpha-2}u(y)\right|\,dS_{y}= \left|J^{s+\alpha-2}u(x)\right|,
\end{equation*}
whenever $x \in\mathbb{R}^{n}$ be a Lebesgue point\footnote{ If $\omega_{n}$ denotes the  volume of the unitary $n-$dimensional sphere we set  $\fint_{\partial B(x,r)}f(y)dS_{y}:=\frac{1}{\omega_{n}r^{n-1}}\int_{\partial B(x,r)}f(y) dS_{y}.$}

Therefore, $I_{4}=0.$

Next,
\begin{equation*}
	\begin{split}
		II
		&=\sum_{p\geq 0}\frac{1}{2^{p+1}}\int_{\mathbb{R}^{n}}(\partial_{y_{1}}\chi_{p})(y)\mathcal{B}_{2}(y)\left(\varphi(x-y)-\varphi(x)\right)\left(J^{s+\alpha-2}u\right)(x-y)\,dy\\
		&\quad +\sum_{p\geq 0}\int_{\mathbb{R}^{n}}\chi_{p}(y)(\partial_{y_{1}}\mathcal{B}_{2})(y)\left(\varphi(x-y)-\varphi(x)\right)\left(J^{s+\alpha-2}u\right)(x-y)\,dy\\
		&\quad -\sum_{p\geq 0}\int_{\mathbb{R}^{n}}\chi_{p}(y)\mathcal{B}_{2}(y)(\partial_{y_{1}}\varphi)(x-y)\left(J^{s+\alpha-2}u\right)(x-y)\,dy\\
		&=II_{1}+II_{2}+II_{3}.
	\end{split}
\end{equation*}
The terms   $II_{1}$ and $II_{2}$    satisfy
\begin{equation*}
		|II_{1}|\lesssim \|\varphi\|_{L^{\infty}_{x}}\sum_{p\geq 0}\frac{1}{2^{p}}\left(\mathcal{B}_{2}*\left|J^{s+\alpha-2}u\right|\right)(x),
\end{equation*}
and  
\begin{equation*}
			|	II_{2}|\lesssim\|\varphi\|_{L^{\infty}_{x}}\sum_{p\geq 0}\left(\chi_{p}e^{-|\cdot|}*\left|J^{s+\alpha-2}u\right|\right)(x),
\end{equation*}
where we have used \eqref{e1}.

On the other hand  we have after rewriting  the  following bound  for $I_{3}$
\begin{equation*}
	\begin{split}
		|	II_{3}|\leq\sum_{p\geq 0}\left(\left(\chi_{p}\mathcal{B}_{2}\right)*\left|\partial_{x_{1}}\varphi J^{s+\alpha-2}u\right|
		\right)(x).
	\end{split}
\end{equation*} 
Therefore, Young's inequality  ensure that 
\begin{equation*}
\int_{0}^{T}\max\left(	\|II_{1}(t)\|_{L^{2}_{x}},\,	\|II_{2}(t)\|_{L^{2}_{x}},\,	\|II_{3}(t)\|_{L^{2}_{x}}\right)\, dt \lesssim_{T} \|u\|_{L^{\infty}_{T}H^{s}_{x}}\|\nabla\varphi\|_{L^{\infty}_{x}}.
\end{equation*}
\begin{flushleft}
	{\sc  \underline{Sub case: $n>2:$}}
\end{flushleft}
This case    is quite similar to  the sub-case $j>1+\frac{n}{2}$ below, so for the sake of brevity we omit here and we refer to reader to the pointed out case where are indicated  all the details.
\begin{flushleft}
	{\sc  \fbox{Case  $j>2:$}}
\end{flushleft}
The following sub-cases  are examined in the case that such integer $j$ satisfies the  indicated condition. The reader shall notice that  in some cases for lower dimensions   \textit{\emph{e.g.}}  $n=2,3,4$ some cases are empty.    

We start by the easier case.
\begin{flushleft}
	{\sc  \underline{Sub case: $j>1+\frac{n}{2}:$}}
\end{flushleft}
\begin{equation*}
	\begin{split}
	I&=	\lim_{\epsilon\downarrow 0}\int_{B(0;2)\setminus B(0;\epsilon)}\rho\left(\frac{y}{2}\right)\mathcal{B}_{2j-2}(y)\left(\varphi(x-y)-\varphi(x)\right)\left(\partial_{y_{1}}J^{s+\alpha-2}u\right)(x-y)\,dy\\
	&=-	\lim_{\epsilon\downarrow 0}\int_{B(0;2)\setminus B(0;\epsilon)}\partial_{y_{1}}\left(\rho\left(\frac{y}{2}\right)\mathcal{B}_{2j-2}(y)\left(\varphi(x-y)-\varphi(x)\right)\right) \left(J^{s+\alpha-2}u\right)(x-y),dy\\
	&\quad +\nu_{1}	\lim_{\epsilon\downarrow 0}\int_{\partial B(0;\epsilon)}\rho\left(\frac{y}{2}\right)\mathcal{B}_{2j-2}(y)\left(\varphi(x-y)-\varphi(x)\right)\left(J^{s+\alpha-2}u\right)(x-y)\,dS_{y}\\
	&=-\frac{1}{2}\int_{B(0;2)}(\partial_{y_{1}}\rho)\left(\frac{y}{2}\right)\mathcal{B}_{2j-2}(y)\left(\varphi(x-y)-\varphi(x)\right)\left(J^{s+\alpha-2}u\right)(x-y)\,dy\\
	&\quad -	\lim_{\epsilon\downarrow 0}\int_{B(0;2)\setminus B(0;\epsilon)}\rho\left(\frac{y}{2}\right)(\partial_{y_{1}}\mathcal{B}_{2j-2})(y)\left(\varphi(x-y)-\varphi(x)\right)\left(J^{s+\alpha-2}u\right)(x-y)\,dy\\
	&\quad +\int_{B(0;2)}\rho\left(\frac{y}{2}\right)\mathcal{B}_{2j-2}(y)(\partial_{y_{1}}\varphi)(x-y)\left(J^{s+\alpha-2}u\right)(x-y)\,dy\\
	&\quad +	\lim_{\epsilon\downarrow 0}\int_{\partial B(0;\epsilon)}\vartheta_{1}\rho\left(\frac{y}{2}\right)\mathcal{B}_{2j-2}(y)\left(\varphi(x-y)-\varphi(x)\right)\left(J^{s+\alpha-2}u\right)(x-y)\,dS_{y}\\
	&=I_{1}+I_{2}+I_{3}+I_{4}.
	\end{split}	
\end{equation*}
%where $\nu=(\nu_{1},\nu_{2},\dots,\nu_{n})$  denotes the inward pointing normal on  $\partial B(0;\epsilon).$
The terms $I_{1}$ and $I_{3}$ can be easily  bounded by using  Young's inequality. Indeed, 
\begin{equation*}
\int_{0}^{T}	\|I_{1}(t)\|_{L^{2}_{x}}\, dt \lesssim_{T}\|\varphi\|_{L^{\infty}_{x}}\|u\|_{L^{\infty}_{T}H^{s}_{x}}.
\end{equation*} 
and
\begin{equation*}
		\int_{0}^{T}\|I_{3}(t)\|_{L^{2}_{x}}\, dt\lesssim_{T}\|\partial_{x_{1}}\varphi\|_{L^{\infty}_{x}}\|u\|_{L^{\infty}_{T}H^{s}_{x}}.
\end{equation*}
Since 
\begin{equation*}
\left|
\partial_{y_{1}}\mathcal{B}_{2j-2}(y)\right|
\lesssim \frac{\Gamma\left(j-1-\frac{n}{2}\right)}{\Gamma(j-1)},\quad \mbox{for}\quad 0< |y|\leq 2,
\end{equation*}
then 
\begin{equation*}
	\int_{0}^{T}\|I_{2}(t)\|_{L^{2}_{x}}\, dt\lesssim_{T}\frac{\Gamma\left(j-1-\frac{n}{2}\right)}{\Gamma(j-1)}\|\varphi\|_{L^{\infty}_{x}}\|u\|_{L^{\infty}_{T}H^{s}_{x}}.
\end{equation*}
	By \eqref{mvt} 
 we provide the following upper bound
\begin{equation}\label{p11}
\begin{split}
	|I_{4}|&\lesssim\left\|
\partial_{x}\varphi\right\|_{L^{\infty}_{x}}	\lim_{\epsilon\downarrow 0}\int_{\partial B(0;\epsilon)}\rho\left(\frac{y}{2}\right)\mathcal{B}_{2j-2}(y)|y|\left|J^{s+\alpha-2}u\right|(x-y)\,dS_{y}\\
&\approx\frac{\Gamma\left(j-1-\frac{n}{2}\right)}{\Gamma\left(j-1\right)} \left\|\partial_{x}\varphi\right\|_{L^{\infty}_{x}}	\lim_{\epsilon\downarrow 0}\int_{\partial B(0;\epsilon)}\rho\left(\frac{y}{2}\right)|y|\left|J^{s+\alpha-2}u\right|(x-y)\,dS_{y},
\end{split}
\end{equation}
where we have used that  
\begin{equation}\label{p12}
	\mathcal{B}_{2j-2}(y)\approx\frac{\pi^{\frac{n}{2}}\Gamma\left(j-1-\frac{n}{2}\right)}{\Gamma(j-1)},\quad   \mbox{for}\,\,  y\rightarrow 0,
\end{equation}
see  Lemma \ref{Lemmasimpt}.

Notice that 
\begin{equation}\label{p13}
	\begin{split}
		&\lim_{\epsilon\downarrow 0}\int_{\partial B(0;\epsilon)}\rho\left(\frac{y}{2}\right)|y|\left|J^{s+\alpha-2}u\right|(x-y)\,dS_{y}\\
		&\lesssim \lim_{\epsilon\downarrow 0}\epsilon^{n}\int_{\partial B(0;\epsilon)}\frac{\left|J^{s+\alpha-2}u\right|(x-y)}{|y|^{n-1}}\,dS_{y},
	\end{split}
\end{equation}
where 
\begin{equation}\label{p14}
	\lim_{\epsilon\downarrow 0}\fint_{\partial B(x;\epsilon)}\left|J^{s+\alpha-2}u(y)\right|\,dS_{y}=|J^{s+\alpha-2}u(x)|,
\end{equation}
whenever $x$ be a Lebesgue point.

Therefore,  the estimates \eqref{p11}-\eqref{p14} imply  $I_{3}=0.$

An argument quite similar to the one used above also applies to  prove that $I_{4}=0,$ and  to avoid repeating  the same arguments  we will omit  the details.

On the other hand 
\begin{equation*}
	\begin{split}
	II
	&=\sum_{p\geq 0}\frac{1}{2^{p+1}}\int_{\mathbb{R}^{n}}(\partial_{y_{1}}\chi_{p})(y)\mathcal{B}_{2j-2}(y)\left(\varphi(x-y)-\varphi(x)\right)\left(J^{s+\alpha-2}u\right)(x-y)\,dy\\
	&\quad +\sum_{p\geq 0}\int_{\mathbb{R}^{n}}\chi_{p}(y)(\partial_{y_{1}}\mathcal{B}_{2j-2})(y)\left(\varphi(x-y)-\varphi(x)\right)\left(J^{s+\alpha-2}u\right)(x-y)\,dy\\
	&\quad -\sum_{p\geq 0}\int_{\mathbb{R}^{n}}\chi_{p}(y)\mathcal{B}_{2j-2}(y)(\partial_{y_{1}}\varphi)(x-y)\left(J^{s+\alpha-2}u\right)(x-y)\,dy\\
	&=II_{1}+II_{2}+II_{3}.
	\end{split}
\end{equation*}
The  expression $II_{1}$  clear that it  satisfies the inequality
\begin{equation}\label{p1}
	\begin{split}
	|II_{1}|\lesssim \|\varphi\|_{L^{\infty}_{x}}\sum_{p\geq 0}\frac{1}{2^{p}}\left(\mathcal{B}_{2j-2}*\left|J^{s+\alpha-2}u\right|
	\right)(x).
\end{split}
\end{equation}
Since $\mathcal{B}_{2j-2}$ is smooth away from zero, the  following identity holds: 
for $y\in \mathbb{R}^{n}-\{0\},$
\begin{equation*}
	(\partial_{y_{1}}\mathcal{B}_{2j-2})(y)=-\frac{y_{1}}{2(j-2)}\mathcal{B}_{2j-4}(y),\quad  \mbox{whenever}\quad j>2.
\end{equation*}
If we set 
\begin{equation*}
	\widetilde{\mathcal{B}_{j}}(y):=y_{1}\mathcal{B}_{j-2}(y), \quad \mbox{for}\quad  j>2,
\end{equation*}
then 
\begin{equation}\label{p2}
	\begin{split}
	|	II_{2}|\lesssim\frac{\|\varphi\|_{L^{\infty}_{x}}}{j-2}\sum_{p\geq 0}\left(\chi_{p}\widetilde{\mathcal{B}_{2j}}*\left|J^{s+\alpha-2}u\right|\right)(x),
	\end{split}
\end{equation}
 and 
\begin{equation}\label{p3}
	\begin{split}
	|	II_{3}|\leq\sum_{p\geq 0}\left(\left(\chi_{p}\mathcal{B}_{2j-2}\right)*\left|\partial_{x_{1}}\varphi J^{s+\alpha-2}u\right|\right)(x), \quad \forall x\in\mathbb{R}^{n}.
	\end{split}
\end{equation}
Finally, gathering \eqref{p1},\eqref{p2} and \eqref{p3} and taking into consideration \eqref{sum1}
\begin{equation*}
\begin{split}
\|	II\|_{L^{2}_{x}}&\lesssim  \|\varphi\|_{L^{\infty}_{x}}\left\|\mathcal{B}_{2j-2}*\left|J^{s+\alpha-2}u\right|\right\|_{L^{2}_{x}}+ \frac{\|\varphi\|_{L^{\infty}_{x}}}{2(j-2)} \left\|\widetilde{\mathcal{B}_{2j}}*\,\left|J^{s+\alpha-2}u\right| \,\right\|_{L^{2}_{x}}\\
&\quad +\left\|\mathcal{B}_{2j-2}*\left|\partial_{x_{1}}\varphi J^{s+\alpha-2}u\right|\right\|_{L^{2}_{x}},
\end{split}	
\end{equation*}
which  by Young's inequality  allow us to obtain the bound
\begin{equation*}
	\begin{split}
	\int_{0}^{T}	\|	II(t)\|_{L^{2}_{x}}\, dt&\lesssim\int_{0}^{T}\left\{ \|\varphi\|_{L^{\infty}_{x}}\|J^{s}u(t)\|_{L^{2}_{x}}+\frac{\|\nabla_{x}\varphi\|_{L^{\infty}_{x}}}{2(j-2)} \left\|\widetilde{\mathcal{B}_{2j}}\right\|_{L^{1}}\|J^{s}u(t)\|_{L^{2}_{x}}\right\}\, dt\\
		&\quad +T\left\|\partial_{x_{1}}\varphi\right\|_{L^{\infty}_{x}}\|u\|_{L^{\infty}_{T}H^{s}_{x}}\\
		&\lesssim_{n,T}\|u\|_{L^{\infty}_{T}H^{s}_{x}}\|\nabla\varphi\|_{W^{1,\infty}_{x}},
	\end{split}
\end{equation*}
where we have used that 
 \begin{equation*}
\left\|\widetilde{\mathcal{B}_{2j}}\right\|_{L^{1}}\approx_{n} \frac{(j-1)(j-2)}{j^{\frac{3}{2}}},\quad \mbox{whenever 
}\, j>2.
\end{equation*}
\begin{flushleft}
	{\sc  \underline{Sub case: $j=1+\frac{n}{2}:$}}
\end{flushleft}
This case  is quite similar to the  sub case $n=2,$ so that  for the sake of brevity we omit the details.
\begin{flushleft}
	{\sc  \underline{Sub case: $j<1+\frac{n}{2}:$}}
\end{flushleft}
We start by specifying the implicit constant  in inequality (c) in  Lemma \ref{b2}. More precisely, for any multi-index $\beta\in(\mathbb{N}_{0})^{n}$ with $|\beta|=1,$ the    following   identity holds   for $y\in\mathbb{R}^{n}\setminus \{0\},$
\begin{equation}\label{ineq1.1}
	\begin{split}
		\left(\partial_{y}^{\beta}\mathcal{B}_{2j-2}\right)(y)&=-\frac{y^{\beta}}{|y|}\left(\mathcal{B}_{2j-2}(y)+\gamma_{j}e^{-|y|}\int_{0}^{\infty}e^{-s|y|}\, s\left(s+\frac{s^{2}}{2}\right)^{\frac{n-(2j-2)-1}{2}}\,ds\right)
	\end{split}
\end{equation}
where 
\begin{equation*}
	\gamma_{j}:=\frac{1}{(2\pi)^{\frac{n-1}{2}}2^{j-1}\Gamma(j-1)\Gamma\left(\frac{n-2j+3}{2}\right)}.
\end{equation*}
Therefore, a rough upper bound for \eqref{ineq1.1} is 
\begin{equation}\label{bound1}
	\begin{split}
		&\left| \left(\partial_{y}^{\beta}\mathcal{B}_{2j-2}\right)(y)\right|\\
		%&\leq \mathcal{B}_{2j-2}(y)+\gamma_{j}\int_{0}^{1}e^{-s|y|}\, s\left(s+\frac{s^{2}}{2}\right)^{\frac{n-2j+1}{2}}\,ds\\
		%&\quad +\gamma_{j}\int_{1}^{\infty}e^{-s|y|}\, s\left(s+\frac{s^{2}}{2}\right)^{\frac{n-2j+1}{2}}\,ds\\
		&\leq  \mathcal{B}_{2j-2}(y)+e^{-|y|}\gamma_{j}\left(\frac{3}{2}\right)^{\frac{n+1}{2}}\left(\left(\frac{2}{3}\right)^{j}+2^{j}\frac{\Gamma(n-2j+3)}{|y|^{n-2j+3}}\right),\,  y\in\mathbb{R}^{n}\setminus \{0\}.
	\end{split}
\end{equation}
Now, we shall  pay attention to the terms  in the  r.h.s above. By Legendre's\footnote{ For $z\in\mathbb{C}$  the \emph{Legendre's duplication Gamma formula} is given by 
$$	\sqrt{\pi}\Gamma(2z)=2^{2z-1}\Gamma(z)\Gamma\left(z+\frac{1}{2}\right).$$}
	duplication formula for Gamma function  we obtain 
\begin{equation*}
	\sup_{j\gg 1}\left\{\gamma_{j}\left(\frac{2}{3}\right)^{j}, \gamma_{j}2^{j}\Gamma(n-2j+3)\right\}\lesssim_{n} 1.
\end{equation*}
Next, we turn our attention to $I,$ 
\begin{equation*}
	\begin{split}
		I&=	\lim_{\epsilon\downarrow 0}\int_{B(0;2)\setminus B(0;\epsilon)}\rho\left(\frac{y}{2}\right)\mathcal{B}_{2j-2}(y)\left(\varphi(x-y)-\varphi(x)\right)\left(\partial_{y_{1}}J^{s+\alpha-2}u\right)(x-y)\,dy\\
		&=-\frac{1}{2}\int_{B(0;2)}(\partial_{y_{1}}\rho)\left(\frac{y}{2}\right)\mathcal{B}_{2j-2}(y)\left(\varphi(x-y)-\varphi(x)\right)\left(J^{s+\alpha-2}u\right)(x-y)\,dy\\
		&\quad -	\lim_{\epsilon\downarrow 0}\int_{B(0;2)\setminus B(0;\epsilon)}\rho\left(\frac{y}{2}\right)(\partial_{y_{1}}\mathcal{B}_{2j-2})(y)\left(\varphi(x-y)-\varphi(x)\right)\left(J^{s+\alpha-2}u\right)(x-y)\,dy\\
		&\quad +\int_{B(0;2)}\rho\left(\frac{y}{2}\right)\mathcal{B}_{2j-2}(y)(\partial_{y_{1}}\varphi)(x-y)\left(J^{s+\alpha-2}u\right)(x-y)\,dy\\
		&\quad +	\lim_{\epsilon\downarrow 0}\int_{\partial B(0;\epsilon)}\vartheta_{1} \rho\left(\frac{y}{2}\right)\mathcal{B}_{2j-2}(y)\left(\varphi(x-y)-\varphi(x)\right)\left(J^{s+\alpha-2}u\right)(x-y)\,dS_{y}\\
		&=I_{1}+I_{2}+I_{3}+I_{4}.
	\end{split}
\end{equation*}
%where $\nu=(\nu_{1},\nu_{2},\dots,\nu_{n})$  denotes the inward pointing normal on  $\partial B(0;\epsilon).$
Nevertheless, we realize that after incorporating the  bound obtained in  \eqref{bound1}  combined with the arguments already    described in the case $n=2$   implies that  
$I_{3}=0,$
\begin{equation*}
\int_{0}^{T}	\max\left(	\|I_{1}(t)\|_{L^{2}_{x}},\,	\|I_{2}(t)\|_{L^{2}_{x}},\,	\|I_{4}(t)\|_{L^{2}_{x}}\right)\, dt \lesssim_{T} \|u\|_{L^{\infty}_{T}H^{s}_{x}}\|\varphi\|_{W^{1,\infty}_{x}}.
\end{equation*}
and 
\begin{equation*}
\int_{0}^{T}	\max\left(	\|II_{1}(t)\|_{L^{2}_{x}},\,	\|II_{2}(t)\|_{L^{2}_{x}},\,	\|II_{3}(t)\|_{L^{2}_{x}}\right)\, dt\lesssim_{T} \|u\|_{L^{\infty}_{T}H^{s}_{x}}\|\varphi\|_{W^{1,\infty}_{x}}.
\end{equation*}
%Since
%\begin{equation*}
%	\Gamma\left(j-1-\frac{n}{2}\right)\sim_{n}\frac{\Gamma(j-1)}{(j-1)^{\frac{n}{2}}}\quad \mbox{for}\quad j\gg 1,
%\end{equation*}
In summary, 
  we have obtained  
\begin{equation*}
	\begin{split}
			\int_{0}^{T}	|\Theta_{1,2}(t)|\, dt&\lesssim_{\alpha,n,T,\nu}\|u\|_{L^{\infty}_{T}H^{s}_{x}}\|\varphi\|_{W^{1,\infty}_{x}}\left(\sum_{j=2}^{\infty}\frac{1}{j^{\alpha+1}}\right)\\
			&\lesssim_{\alpha} \|u\|_{L^{\infty}_{T}H^{s}_{x}}\|\varphi\|_{W^{1,\infty}_{x}},
	\end{split}
\end{equation*}
for any $\alpha\in (0,2).$
\begin{rem}
\textit{Notice that  the last inequality above is quite  instructive to understand the effects of  the dispersion  on the solutions, since   the argument above   suggest  that  for $\alpha\leq 0,$ the pursuit smoothing effect  does not hold.}
\end{rem}
%	\begin{flushleft}
%	{\sc  \underline{Term $\Theta_{2}:$}}
%\end{flushleft}
Finally, by inequality \eqref{KPDESI}
\begin{equation*}
\begin{split}
\Theta_{2}(t)&=\int_{\mathbb{R}^{n}} J^{s}\left(u\partial_{x_{1}}u\right) J^{s}u \varphi\,dx\\
%&=\int_{\mathbb{R}^{n}} \left(\left[J^{s}; u\right]\partial_{x_{1}}u +u\partial_{x_{1}}J^{s}u\right) J^{s}u \varphi\,dx\\
&=\int_{\mathbb{R}^{n}} J^{s}u \varphi \left[J^{s}; u\right]\partial_{x_{1}}u\,dx-\frac{1}{2}\int_{\mathbb{R}^{n}}\left(J^{s}u\right)^{2}\left(\partial_{x_{1}}u\varphi+u\partial_{x_{1}}\varphi\right)\,dx\\
&\lesssim \|\varphi\|_{L^{\infty}_{x}}
\left\|J^{s}u(t)\right\|_{L^{2}_{x}}^{2}\left\|\nabla u(t)\right\|_{L^{\infty}_{x}}+\left\|J^{s}u(t)\right\|_{L^{2}_{x}}^{2}\left(\|\partial_{x_{1}} u(t)\|_{L^{\infty}_{x}}+\|\partial_{x_{1}}\varphi\|_{L^{\infty}_{x}}\|u(t)\|_{L^{\infty}_{x}}\right).
\end{split}
\end{equation*}
Finally, gathering the estimates above  we obtain after integrating in the  time  variable  the expression \eqref{energy} 
\begin{equation*}
\begin{split}
&\int_{0}^{T}\int_{\mathbb{R}^{n}}\left(\left(J^{s+\frac{\alpha}{2}}u(x,t)\right)^{2}+\left(\partial_{x_{1}}J^{s+\frac{\alpha-2}{2}}u(x,t)\right)^{2}\right)\partial_{x_{1}}\varphi\,dx\,dt\\
&\lesssim_{n,\alpha,\nu,T} \left(1+T+\|\nabla u\|_{L^{1}_{T}L^{\infty}_{x}}+T\|u\|_{L^{\infty}_{T}H^{r}_{x}}\right)\|u\|_{L^{\infty}_{T}H^{s}_{x}}^{2},
\end{split}
\end{equation*}
	whenever $r>\frac{n}{2}.$
\end{proof}
\section{Proof of Theorem \ref{zk9}}\label{seccion5}
In this section we  focus our attention in to provide some immediate applications of the smoothing effect   deduced  in the previous section. in this sense, we prove that  solution of the IVP \ref{zk4} satisfy the principle of propagation of regularity in dispersive equations.
More precisely, we prove Theorem \ref{zk9}  and its method of proof  follows the ideas from \cite{ILP1}, \cite{KLPV},\cite{AM1}, and   \cite{AMZK}.

For the proof we  consider the following standard notation, that it has shown to be versatile. In detail, 
 for $\epsilon>0$ and $\tau\geq 5\epsilon$ we consider  the following  families  of functions
\begin{equation*}
	\chi_{\epsilon, \tau},\widetilde{\phi_{\epsilon,\tau}}, \phi_{\epsilon,\tau},\psi_{\epsilon}\in C^{\infty}(\mathbb{R}),
\end{equation*} 
satisfying the   conditions indicated below:
\begin{itemize}
	\item[(i)] $
	\chi_{\epsilon, \tau}(x)=
	\begin{cases} 
		1 & x\geq \tau \\
		0 & x\leq \epsilon 
	\end{cases},
	$
	\item[(ii)] $\supp(\chi_{\epsilon, \tau}')\subset[\epsilon,\tau],$ 
	\item[(iii)] $\chi_{\epsilon, \tau}'(x)\geq 0,$
	\item[(iv)] $
	\chi_{\epsilon, \tau}'(x)\geq\frac{1}{10(\tau-\epsilon)}\mathbb{1}_{[2\epsilon,\tau-2\epsilon]}(x),$
	\item[(v)] $\supp\left(\widetilde{\phi_{\epsilon,\tau}}\right),\supp(\phi_{\epsilon,\tau})\subset \left[\frac{\epsilon}{4},\tau\right],$
	\item[(vi)] $\phi_{\epsilon,\tau}(x)=\widetilde{\phi_{\epsilon,\tau}}(x)=1,\,\mbox{if}\quad  x\in \left[\frac{\epsilon}{2},\epsilon\right],$
	\item[(vii)] $\supp(\psi_{\epsilon})\subset \left(-\infty,\frac{\epsilon}{4}\right].$
	\item[(viii)] For all $x\in\mathbb{R}$ the following quadratic partition of the unity holds
	$$
	\chi_{\epsilon, \tau}^{2}(x)+\widetilde{\phi_{\epsilon,\tau}}^{2}(x)+\psi_{\epsilon}(x)=1,$$
	\item[(ix)] also for all $x\in\mathbb{R}$
	$$
	\chi_{\epsilon, \tau}(x)+\phi_{\epsilon,\tau}(x)+\psi_{\epsilon}(x)=1,\quad x\in\mathbb{R}.
	$$
\end{itemize}
For a more detailed construction of these families of weighted  functions see  \cite{ILP1}.
\begin{proof}
In order to not  saturate the notation the  weighted functions $\chi_{\epsilon, \tau},\phi_{\epsilon, \tau},\widetilde{\phi_{\epsilon,\tau}}$ and $\psi_{\epsilon}$
  will be considered as functions    of the variable $\nu\cdot x+\omega t$ \textit{e.g.}  in the case of $\chi_{\epsilon, \tau}$  we  understand it as 
  \begin{equation}
  	\chi_{\epsilon, \tau}(x,t):=\chi_{\epsilon, \tau}(\nu\cdot x+\omega t),
  \end{equation}
and the dependence on $x,t$ will be suppressed.

	Performing energy estimates  allow us to obtain 
		\begin{equation}\label{energy1}
			\begin{split}
				&\frac{d}{dt}\int_{\mathbb{R}^{n}}\left(J^{s}u\right)^{2}\chi_{\epsilon,\tau}^{2}\,dx\underbrace{-\frac{\omega }{2}\int_{\mathbb{R}^{n}}\left(J^{s}u\right)^{2}\chi_{\epsilon,\tau}\chi_{\epsilon,\tau}'\,dx}_{\Theta_{1}(t)}\underbrace{-\int_{\mathbb{R}^{n}}J^{s}u \partial_{x_{1}}(-\Delta)^{\frac{\alpha}{2}}J^{s}u \chi_{\epsilon,\tau}^{2}\,dx}_{\Theta_{2}(t)}\\
				&\underbrace{+\int_{\mathbb{R}^{n}}J^{s}uJ^{s}(u\partial_{x_{1}}u)\chi^{2}_{\epsilon,\tau}\,dx}_{\Theta_{3}(t)}=0.
			\end{split}
		\end{equation}
	The proof  follows by using an inductive argument we describe below:
	\begin{figure}[h]\label{figure1.1}
		\begin{tikzcd}[column sep=large]
			J^{s}\arrow{r}\arrow{d}
			&J^{s+\frac{2-\alpha}{2}}\arrow{r}{\text{P.R}}\arrow{d}
			&J^{s+1}\arrow{d}\\
			J^{s+\frac{\alpha}{2}}&J^{s+1}&J^{s+1+\frac{\alpha}{2}}
		\end{tikzcd}
		\caption{Description of the   argument  to reach more regularity  at steps of order $\alpha/2$ in the inductive process. 
			The abbreviations P.R stands for \emph{Propagated Regularity} and the down arrows indicate local regularity gained. The downward arrows show the local gain of regularity  at the corresponding  step}
	\end{figure}
the idea is to start   by showing that regularity of the solutions  close to Sobolev index where there exist local well-posedness,  is propagated with infinity speed over the  moving half spaces. In this process, the smoothing effect is fundamental  to estimate several terms when we  perform the  energy estimate \eqref{energy1}  \textit{e.g.} the  to show that  $\Theta_{1}\in L^{1}_{T}$ we requires the extra regularity provided by the smoothing.  Since the smoothness provided is too weak  we require two steps (compare to ZK equation see \cite{LPZK}) and \cite{AMZK})  to reach  one extra derivative. The  figure \ref{figure1.1} describes the two steps inductive  process.  
		\begin{flushleft}
		{\sc	Case : $s\in (s_{n},s_{n}+\frac{\alpha}{2})$}
	\end{flushleft}
\begin{flushleft}
	{\sc \underline{Step 1} }
\end{flushleft}
%		\begin{flushleft}
%		{\sc  \underline{Term $\Theta_{1}:$}}
%	\end{flushleft}
Notice that  
\begin{equation}\label{teta1}
	\begin{split}
		\int_{0}^{T}|\Theta_{1}(t)|\, dt&\lesssim_{\nu}\int_{0}^{T}\int_{\mathbb{R}^{n}}(J^{s}u)^{2}\chi_{\epsilon, \tau}\chi_{\epsilon, \tau}'\, dx\, dt\\
		&\lesssim \int_{0}^{T}\int_{\mathbb{R}^{n}}\mathbb{1}_{\mathcal{H}_{\{\epsilon-\omega t,\nu\}}\cap\mathcal{H}_{\{\tau-\omega t,\nu\}}^{c}}(J^{s}u)^{2}\,  dx\, dt.
	\end{split}
\end{equation}
Although,   by Lemma \ref{main2}  the solutions enjoys of extra regularity on the channel  
$$\mathcal{H}_{\{\epsilon-\omega t,\nu\}}\cap\mathcal{H}_{\{\tau-\omega t,\nu\}}^{c},\quad \mbox{for}\, \epsilon>0,\, \tau>5\epsilon,$$
More precisely, we can choose $\varphi$ in Lemma \ref{main2} properly  to obtain  
\begin{equation}
	\int_{0}^{T}\int_{\mathcal{H}_{\{\epsilon-\omega t,\nu\}}\cap\mathcal{H}_{\{\tau-\omega t,\nu\}}^{c}}\left(J^{s+\frac{\alpha}{2}}u\right)^{2}\,dx\, dt\lesssim c.
\end{equation}
In virtue of Lemma \ref{zk37}   \emph{mutatis mutandis} in the weighted functions  we get 
\begin{equation}\label{smoot}
	\int_{0}^{T}\int_{\mathcal{H}_{\{\epsilon-\omega t,\nu\}}\cap\mathcal{H}_{\{\tau-\omega t,\nu\}}^{c}}\left(J^{r}u\right)^{2}\,dx\, dt\lesssim c,
\end{equation}
for $r\in(0,s_{n}+\frac{\alpha}{2}),$  for $\epsilon>0$ and $\tau\geq 5\epsilon.$

Then $\Theta_{1}\in L^{1}_{T}.$
%	\begin{flushleft}
%	{\sc  \underline{Term $\Theta_{2}:$}}
%\end{flushleft}
The arguments in \eqref{decomposition}-\eqref{bo3} allow us to rewrite $\Theta_{2}$ as follows:
		 \begin{equation}\label{comm1.1}
		 	\begin{split}
		 		\Theta_{2}(t)&=\frac{1}{2}\int_{\mathbb{R}^{n}} J^{s}u \left[J^{\alpha}\partial_{x_{1}}; \chi_{\epsilon,\tau}^{2}\right]J^{s}u\,dx+\frac{1}{2}\int_{\mathbb{R}^{n}} J^{s}u \left[\mathcal{K}_{\alpha}\partial_{x_{1}}; \chi_{\epsilon,\tau}^{2}\right]J^{s}u\,dx\\
		 				&=\Theta_{2,1}(t)+\Theta_{2,2}(t).
		 			\end{split}
		 \end{equation}
Next, we  consider the operator
\begin{equation}\label{comm1.2}
{c_{\alpha}}(x,D):=\left[J^{\alpha}\partial_{x_{1}}; \chi_{\epsilon,\tau}^{2}\right].
\end{equation}
Thus, by using pseudo-differential calculus there exist operators $p_{\alpha-k}(x,D),$ \linebreak $\,j\in \{1,2\cdots,m\}$  for some  $m\in \mathbb{N}$  such that 
	 \begin{equation}\label{comm1.3}
	 	\begin{split}
	 			c_{\alpha}(x,D)=p_{\alpha}(x,D)+p_{\alpha-1}(x,D)+\dots+p_{\alpha-m}(x,D)+r_{\alpha-m-1}(x,D),
	 	\end{split}
	 \end{equation}
 where $p_{\alpha-j}\in \mathrm{OP}\mathbb{S}^{\alpha-j}$ and $r_{\alpha-m-1}\in  \mathrm{OP}\mathbb{S}^{\alpha-1-m}.$ %for $j\in\{1,2,\dots, m\}.$ 
 
The representation above presents two main difficulties. The first one  consist into describe   accurately  the  terms $p_{\alpha-j}(x,D)$ for each $j.$ The second problem  deals into  determine $m$  adequately. 

We will show  later that it is  only required to estimate $p_{\alpha}(x,D)$ and $p_{\alpha-1}(x,D).$  According to \eqref{comono1}-\eqref{energy2.1.1} 
\begin{equation}\label{disperive}
		p_{\alpha}(x,D)=\partial_{x_{1}}(\chi_{\epsilon, \tau}^{2})J^{\alpha}-\alpha\partial_{x_{1}}(\chi_{\epsilon, \tau}^{2}) J^{\alpha-2}\partial_{x_{1}}^{2}-\alpha\sum_{\mathclap{\substack{|\beta|=1\\\beta\neq \mathrm{e}_{1}}}}\partial_{x}^{\beta}(\chi_{\epsilon, \tau}^{2})J^{\alpha-2}\partial_{x}^{\beta}\partial_{x_{1}},
\end{equation}
and 
\begin{equation*}
	\begin{split}
		&p_{\alpha-1}(x,D)
		=-\alpha\sum_{|\beta|=1}\partial_{x_{1}}\partial_{x}^{\beta}(\chi_{\epsilon, \tau}^{2})\partial_{x_{1}}J^{\alpha-2}-\frac{\alpha}{2\pi}\sum_{|\beta|=1}\partial_{x}^{\beta}\partial_{x_{1}}(\chi_{\epsilon, \tau}^{2})\partial_{x}^{\beta}J^{\alpha-2}\\
		&-\frac{\alpha}{2\pi}\sum_{|\beta|=1}\partial_{x}^{2\beta}(\chi_{\epsilon, \tau}^{2})\partial_{x_{1}}J^{\alpha-2}
		+\frac{\alpha(\alpha-2)}{2\pi}\sum_{|\beta_{2}|=1}\sum_{|\beta_{1}|=1}\partial_{x}^{\beta_{2}}\partial_{x}^{\beta_{1}}(\chi_{\epsilon, \tau}^{2})\partial_{x_{1}}\partial_{x}^{\beta_{2}}\partial_{x}^{\beta_{1}}J^{\alpha-4}.
	\end{split}
\end{equation*}
%and
%\begin{equation*}
%	\begin{split}
%		p_{\alpha-1}(x,\xi)
%		&=-\alpha\sum_{|\beta|=1}(2\pi \mathrm{i}\xi_{1})\langle 2\pi\xi\rangle^{\alpha-2}\partial_{x_{1}}\partial_{x}^{\beta}\varphi-\frac{\alpha}{2\pi}\sum_{|\beta|=1}(2\pi \mathrm{i}\xi)^{\beta}\langle 2\pi \xi\rangle ^{\alpha-2}\partial_{x}^{\beta}\partial_{x_{1}}\varphi\\
%		&\quad  -\frac{\alpha}{2\pi}\sum_{|\beta|=1}(2\pi\mathrm{i}\xi_{1})\langle2\pi \xi \rangle ^{\alpha-2}\partial_{x}^{2\beta}\varphi\\
%		&\quad +\frac{\alpha(\alpha-2)}{2\pi}\sum_{|\beta_{2}|=1}\sum_{|\beta_{1}|=1}(2\pi\mathrm{i}\xi_{1})(2\pi\mathrm{i}\xi)^{\beta_{1}}(2\pi\mathrm{i}\xi)^{\beta_{2}}\langle 2\pi\xi \rangle^{\alpha-4}\partial_{x}^{\beta_{2}}\partial_{x}^{\beta_{1}}
%		\varphi.
%	\end{split}
%\end{equation*}
The remainder terms  are obtained by using  pseudo-differential calculus and   these ones  can be rewritten  as 
\begin{equation}\label{decomp1.2}
	\begin{split}
		p_{\alpha-j}(x,D)=\sum_{|\beta|=j}c_{\beta,j}\partial_{x}^{\beta}(\chi_{\epsilon,\tau}^{2})\Psi_{\beta,j}J^{\alpha-j},\quad j\geq 2,
	\end{split}
\end{equation}
where $\Psi_{\beta,j}\in \mathrm{OP}\mathbb{S}^{0}$ for $j\in \{2,\dots,m\}.$ 

Setting $m$  as being
%\begin{equation}\label{remaindercontrol}
%	2s+\alpha-1-s_{n}<m\leq 2s+\alpha-s_{n},
%\end{equation}
%that is,
\begin{equation*}
	m=\lceil2s+\alpha-1-s_{n}\rceil.
\end{equation*}
Thus,
\begin{equation*}
	\begin{split}
		\int_{0}^{T}|\Theta_{2,1,m+1}(t)|\, dt&
		\leq T\|u_{0}\|_{L^{2}_{x}}\left\|J^{s}r_{\alpha-m-1}(x,D)J^{s}u\right\|_{L^{\infty}_{T}L^{2}_{x}}\\
		&\lesssim_{\epsilon,\tau,\alpha,n}T\|u_{0}\|_{L^{2}_{x}}\|u\|_{L^{\infty}_{T}H^{s_{n}+}_{x}}.
	\end{split}
\end{equation*}
When replacing \eqref{disperive} into  $\Theta_{2,1}$ we obtain 
	\begin{equation}\label{disppart}
	\begin{split}
		\Theta_{2,1}(t)		&=\frac{1}{2}\int_{\mathbb{R}^{n}}J^{s}uJ^{s+\alpha}u\partial_{x_{1}}(\chi_{\epsilon,\tau}^{2})\,dx-\frac{\alpha}{2}\int_{\mathbb{R}^{n}}J^{s}u J^{s+\alpha-2}\partial_{x_{1}}^{2}u\partial_{x_{1}}(\chi_{\epsilon,\tau}^{2})\,dx\\
		&\quad -\frac{\alpha}{2}\sum_{\mathclap{\substack{|\beta|=1\\\beta\neq \mathrm{e}_{1}}}}\,\int_{\mathbb{R}^{n}}J^{s}uJ^{s+\alpha-2}\partial_{x_{1}}\partial_{x}^{\beta}u \partial_{x}^{\beta}(\chi_{\epsilon,\tau}^{2})\,dx+\frac{1}{2}\sum_{j=2}^{m}\int_{\mathbb{R}^{n}}J^{s}up_{\alpha-j}(x,D)J^{s}u\,dx\\
		&\quad +\frac{1}{2}\int_{\mathbb{R}^{n}}J^{s}ur_{\alpha-m-1}(x,D)J^{s}u\, dx\\
		&=\Theta_{2,1,1}(t)+\Theta_{2,1,2}(t)+\Theta_{2,1,3}(t)+\sum_{j=2}^{m}\Theta_{2,1,j+2}(t)\\
		&\quad +\frac{1}{2}\int_{\mathbb{R}^{n}}J^{s}ur_{\alpha-m-1}(x,D)J^{s}u\, dx.
	\end{split}
\end{equation}
%	\begin{flushleft}
%	{\sc  \underline{Term $\Theta_{2,1,1}:$}}
%\end{flushleft}
By using an argument similar to the one described in \eqref{r1}-\eqref{eq2} there exists $r_{\frac{\alpha}{2}-2}(x,D)\in\mathrm{OP}\mathbb{S}^{\frac{\alpha}{2}-2},$ such that
\begin{equation}\label{primersmoo}
	\begin{split}
	\Theta_{2,1,1}(t)&=	\nu_{1}\int_{\mathbb{R}^{n}}\left(J^{s+\frac{\alpha}{2}}u\right)^{2}\chi_{\epsilon,\tau}\chi_{\epsilon,\tau}'\,dx+\frac{1}{2}\int_{\mathbb{R}^{n}}J^{s+\frac{\alpha}{2}}u\left[J^{\frac{\alpha}{2}};\partial_{x_{1}}(\chi_{\epsilon,\tau}^{2})\right]J^{s}u\,dx\\
	&=\Theta_{2,1,1,1}(t)+\Theta_{2,1,1,2}(t).
	\end{split}
\end{equation}
The term containing the  commutator expression   is  quite more complicated to handle since  at a first sight   some  upper bound  would require more regularity. However, we will show that this is not the case,  since there are several cancellations  that allow close the argument without any  additional assumption.

First,  we rewrite $\Theta_{2,1,1,2}$ as follows
\begin{equation}\label{commudecomp1..1}
	\begin{split}
		\Theta_{2,1,1,2}(t)&=\frac{1}{2}\int_{\mathbb{R}^{n}}J^{s}u\left[J^{\alpha}; \partial_{x_{1}}\left(\chi_{\epsilon, \tau}^{2}\right)\right]J^{s}u\, dx\\
		&\quad -\frac{1}{2}\int_{\mathbb{R}^{n}}J^{s}u\left[ J^{\frac{\alpha}{2}}; \partial_{x_{1}}\left(\chi_{\epsilon, \tau}^{2}\right)\right] J^{s+\frac{\alpha}{2}}u\, dx\\
		&=\Lambda_{1}(t)+\Lambda_{2}(t).
	\end{split}
\end{equation} 
We focus our attention on $\Lambda_{1}$.  In the same spirit  of the decomposition used in \eqref{decomp1.2},   is clear that  for some $m_{1}\in\mathbb{N}$ there exist operators\\ $q_{\alpha-1}(x,D),q_{\alpha-2}(x,D), \dots, q_{\alpha-m_{1}}(x,D),r_{\alpha-m_{1}-1}(x,D)$  such that 
\begin{equation}\label{commudecomp1..2}
\left[J^{\alpha}; \partial_{x_{1}}\left(\chi_{\epsilon, \tau}^{2}\right)\right]=q_{\alpha-1}(x,D)+q_{\alpha-2}(x,D)+\dots+q_{\alpha-m_{1}}(x,D)+ r_{\alpha-m_{1}-1}(x,D),
\end{equation}
where
\begin{equation}\label{commudecomp1..3}
		q_{\alpha-j}(x,D)=\sum_{|\beta|=j}c_{\beta,j}\partial_{x}^{\beta}\partial_{x_{1}}(\chi_{\epsilon,\tau}^{2})\Psi_{\beta,j}J^{\alpha-j},\quad j\geq 1,\quad r_{\alpha-m_{1}-1}\in \mathrm{OP}\mathbb{S}^{\alpha-m_{1}-1}.
\end{equation}
and $\Psi_{\beta,j}\in\mathrm{OP}\mathbb{S}^{0}$ for all $\beta.$

Hence,
\begin{equation*}
	\begin{split}
		\Lambda_{1}(t)&=\frac{1}{2}\sum_{j=1}^{m_{1}}\int_{\mathbb{R}^{n}}J^{s}uq_{\alpha-j}(x,D)J^{s}u\, dx+\frac{1}{2}\int_{\mathbb{R}^{n}}J^{s}u r_{\alpha-m_{1}-1}(x,D)J^{s}u\, dx\\
		&=\sum_{j=1}^{m_{1}}\Lambda_{1,j}(t)+\frac{1}{2}\int_{\mathbb{R}^{n}}J^{s}u r_{\alpha-m_{1}-1}(x,D)J^{s}u\, dx
	\end{split}
\end{equation*}
A straightforward calculus shows that
\begin{equation}\label{com1}
	q_{\alpha-{1}}(x,D)=-\frac{\alpha}{2}\sum_{|\beta|=1}\partial_{x}^{\beta}\partial_{x_{1}}\left(\chi_{\epsilon, \tau}^{2}\right)J^{\alpha-2}\partial_{x}^{\beta}
\end{equation} 
that after  replacing it into \eqref{commudecomp1..1}  yield
\begin{equation*}
	\begin{split}
	\Lambda_{1,1}(t) &=-\frac{\alpha}{2}\sum_{|\beta|=1}\int_{\mathbb{R}^{n}}J^{s}u J^{s+\alpha-2}\partial_{x}^{\beta}u \partial_{x}^{\beta}\partial_{x_{1}}(\chi_{\epsilon, \tau}^{2})\, dx\\
		&=-\frac{\alpha}{2}\sum_{|\beta|=1}\int_{\mathbb{R}^{n}}J^{s+\frac{\alpha-2}{2}}u J^{s+\frac{\alpha-2}{2}}\partial_{x}^{\beta}u \partial_{x}^{\beta}\partial_{x_{1}}(\chi_{\epsilon, \tau}^{2})\, dx\\
		&\quad  -\frac{\alpha}{2}\sum_{|\beta|=1}\int_{\mathbb{R}^{n}}J^{s+\frac{\alpha-2}{2}}u\left[J^{\frac{2-\alpha}{2}}; \partial_{x}^{\beta}\partial_{x_{1}}(\chi_{\epsilon, \tau}^{2})\right] J^{s+\frac{\alpha-2}{2}}\partial_{x}^{\beta}u \, dx\\
		&=\frac{\alpha}{4}\sum_{|\beta|=1}\int_{\mathbb{R}^{n}}\left(J^{s+\frac{\alpha-2}{2}}u\right)^{2}  \partial_{x}^{2\beta}\partial_{x_{1}}(\chi_{\epsilon, \tau}^{2})\, dx\\
			&\quad  -\frac{\alpha}{2}\sum_{|\beta|=1}\int_{\mathbb{R}^{n}}J^{s+\frac{\alpha-2}{2}}u\left[J^{\frac{2-\alpha}{2}}; \partial_{x}^{\beta}\partial_{x_{1}}(\chi_{\epsilon, \tau}^{2})\right] J^{s+\frac{\alpha-2}{2}}\partial_{x}^{\beta}u \, dx\\
			&=\Lambda_{1,1}(t)+\Lambda_{1,2}(t).
	\end{split}
\end{equation*}
From \eqref{smoot} we obtain after fixing properly $\epsilon$ and $ \tau$ that 
\begin{equation*}
	\begin{split}
		\int_{0}^{T}|\Lambda_{1,1}(t)|\, dt  <\infty.
	\end{split}
\end{equation*}
In the case of $\Lambda_{1,2}$  we a void a new  commutator decomposition  as follows
\begin{equation*}
	\begin{split}
		&\Lambda_{1,2}(t)=\\
		&-\frac{\alpha}{2}\sum_{|\beta|=1}\left\{\int_{\mathbb{R}^{n}}J^{s+\frac{\alpha-2}{2}}(u\chi_{\epsilon, \tau})\left[J^{\frac{2-\alpha}{2}}; \partial_{x}^{\beta}\partial_{x_{1}}(\chi_{\epsilon, \tau}^{2})\right] J^{s+\frac{\alpha-2}{2}}\partial_{x}^{\beta}(u\chi_{\epsilon, \tau}+u\phi_{\epsilon, \tau}+u\psi_{\epsilon}u) \, dx\right.\\
		&\quad \left.+\int_{\mathbb{R}^{n}}J^{s+\frac{\alpha-2}{2}}(u\phi_{\epsilon, \tau})\left[J^{\frac{2-\alpha}{2}}; \partial_{x}^{\beta}\partial_{x_{1}}(\chi_{\epsilon, \tau}^{2})\right] J^{s+\frac{\alpha-2}{2}}\partial_{x}^{\beta}(u\chi_{\epsilon, \tau}+u\phi_{\epsilon, \tau}+u\psi_{\epsilon}u) \, dx\right.\\
			&\quad \left.+\int_{\mathbb{R}^{n}}J^{s+\frac{\alpha-2}{2}}(u\psi_{\epsilon, \tau})\left[J^{\frac{2-\alpha}{2}}; \partial_{x}^{\beta}\partial_{x_{1}}(\chi_{\epsilon, \tau}^{2})\right] J^{s+\frac{\alpha-2}{2}}\partial_{x}^{\beta}(u\chi_{\epsilon, \tau}+u\phi_{\epsilon, \tau}+u\psi_{\epsilon}u) \, dx\right\}\\
	\end{split}
\end{equation*}
In virtue of Theorem \ref{continuity}  and
 Lemma \ref{zk19} we obtain 
\begin{equation*}
	\begin{split}
	|	\Lambda_{1,2}(t)|&\lesssim \|J^{s}(u\chi_{\epsilon,\tau})\|_{L^{2}_{x}}^{2}+\left\|J^{s}(u\phi_{\epsilon,\tau})\right\|_{L^{2}_{x}}^{2}+\|u_{0}\|_{L^{2}_{x}}^{2},
	\end{split}
\end{equation*}
although 
\begin{equation}\label{partition1.1}
	J^{s}(u\chi_{\epsilon, \tau})=\chi_{\epsilon, \tau}J^{s}u+[J^{s}; \chi_{\epsilon, \tau}](u\chi_{\epsilon, \tau}+u\phi_{\epsilon, \tau}+u\psi_{\epsilon}),
\end{equation}
and the  first term in the r.h.s  is the quantity to estimate after applying Gronwall's inequality. The remainder terms are of order $s-1$ and these are estimated by using \eqref{smoot}, Lemma \ref{lemm} and Lemma  \ref{zk19}    after integrating in time.

For $\Lambda_{1,j}$ with $j>1$  is   easily handled  since  the regularity  required  for such terms is less than $s$ and therefore after integrating in time,  the inequality \eqref{smoot} is the key part. More precisely,
if we omit the constants in front  and   we replace \eqref{com1} yield
\begin{equation*}
	\begin{split}
	&	\Lambda_{1,j}(t)=\\
	&\sum_{j=2}^{m_{1}}\sum_{|\beta|=j}\int_{\mathbb{R}^{n}}J^{s}(u\chi_{\epsilon, \tau}+u\phi_{\epsilon, \tau}+u\psi_{\epsilon})\partial_{x}^{\beta}\partial_{x_{1}}(\chi_{\epsilon,\tau}^{2})\Psi_{\beta,j}J^{\alpha-j+s}(u\chi_{\epsilon, \tau}+u\phi_{\epsilon, \tau}+u\psi_{\epsilon}u)\, dx.
	\end{split}
\end{equation*}
Thus, by  combining Lemma \ref{lem1} and Theorem \ref{continuity} produce
\begin{equation*}
		|	\Lambda_{1,j}(t)|\lesssim \|J^{s}(u\chi_{\epsilon,\tau})\|_{L^{2}_{x}}^{2}+\left\|J^{s}(u\phi_{\epsilon,\tau})\right\|_{L^{2}_{x}}^{2}+\|u_{0}\|_{L^{2}_{x}}^{2},
\end{equation*}
for $j=2,3,\dots, m_{1}.$

At this point we apply \eqref{partition1.1}  as we did above  to bound $\|J^{s}(u\chi_{\epsilon, \tau})\|_{L^{2}_{x}},$  and for $\|J^{s}(u\phi_{\epsilon, \tau})\|_{L^{2}_{x}}$  it is only required to combine  Lemma \ref{lemm}   together with \eqref{smoot}.

In addition, 
\begin{equation*}
	\begin{split}
			\frac{1}{2}\int_{\mathbb{R}^{n}}J^{s}u r_{\alpha-m_{1}-1}(x,D)J^{s}u\, dx&=\frac{1}{2}\int_{\mathbb{R}^{n}}uJ^{s}r_{\alpha-m_{1}-1}(x,D)J^{s}u\, dx
	\end{split}
\end{equation*} 
 which implies  after setting 
 \begin{equation*}
 	 m_{1}= \left\lceil 2s+\alpha-1-s_{n}\right\rceil
 \end{equation*}
then
 \begin{equation*}
 \int_{0}^{T}	\left|\frac{1}{2}\int_{\mathbb{R}^{n}}uJ^{s}r_{\alpha-m_{1}-1}(x,D)J^{s}u\, dx\, dt \right|\lesssim T\|u_{0}\|_{L^{2}_{x}}\|u\|_{L^{\infty}_{T}H^{s_{n}+}_{x}}<\infty,
 \end{equation*}
where we have  used Theorem \ref{continuity} in the last  inequality above.  

A quite similar argument  applies to $\Lambda_{2}$ in \eqref{commudecomp1..1}, although for the seek of brevity we omit the details.

An idea quite similar to the one  used to bound  the term $\Lambda_{1}$  also  applies  for $\Theta_{2,1,j+2}$ for $j=2,3,\dots,m$ in  \eqref{disppart}. Indeed,
\begin{equation*}
|\Theta_{2,1,j+2}(t)|\lesssim\|J^{s}(u\chi_{\epsilon,\tau})\|_{L^{2}_{x}}^{2}+\left\|J^{s}(u\phi_{\epsilon,\tau})\right\|_{L^{2}_{x}}^{2}+\|u_{0}\|_{L^{2}_{x}}^{2},\quad \mbox{for}\quad j=2,3,\dots m.
\end{equation*}
%The terms in \eqref{primersmoo}  are estimated by considering 
%\begin{equation}\label{partition}
%	\chi_{\epsilon,\tau}(x)+\phi_{\epsilon,\tau}(x)+\psi_{\epsilon}(x)=1,\quad x\in\mathbb{R},
%\end{equation}
%the argument  is quite similar and we only provide  the details for $\Theta_{2,1,1,4}$ for the reader convenience.

%By using \eqref{partition} we get  
%\begin{equation*}
%	\begin{split}
%		&\Theta_{2,1,1,4}(t)\\
%	\end{split}
%\end{equation*}
%Next, we  combine Lemma \ref{lem1} and H\"{o}lder inequality to obtain 
%\begin{equation*}
%	\begin{split}
%	\int_{0}^{T}	|\Theta_{2,1,1,2}(t)|\, dt
%		&\lesssim \|J^{s}(u\chi_{\epsilon,\tau})\|_{L^{2}_{x}}^{2}+\|J^{s}(u\phi_{\epsilon,\tau})\|_{L^{2}_{x}}^{2}+\|u_{0}\|_{L^{2}_{x}}^{2}
%	\end{split}
%\end{equation*}
%	\begin{flushleft}
%	{\sc  \underline{Term $\Theta_{2,1,2}:$}}
%\end{flushleft}
This term  is quite important since  it contains  part of the smoothing effect we desire to  obtain.  In the first place, we rewrite the term as 
\begin{equation*}
	\begin{split}
		\Theta_{2,1,2}(t)&=\frac{\alpha}{2}\int_{\mathbb{R}^{n}}\left(J^{s+\frac{\alpha-2}{2}}\partial_{x_{1}}u\right)^{2}\partial_{x_{1}}\left(\chi_{\epsilon, \tau}^{2}\right)\, dx\\
		&\quad +\frac{\alpha}{2}\int_{\mathbb{R}^{n}}J^{s+\frac{\alpha-2}{2}}\partial_{x_{1}}u\left[J^{\frac{2-\alpha}{2}}; \partial_{x_{1}}(\chi_{\epsilon, \tau}^{2})\right] \partial_{x_{1}}J^{s}u\, dx\\
		&=\Lambda_{3}(t)+\Lambda_{4}(t).
	\end{split}
\end{equation*}
Notice that $\Lambda_{3,1}$ is the  term to be estimated after integrating in time (it contains the  desired smoothing). 

On the other hand, $\Lambda_{4}$ does not contains    terms that  will provide some useful information  in our analysis. In fact,  to provide  upper bounds  for  $\Lambda_{4}$     we require apply a decomposition of the commutator expression   quite similar to that in \eqref{commudecomp1..1} and  following  the arguments used to bound $\Lambda_{1}.$
%	\begin{flushleft}
%	{\sc  \underline{Term $\Theta_{2,1,3}:$}}
%\end{flushleft}
The expression of $\Theta_{ 2,1,3}$  contains  several interactions   that make up the  smoothing  effect. 

Nevertheless,  we require to decouple such interactions   to close the argument. 

In this sense,   we claim that  there exist $\lambda>0$ such that 
\begin{equation*}
	\begin{split}
		&\lambda\left(\int_{\mathbb{R}^{n}}\left(J^{s+\frac{\alpha}{2}}u\right)^{2}\partial_{x_{1}}\left(\chi_{\epsilon, \tau}^{2}\right)\, dx+\int_{\mathbb{R}^{n}}\left(J^{s+\frac{\alpha-2}{2}}\partial_{x_{1}}u\right)^{2}\partial_{x_{1}}\left(\chi_{\epsilon, \tau}^{2}\right)\, dx\right)\\
		&\leq \Theta_{2,1,1,1}(t)+\Lambda_{3,1}(t)+ \Theta_{2,1,3}(t),
	\end{split}
\end{equation*}
whenever
 the condition  holds true: $\nu_{1}>0$ and 
 \begin{equation}\label{cond11}
 	0<	\sqrt{\nu_{2}^{2}+\nu_{3}^{2}+\dots+\nu_{n}^{2}}<\min\left\{ \frac{2\nu_{1}}{C\sqrt{\alpha(n-1)}},\frac{\nu_{1}(1+\alpha)}{\alpha\epsilon\sqrt{n-1}}\right\},
 \end{equation}
 with   $\epsilon$ satisfying
 \begin{equation}\label{cond12}
 	0<\epsilon<\frac{\nu_{1}}{|\overline{\nu}|\sqrt{n-1}}-\frac{\alpha\sqrt{n-1}|\overline{\nu}|}{4\nu_{1}}C^{2},
 \end{equation} 
where $|\overline{\nu}|:=\sqrt{\nu_{2}^{2}+\nu_{3}^{2}+\dots+\nu_{n}^{2}}$ and
 $$C:=\inf_{f\in L^{2}(\mathbb{R}^{n}),f\neq 0}\frac{\|J^{-1}\partial_{x_{j}}f\|_{L^{2}}}{\|f\|_{L^{2}}},\quad j=2,3,\dots,n.$$
 
 The reader can  check that the proof is analogous to the  one furnished   in  {\sc Claim 1} in the proof of Lemma \ref{main2} (see
  \eqref{claim1} for more details).
%	\begin{flushleft}
%	{\sc  \underline{Term $\Theta_{2,2}:$}}
%\end{flushleft}
Since
\begin{equation*}
	\begin{split}
	\Theta_{2,2}(t)&=	\frac{1}{2}\int_{\mathbb{R}^{n}} J^{s}(u\chi_{\epsilon,\tau}+u\phi_{\epsilon,\tau}+u\psi_{\epsilon}) \left[\mathcal{K}_{\alpha}\partial_{x_{1}}; \chi_{\epsilon,\tau}^{2}\right]J^{s}(u\chi_{\epsilon,\tau}+u\phi_{\epsilon,\tau}+u\psi_{\epsilon})\,dx.
	\end{split}
\end{equation*}
An argument similar to the  used  to  handle \eqref{commutatrot1}  combined with  Lemma \ref{lem1} and corollary \ref{separated} implies that 
\begin{equation}\label{kernelruin}
	\begin{split}
		|\Theta_{2,2}(t)|
		&\lesssim \left\{\|J^{s}(u\chi_{\epsilon,\tau})\|_{L^{2}_{x}}^{2}+\left\|J^{s}(u\phi_{\epsilon,\tau})\right\|_{L^{2}_{x}}^{2}+\|u_{0}\|_{L^{2}_{x}}^{2}\right\}.	
	\end{split}
\end{equation}
Notice that
\begin{equation*}
	J^{s}(u\chi_{\epsilon,\tau})=\chi_{\epsilon,\tau}J^{s}u+\left[J^{s};\chi_{\epsilon,\tau}\right] (u\chi_{\epsilon,\tau}+u\phi_{\epsilon,\tau}+u\psi_{\epsilon}),
\end{equation*}
the first term in the r.h.s above is the quantity to be estimated after taking the $L^{2}-$norm and the remainder terms are of order $s-1.$ Although,
to control $\|J^{s}(u\phi_{\epsilon,\tau})\|_{L^{2}_{x}}$    we only require  to use Lemma \ref{lemm}  combined with \eqref{smoot}, we skip the details, in  such a case we obtain 
\begin{equation*}
	\int_{0}^{T}\|J^{s}(u(\cdot, t)\phi_{\epsilon, \tau}(\cdot, t))\|_{L^{2}_{x}}^{2}\, dt<c,
\end{equation*}
for some positive constant $c.$
%	\begin{flushleft}
%	{\sc  \underline{Term $\Theta_{3}:$}}
%\end{flushleft}

We decompose the  nonlinear term as  follows 
\begin{equation*}
	\begin{split}
			\Theta_{3}(t)&=-\int_{\mathbb{R}^{n}}\chi_{\epsilon,\tau}J^{s}u\, \left[J^{s}; \chi_{\epsilon,\tau}\right]u\partial_{x_{1}}u\, dx		+\int_{\mathbb{R}^{n}}\chi_{\epsilon,\tau}J^{s}u\left[J^{s}; u\chi_{\epsilon,\tau}\right]\partial_{x_{1}}u\\
			&\quad -\frac{1}{2}\int_{\mathbb{R}^{n}}(J^{s}u)^{2}\chi_{\epsilon,\tau}^{2}\, dx-\nu_{1}\int_{\mathbb{R}^{n}}u(J^{s}u)^{2}\chi_{\epsilon,\tau}\chi_{\epsilon,\tau}'\, dx\\
			&=\Theta_{3,1}(t)+\Theta_{3,2}(t)+\Theta_{3,3}(t)+\Theta_{3,4}(t).
	\end{split}
\end{equation*}   
The term $\Theta_{3,1}$ requires to   describe the commutator 
$\left[J^{s}; \chi_{\epsilon,\tau}\right].$
Although,  to obtain  such  expression have been a reiterative argument  in this paper and for the sake of brevity  we  will show the crucial steps. Thus,  
if we set 
\begin{equation}\label{main}
	b_{s-1}(x,D):=\left[J^{s}; \chi_{\epsilon,\tau}\right],
\end{equation}
then $b_{s-1}\in\mathrm{OP}S^{s-1},$ and its principal symbol admits the following decomposition
\begin{equation*}\label{}
	\begin{split}
		b_{s-1}(x,\xi)
		&=\sum_{1\leq|\alpha|\leq l}\frac{(2\pi i)^{-|\alpha|}}{\alpha!}\left\{\partial^{\alpha}_{\xi}\left(\langle \xi\rangle^{s}\right)\partial_{x}^{\alpha}\left(\chi_{\epsilon,\tau}(x,t)\right)\right\}+\kappa_{s-l-1}(x,\xi)\\
		&=\sum_{j=1}^{l}\sum_{|\alpha|= j}\frac{(2\pi i)^{-|\alpha|}}{\alpha!}\left\{\partial^{\alpha}_{\xi}\left(\langle \xi\rangle^{s}\right)\partial_{x}^{\alpha}\left(\chi_{\epsilon,\tau}(x,t)\right)\right\}+\kappa_{s-l-1}(x,\xi)\\
	\end{split}
\end{equation*}
where we choose $l>\lceil s-2-\frac{n}{2}\rceil.$ %which implies that  $\kappa_{s-l-1} \in \mathbb{S}^{l+1-s}\subset \mathbb{S}^{0}.$ Notice that for  $\kappa_{s-l-1},$   the operator 
%\begin{equation*}
%	\Psi_{\kappa_{s-l-1}}g(x):= \int_{\mathbb{R}^{n}}e^{2\pi i x\cdot \xi} \kappa_{s-l-1}(x,\xi)\, \widehat{g}(\xi)\,d\xi,\qquad g\in \mathcal{S}(\mathbb{R}^{n}),
%\end{equation*}
%maps $L^{2}(\mathbb{R}^{n})$ into $L^{2}(\mathbb{R}^{n})$.
%
%Since $\kappa_{s-l-1}\in \mathbb{S}^{0},$  the  Theorem \ref{continuity}  implies that 
%\begin{equation*}
%	\left\|\Psi_{\kappa_{s-l-1}}f\right\|_{L^{2}}\lesssim \|f\|_{L^{2}}.
%\end{equation*}

Additionally, for multi-index $\alpha,\beta $ with $\beta\leq \alpha$ we define  
\begin{equation*}
	\eta_{\alpha,\beta}(x,\xi):=\frac{(2\pi i\xi)^{\beta}}{\left(1+|\xi|^{2}\right)^{\frac{|\alpha|}{2}}},\qquad x,\xi\in\mathbb{R}^{n}.
\end{equation*}
Thus,  $ \eta_{\alpha,\beta}\in \mathbb{S}^{|\beta|-|\alpha|}\subset\mathbb{S}^{0},$  and  
\begin{equation}\label{e11.1}
	\Psi_{\eta_{\alpha,\beta}}g(x):=\int_{\mathbb{R}^{n}}e^{2\pi ix\cdot \xi}\eta_{\alpha,\beta}(x,\xi)\widehat{g}(\xi)\,d\xi,\quad g\in \mathcal{S}(\mathbb{R}^{n}),
\end{equation}
with this  at hand  we rearrange  the terms  in the decomposition of the symbol $q_{s-1}$  to obtain
\begin{equation*}
	\begin{split}
		\Psi_{q_{s-1}}f(x)=\sum_{j=1}^{l}\sum_{|\alpha|=j}\sum_{\beta\leq\alpha} \omega_{\alpha,\beta,\nu,s} \,\partial_{x}^{\alpha}\chi_{\epsilon, \tau}(x,t)\Psi_{\eta_{\alpha,\beta}}J^{s-|\alpha|}f(x)+ \Psi_{\kappa_{s-m-1}}f(x), 
	\end{split}
\end{equation*}
where  $f\in \mathcal{S}(\mathbb{R}^{n})$ and   $\omega_{\alpha,\beta,\nu,s}$ denotes a constant depending on the parameters indicated.

Now, we turn back to $\Theta_{3,1}$  from where we get  after combining Theorem \ref{continuity}, Lemma \ref{lem1}  and Theorem \ref{leibniz}
\begin{equation*}
	\begin{split}
	&	|\Theta_{3,1}(t)|\\
	&\lesssim \|\chi_{\epsilon,\tau}J^{s}u \|_{L^{2}_{x}}\sum_{j=1}^{l}\sum_{|\alpha|=j}\sum_{\beta\leq\alpha} |\omega_{\alpha,\beta,\nu,s}| \left\|\partial_{x}^{\alpha} \chi_{\epsilon, \tau}\Psi_{\eta_{\alpha,\beta}}J^{s-|\alpha|}\partial_{x_{1}}\left(\left(u\chi_{\epsilon, \tau}\right)^{2}\right)\right\|_{L^{2}_{x}}\\
	&\quad +\|\chi_{\epsilon,\tau}J^{s}u \|_{L^{2}_{x}}\sum_{j=1}^{l}\sum_{|\alpha|=j}\sum_{\beta\leq\alpha} |\omega_{\alpha,\beta,\nu,s}| \left\|\partial_{x}^{\alpha} \chi_{\epsilon, \tau}\Psi_{\eta_{\alpha,\beta}}J^{s-|\alpha|}\partial_{x_{1}}\left(\left(u\widetilde{\phi_{\epsilon, \tau}}\right)^{2}\right)\right\|_{L^{2}_{x}}\\
	&\quad +\|\chi_{\epsilon,\tau}J^{s}u\|_{L^{2}_{x}}\sum_{j=1}^{l}\sum_{|\alpha|=j}\sum_{\beta\leq\alpha} |\omega_{\alpha,\beta,\nu,s}| \left\|\partial_{x}^{\alpha} \chi_{\epsilon, \tau}\Psi_{\eta_{\alpha,\beta}}J^{s-|\alpha|}\partial_{x_{1}}\left(\left(u\psi_{\epsilon}\right)^{2}\right)\right\|_{L^{2}_{x}}\\
	&\lesssim \|\chi_{\epsilon,\tau}J^{s}u\|_{L^{2}_{x}}\left\{ \left\|J^{s}\left(\left(u\chi_{\epsilon, \tau}\right)^{2}\right)\right\|_{L^{2}_{x}}+\left\|J^{s}\left(\left(u\widetilde{\phi_{\epsilon, \tau}}\right)^{2}\right)\right\|_{L^{2}_{x}}+\|u\|_{L^{4}}^{2}\right\}\\
	&\lesssim \|u\|_{L^{\infty}_{x}}\|\chi_{\epsilon,\tau}J^{s}u\|_{L^{2}_{x}}\left\{ \left\|J^{s}\left(u\chi_{\epsilon, \tau}\right)\right\|_{L^{2}_{x}}+\left\|J^{s}\left(u\widetilde{\phi_{\epsilon, \tau}}\right)\right\|_{L^{2}_{x}}+\|u_{0}\|_{L^{2}}\right\},
	\end{split}
\end{equation*}
from this last inequality we have that $\|\chi_{\epsilon,\tau}J^{s}u\|_{L^{2}_{x}}$ is the quantity to be estimated and 
 if additionally we make use of \eqref{main} it follows that 
\begin{equation}\label{m2}
	J^{s}(u\chi_{\epsilon, \tau})=\chi_{\epsilon, \tau}J^{s}u+q_{s-1}(x,D)(u\chi_{\epsilon, \tau}+u\phi_{\epsilon,\tau}+u\psi_{\epsilon}),
\end{equation}
and notice that from the decomposition above the  first term in the r.h.s above is the  quantity to be estimated  by Gronwall's inequality, and the remainder terms are of lower order and localized.

To handle the expression $\left\|J^{s}\left(u\widetilde{\phi_{\epsilon, \tau}}\right)\right\|_{L^{2}_{x}}$ we     only require to  consider a suitable  cut-off function    and combine Lemma \ref{lemm}  together with \eqref{smoot} to show that 
\begin{equation}\label{m1}
	\int_{0}^{T}\left\|J^{s}\left(u(\cdot,t)\widetilde{\phi_{\epsilon, \tau}(\cdot,t)}\right)\right\|_{L^{2}_{x}}^{2}\, dt<\infty.
\end{equation}
a quite similar analysis  is used to show that 
\begin{equation}\label{m3}
	\int_{0}^{T}\left\|J^{s}\left(u(\cdot,t)\phi_{\epsilon, \tau}(\cdot,t)\right)\right\|_{L^{2}_{x}}^{2}\, dt<\infty.
\end{equation}
%	\begin{flushleft}
%	{\sc  \underline{Term $\Theta_{3,2}:$}}
%\end{flushleft}
Since 
\begin{equation*}
	\begin{split}
		\Theta_{3,2}(t)&=\int_{\mathbb{R}^{n}}\chi_{\epsilon,\tau}J^{s}u\left[J^{s}; u\chi_{\epsilon,\tau}\right]\partial_{x_{1}}u\, dx\\
		&=\int_{\mathbb{R}^{n}}\chi_{\epsilon,\tau}J^{s}u\left[J^{s}; u\chi_{\epsilon,\tau}\right]\partial_{x_{1}}\left(u\chi_{\epsilon, \tau}+u\phi_{\epsilon, \tau}+u\psi_{\epsilon}\right)\, dx\\
		&=\Theta_{3,2,1}(t)+\Theta_{3,2,2}(t)+\Theta_{3,2,3}(t).
	\end{split}
\end{equation*}
The arguments  to   bound $\Theta_{3,2,1}$ and $\Theta_{3,2,2}$ are quite similar and   are obtained by using Theorem \ref{KPDESI}
\begin{equation*}
	|\Theta_{3,2,1}(t)|\lesssim  \|\chi_{\epsilon, \tau} J^{s}u \|_{L^{2}_{x}}\\|J^{s}(u\chi_{\epsilon, \tau})\|_{L^{2}_{x}}\|\nabla u \|_{L^{\infty}_{x}}
\end{equation*}
and 
\begin{equation*}
	|\Theta_{3,2,2}(t)|\lesssim \|\nabla u \|_{L^{\infty}_{x}}\left\{\\|\chi_{\epsilon, \tau}J^{s}u\|_{L^{2}_{x}}^{2}+\|J^{s}(u\phi_{\epsilon, \tau})\|_{L^{2}_{x}}^{2}+\|J^{s}(u\chi_{\epsilon, \tau})\|_{L^{2}_{x}}^{2}\right\}.
\end{equation*}
The arguments used   control  the terms $\|J^{s}(u\chi_{\epsilon, \tau})\|_{L^{2}_{x}}^{2}$ and $\|J^{s}(u\phi_{\epsilon, \tau})\|_{L^{2}_{x}}^{2}$ were already described in \eqref{m2}  resp.  \eqref{m3}.
 
 For $\Theta_{3,2,3},$  Lemma \ref{lem1} yields
\begin{equation*}
	|\Theta_{3,2,3}(t)|\lesssim \|\chi_{\epsilon, \tau} J^{s}u \|_{L^{2}_{x}} \|u\|_{L^{\infty}_{x}}\|u_{0}\|_{L^{2}_{x}}.
\end{equation*}
%\begin{flushleft}
%	{\sc  \underline{	Terms $\Theta_{3,3}$ and $\Theta_{3,4}:$}}
%\end{flushleft}
Notice that $\Theta_{3,3}$ is the term to estimate by Gronwall's inequality  and $\Theta_{3,4}$ can be  bounded  above    by  using Sobolev embedding  and \eqref{smoot}.

Finally, we gather the  information corresponding to this step  after applying Gronwall's inequality  combined with integration in the time  variable  imply that for any $\epsilon>0$ and $\tau\geq 5\epsilon,$
\begin{equation}\label{final1}
	\begin{split}
			&\sup_{0<t<T}\int_{\mathbb{R}^{n}}(J^{s}u(x,t))^{2}\chi_{\epsilon, \tau}^{2}(\nu\cdot x+\omega t )\, dx\\ &\qquad +\int_{0}^{T}\int_{\mathbb{R}^{n}}(J^{s+\frac{\alpha}{2}}u(x,t))^{2}\left(\chi_{\epsilon, \tau}\chi_{\epsilon, \tau}'\right)(\nu\cdot x+\omega t)\, dx\, dt\leq c.
	\end{split}
\end{equation}
This estimate finish the step 1.
\begin{flushleft}
	{\sc	\underline{Step 2:}}
\end{flushleft}
In this case  we consider $s\in (s_{n},s_{n}+\frac{\alpha}{2})$. Thus, 
\begin{equation*}
	s_{n}+1-\frac{\alpha}{2}<s+1-\frac{\alpha}{2}<s_{n}+1.
\end{equation*}
As we did in step 1 we perform energy estimates, the main difference is that now are at level $s+\frac{2-\alpha}{2}.$ More precisely,
\begin{equation}\label{energy2}
	\begin{split}
		&\frac{d}{dt}\int_{\mathbb{R}^{n}}\left(J^{s+\frac{2-\alpha}{2}}u\right)^{2}\chi_{\epsilon,\tau}^{2}\,dx\underbrace{-\frac{\omega }{2}\int_{\mathbb{R}^{n}}\left(J^{s+\frac{2-\alpha}{2}}u\right)^{2}\chi_{\epsilon,\tau}\chi_{\epsilon,\tau}'\,dx}_{\Theta_{1}(t)}\\
		& \underbrace{-\int_{\mathbb{R}^{n}}J^{s+\frac{2-\alpha}{2}}u \partial_{x_{1}}(-\Delta)^{\frac{\alpha}{2}}J^{s+\frac{2-\alpha}{2}}u \chi_{\epsilon,\tau}^{2}\,dx}_{\Theta_{2}(t)}
		\underbrace{+\int_{\mathbb{R}^{n}}J^{s+\frac{2-\alpha}{2}}uJ^{s+\frac{2-\alpha}{2}}(u\partial_{x_{1}}u)\chi^{2}_{\epsilon,\tau}\,dx}_{\Theta_{3}(t)}=0.
	\end{split}
\end{equation}
In this step we will focus only on the  more difficult terms  to handle, these  correspond to $\Theta_{1}$ and  $\Theta_{2};$ the term $\Theta_{3}$  can be estimated by using similar arguments  to the ones described in the previous step. Nevertheless, for the reader's convenience we will indicate the steps that  are needed to provided the desired bounds.
%\begin{flushleft}
%	{\sc  \underline{	Term $\Theta_{1}$}}
%\end{flushleft}
An argument similar to the one used in \eqref{teta1}-\eqref{smoot} applied to  the second term in \eqref{final1} produce: for $\epsilon>0$ and $\tau\geq 5\epsilon,$
\begin{equation}\label{eqinter1}
	\int_{0}^{T}\int_{\mathcal{H}_{\{\epsilon-\omega t,\nu\}}\cap\mathcal{H}_{\{\tau-\omega t,\nu\}}} \left(J^{r}u(x,t)\right)^{2}\, dx\, dt\leq c_{r},
\end{equation}
for some positive constant $c_{r},$ and $r\in\left(0, s+\frac{\alpha}{2}\right].$ 

We need  to take into account that $$s+\frac{\alpha}{2}\geq s+\frac{2-\alpha}{2},$$ so that, the regularity in the previous step is enough to control the terms with localized regularity. 

Thus, in virtue of \eqref{eqinter1} we obtain  for $\epsilon>0$ and $\tau\geq 5\epsilon;$
\begin{equation}\label{interp1.2.1}
	\begin{split}
		\int_{0}^{T}|\Theta_{1}(t)|\, dt \lesssim\int_{\mathcal{H}_{\{\epsilon-\omega t,\nu\}}\cap\mathcal{H}^{c} _{\{\tau-\omega t,\nu\}}} \left(J^{s+\frac{2-\alpha}{2}}u(x,t)\right)^{2}\, dx\, dt\leq c_{s,\alpha}.
	\end{split}
\end{equation}
%\begin{flushleft}
%	{\sc  \underline{	Term $\Theta_{2}$}}
%\end{flushleft}
A procedure similar to the one used for $\Theta_{2}$ in the step 1    implies that 
 \begin{equation}\label{comm1.1.1}
	\begin{split}
		\Theta_{2}(t)&=\frac{1}{2}\int_{\mathbb{R}^{n}} J^{s+\frac{2-\alpha}{2}}u \left[J^{\alpha}\partial_{x_{1}}; \chi_{\epsilon,\tau}^{2}\right]J^{s+\frac{2-\alpha}{2}}u\,dx\\
		&\quad +\frac{1}{2}\int_{\mathbb{R}^{n}} J^{s+\frac{2-\alpha}{2}}u \left[\mathcal{K}_{\alpha}\partial_{x_{1}}; \chi_{\epsilon,\tau}^{2}\right]J^{s+\frac{2-\alpha}{2}}u\,dx\\
		&=\Theta_{2,1}(t)+\Theta_{2,2}(t).
	\end{split}
\end{equation}
We shall remind   from  \eqref{comm1.2} that there exist  pseudo-differential operators  $p_{\alpha-j}(x,D),$ where  $\,j\in \{1,2\cdots,m\}$  for some  $m\in \mathbb{N}$    satisfying
\begin{equation}\label{comm1.3}
	\begin{split}
		c_{\alpha}(x,D)=p_{\alpha}(x,D)+p_{\alpha-1}(x,D)+\dots+p_{\alpha-m}(x,D)+r_{\alpha-m-1}(x,D),
	\end{split}
\end{equation}
where $p_{\alpha-j}\in \mathrm{OP}\mathbb{S}^{\alpha-j}$ and $r_{\alpha-m-1}\in  \mathrm{OP}\mathbb{S}^{\alpha-1-m}.$ %for $j\in\{1,2,\dots, m\}.$ 

We choose $m$ as being 
\begin{equation}\label{control commu}
	m=\left\lceil 2s+1-s_{n}\right\rceil.
\end{equation}
Thus,
\begin{equation}\label{decomp1,11,2}
	\begin{split}
\Theta_{2,1}(t)&=\frac{1}{2}\sum_{j=0}^{m}\int_{\mathbb{R}^{n}}J^{s+\frac{2-\alpha}{2}}u p_{\alpha-j}(x,D)J^{s+\frac{2-\alpha}{2}}u\, dx\\
&\quad +\frac{1}{2}\int_{\mathbb{R}^{n}}J^{s+\frac{2-\alpha}{2}}u r_{\alpha-1-m}(x,D)J^{s+\frac{2-\alpha}{2}}u\, dx\\
&=\frac{1}{2}\sum_{j=0}^{m}\int_{\mathbb{R}^{n}}J^{s+\frac{2-\alpha}{2}}u p_{\alpha-j}(x,D)J^{s+\frac{2-\alpha}{2}}u\, dx\\
&\quad +\frac{1}{2}\int_{\mathbb{R}^{n}}uJ^{s+\frac{2-\alpha}{2}} r_{\alpha-1-m}(x,D)J^{s+\frac{2-\alpha}{2}}u\, dx\\
&=\sum_{j=0}^{m}\Theta_{2,1,j}(t)+\Theta_{2,1,m+1}(t).
\end{split}
\end{equation}
We shall remind from \eqref{decomp1.2} that   for each $j\in\{1,2,\dots,m\},$ 
\begin{equation}\label{psudoe}
	\begin{split}
		p_{\alpha-j}(x,D)=\sum_{|\beta|=j}c_{\beta,j}\partial_{x}^{\beta}(\chi_{\epsilon,\tau}^{2})\Psi_{\beta,j}J^{\alpha-j},
	\end{split}
\end{equation} 
with $\Psi_{\beta,j}\in\mathrm{OP}\mathbb{S}^{0}.$

The decomposition above  allow us to  estimate $\Theta_{2,1,j}$ for all $j>1.$ Indeed,
\begin{equation}\label{equiva}
	\begin{split}
			\int_{0}^{T}|\Theta_{2,1,j}(t)|\, dt &\lesssim \sum_{|\beta|=j}|c_{\beta,j}|\int_{0}^{T} \int_{\mathbb{R}^{n}}\left|
			J^{s+\frac{2-\alpha}{2}}u\partial_{x}^{\beta}\left(\chi_{\epsilon, \tau}^{2}\right)\Psi_{\beta,j}J^{s+1-j+\frac{\alpha}{2}}u\,\right|\, dx\, dt\\
			&\lesssim \sum_{|\beta|=j}|c_{\beta,j}|\int_{0}^{T} \int_{\mathbb{R}^{n}}\left|
			J^{s+\frac{2-\alpha}{2}}u\, \theta_{1}^{2}\, \Psi_{\beta,j}J^{s+1-j+\frac{\alpha}{2}}u\,\right|\, dx\, dt\\
			&\lesssim \int_{0}^{T} \left\|\theta_{1}J^{s+\frac{2-\alpha}{2}}u\right\|_{L^{2}_{x}}^{2}\, dt +\sum_{|\beta|=j}|c_{\beta,j}| \int_{0}^{T}\left\|\theta_{1}\Psi_{\beta,j}J^{s+1-j+\frac{\alpha}{2}}u\right\|_{L^{2}_{x}}^{2}\, dt,
	\end{split}
\end{equation}
where $\theta_{1}\in C^{\infty}(\mathbb{R}^{n}),$ such that  for all $\beta$ multi-index with $|\beta|=j,$  the following relationship holds 
\begin{equation*}
	\theta_{1}\equiv 1 \quad \mbox{on}\, \supp_{x}\left(\partial_{x}^{\beta}(\chi_{\epsilon, \tau}^{2}(\cdot,t))\right)
\end{equation*}
and
\begin{equation*}
	\supp_{x}\theta_{1}\subset \mathcal{H}_{\left\{\frac{51\epsilon}{16}-\omega t ,\nu\right\}}\cap \mathcal{H}^{c}_{\left\{\tau+\frac{51\epsilon}{16}-\omega t ,\nu\right\}},\, \forall\, t>0,
\end{equation*}
whenever $\epsilon>0$ and $\tau\geq 5\epsilon.$

Then, 
	\begin{equation*}
		\begin{split}
			\int_{0}^{T} \left\|\theta_{1}J^{s+\frac{2-\alpha}{2}}u\right\|_{L^{2}_{x}}^{2}\, dt & \lesssim \int_{0}^{T}\int_{\vspace{20mm} \mathcal{H}_{\left\{\frac{51\epsilon}{16}-\omega t ,\nu\right\}}\cap \mathcal{H}^{c}_{\left\{\frac{\tau+51\epsilon}{16}-\omega t ,\nu\right\}}}\!\!\left(J^{s+\frac{2-\alpha}{2}}u\right)^{2}\, dx\, dt<c,
		\end{split}
	\end{equation*}
after choosing $(\epsilon,\tau)=(\epsilon',\tau') $  properly in \eqref{interp1.2.1}.

We consider   $\theta_{2}\in C^{\infty}(\mathbb{R}^{n})$ satisfying 
\begin{equation*}
	\theta_{2}\equiv 1 \quad \mbox{on}\quad \supp \theta_{1}
\end{equation*}
and 
\begin{equation*}
	\supp \theta_{2}\subset \mathcal{H}_{\left\{\frac{17\epsilon}{16}-\omega t ,\nu\right\}}\cap \mathcal{H}^{c}_{\left\{2\tau+\frac{51\epsilon}{16}-\omega t ,\nu\right\}},\, \forall\, t>0,
\end{equation*}
whenever $\epsilon>0$ and $\tau\geq 5\epsilon.$

Thus, combining Lemma \ref{zk19} and \eqref{eqinter1} implies that 
\begin{equation}\label{equiva2}
	\begin{split}
		\sum_{|\beta|=j}|c_{\beta,j}| \int_{0}^{T}\left\|\theta_{1}\Psi_{\beta,j}J^{s+1-j+\frac{\alpha}{2}}u\right\|_{L^{2}_{x}}^{2}\, dt<c ,
	\end{split}
\end{equation}
for all $j\geq 2.$

From \eqref{control commu} we obtain 
\begin{equation*}
\int_{0}^{T}	|\Theta_{2,1,m+1}(t)|\, dt \lesssim T\|u_{0}\|_{L^{2}_{x}}\|u\|_{L^{\infty}_{T}H^{s_{n}^{+}}_{x}}.
\end{equation*}
It remains to estimate $\Theta_{2,1,1}$ that  did not fall into the scope of the  previous analysis. Although,  when we  replace \eqref{psudoe} in \eqref{decomp1,11,2}    we find 
\begin{equation}\label{karine1}
	\begin{split}
		\Theta_{2,1,1}(t)&=\frac{1}{2}\sum_{|\beta|=1}c_{\beta,1}\int_{\mathbb{R}^{n}}J^{s+\frac{2-\alpha}{2}}u \partial_{x}^{\beta}(\chi_{\epsilon,\tau}^{2})\Psi_{\beta,1}J^{s+\frac{\alpha}{2}}u\, dx\\
		&=\frac{1}{2}\sum_{|\beta|=1}c_{\beta,1}\int_{\mathbb{R}^{n}}J^{s+\frac{1}{2}}u \partial_{x}^{\beta}(\chi_{\epsilon,\tau}^{2})\Psi_{\beta,1}J^{s+\frac{1}{2}}u\, dx\\
		&\quad+\frac{1}{2}\sum_{|\beta|=1}c_{\beta,1}\int_{\mathbb{R}^{n}}J^{s+\frac{1}{2}}u\left[J^{\frac{\alpha-1}{2}};\partial_{x}^{\beta}(\chi_{\epsilon,\tau}^{2})\right] \Psi_{\beta,1}J^{s+\frac{1}{2}}u\, dx\\
		&=\Theta_{2,1,1,1}(t)+\Theta_{2,1,1,2}(t).
	\end{split}
\end{equation}
Since 
$s+\frac{1}{2}\leq s+\frac{\alpha}{2}$ for $\alpha\in[1,2)$  we get  up to constants
\begin{equation}\label{karine2}
	\begin{split}
	\int_{0}^{T}	|\Theta_{2,1,1,1}(t)|\,dt&\lesssim\int_{0}^{T}\int_{\mathbb{R}^{n}}\left(J^{s+\frac{1}{2}}u\right)^{2} \chi_{\frac{\epsilon}{3},\tau+\epsilon}' \,dx\, dt\\
	&\quad + \sum_{|\beta|=1}
\int_{0}^{T}\int_{\mathbb{R}^{n}}\left(\Psi_{\beta,1}J^{s+\frac{1}{2}}u\right)^{2}\left|\partial_{x}^{\beta}(\chi_{\epsilon,\tau}^{2})\right|
\, dx\, dt<\infty,
	\end{split}
\end{equation}
where we have used \eqref{eqinter1}   in the first term in the r.h.s above  and for the second  one  we combine Lemma \ref{zk19}  together with \eqref{eqinter1}.

Instead for $\Theta_{2,1,1,2}$ we decompose the  commutator expression  as
\begin{equation}\label{karine3}
	\begin{split}
		\Theta_{2,1,1,2}(t)&=\frac{1}{2}\sum_{|\beta|=1}c_{\beta,1}\int_{\mathbb{R}^{n}}J^{s}u\left[J^{\frac{\alpha}{2}};\partial_{x}^{\beta}(\chi_{\epsilon,\tau}^{2})\right] \Psi_{\beta,1}J^{s+\frac{1}{2}}u\, dx\\
		&\quad -\frac{1}{2}\sum_{|\beta|=1}c_{\beta,1}\int_{\mathbb{R}^{n}}J^{s}u\left[J^{\frac{1}{2}};\partial_{x}^{\beta}(\chi_{\epsilon,\tau}^{2})\right] \Psi_{\beta,1}J^{s+\frac{\alpha}{2}}u,
	\end{split}
\end{equation}
After   use a decomposition as the one in \eqref{commudecomp1..2} (replacing $\frac{\alpha}{2}$ in \eqref{commudecomp1..2})  combined  with Theorem \ref{continuity}  and inequality \eqref{eqinter1} 
produce
\begin{equation}\label{karien4}
	\int_{0}^{T}|\Theta_{2,1,1,2}(t)|\, dt<\infty.
\end{equation}
%\begin{flushleft}
%	{\sc  \underline{	Term $\Theta_{2,1,0}$}}
%\end{flushleft}
Next, we focus  our attention  into $\Theta_{2,1,0}$. Indeed, from \eqref{disperive} we have the explicit expression for $p_{\alpha}(x,D),$ so that after replacing it in \eqref{decomp1,11,2} yield
\begin{equation*}
	\begin{split}
		\Theta_{2,1,0}(t)&=\frac{1}{2}\int_{\mathbb{R}^{n}}J^{s+\frac{2-\alpha}{2}}u J^{s+\frac{2+\alpha}{2}}u\partial_{x_{1}}\, \left(\chi_{\epsilon, \tau}^{2}\right)\, dx-\frac{\alpha}{2}\int_{\mathbb{R}^{n}}J^{s+\frac{2-\alpha}{2}}u\,J^{s+\frac{\alpha-2}{2}}\partial_{x_{1}}^{2}u\, \partial_{x_{1}}\left(\chi_{\epsilon, \tau}^{2}\right)\, dx\\
		&\quad -\frac{\alpha}{2}\sum_{|\beta|=1, \beta\neq \mathrm{e}_{1}}\int_{\mathbb{R}^{n}}J^{s+\frac{2-\alpha}{2}}u\,  J^{s+\frac{\alpha-2}{2}}\partial_{x_{1}}\partial_{x}^{\beta}u\, \partial_{x}^{\beta}\left(\chi_{\epsilon, \tau}^{2}\right)\, dx\\
		&=	\Theta_{2,1,0,1}(t)+	\Theta_{2,1,0,2}(t)+	\Theta_{2,1,0,3}(t).
	\end{split}
\end{equation*}
%\begin{flushleft}
%	{\sc  \underline{	Term $\Theta_{2,1,0,1}$}}
%\end{flushleft}
In the first place 
\begin{equation*}
	\begin{split}
		\Theta_{2,1,0,1}(t)&=\frac{1}{2}\int_{\mathbb{R}^{n}}\left(J^{s+1}u\right)^{2}\partial_{x_{1}}\left(\chi_{\epsilon, \tau}^{2}\right)\, dx+\frac{1}{2}\int_{\mathbb{R}^{n}}J^{s+1}u\left[J^{\frac{\alpha}{2}}; \partial_{x_{1}}(\chi_{\epsilon, \tau}^{2})\right]J^{s+\frac{2-\alpha}{2}}u\, dx\\
		&=	\Theta_{2,1,0,1,1}(t)+	\Theta_{2,1,0,1,2}(t).
	\end{split}
\end{equation*}
The term $	\Theta_{2,1,0,1,1}(t)$  represents after integrating in time the pursuit smoothing effect.

For $\Theta_{2,1,0,1,2}$ we rewrite it as 
\begin{equation*}
	\begin{split}
		\Theta_{2,1,0,1,2}(t)&=\frac{1}{2}\int_{\mathbb{R}^{n}}J^{s}u\left[J^{1+\frac{\alpha}{2}}; \partial_{x_{1}}\left(\chi_{\epsilon, \tau}^{2}\right)\right] J^{s+1-\frac{\alpha}{2}}u\, dx\\
		&\quad -\frac{1}{2}\int_{\mathbb{R}^{n}}J^{s}u\left[J; \partial_{x_{1}}\left(\chi_{\epsilon, \tau}^{2}\right)\right] J^{s+1}u\, dx\\
		&=\Lambda_{1}(t)+\Lambda_{2}(t).
	\end{split}
\end{equation*} 
To handle $\Lambda_{1}$ we  use  the expression \eqref{commudecomp1..2} (after replacing $\frac{\alpha}{2}$ by $\frac{\alpha}{2}+1$), the  main difference relies in the fact that  we shall   we shall    fix   the positive integer $m_{1}$  as being 
\begin{equation*}
	m_{1}=\left\lceil 2s+\frac{3\alpha}{2}-2-s_{n}\right\rceil.
\end{equation*}
Notice that from such decomposition we obtain terms of lower order  which are  controlled by using \eqref{final1}.
%\begin{flushleft}
%	{\sc  \underline{	Term $\Theta_{2,1,0,2}$}}
%\end{flushleft}
\begin{equation*}
	\begin{split}
		\Theta_{2,1,0,2}(t)&= \frac{\alpha}{2}\int_{\mathbb{R}^{n}}\left(\partial_{x_{1}}J^{s}u\right)^{2}\, \partial_{x_{1}}\left(\chi_{\epsilon, \tau}^{2}\right)\, dx\\
		&\quad +\frac{\alpha}{2}\int_{\mathbb{R}^{n}}\partial_{x_{1}}J^{s}u\left[J^{\frac{2-\alpha}{2}}; \partial_{x_{1}}\left(\chi_{\epsilon, \tau}^{2}\right)\right]J^{s+\frac{\alpha-2}{2}}u\, dx\\
		&=\Theta_{2,1,0,2,1}(t)+\Theta_{2,1,0,2,2}(t).
	\end{split}
\end{equation*}
The term $\Theta_{2,1,0,2,1}$ represents the smoothing effect after integrating in the temporal variable.

Instead, the remainder does not contain  terms that  yield some smoothing, thus we  require  to  estimate  it. In this sense, we  rewrite it as 
\begin{equation*}
	\begin{split}
		\Theta_{2,1,0,2,2}(t)&=\frac{\alpha}{2}\int_{\mathbb{R}^{n}}\partial_{x_{1}}J^{s}u\left[J^{\frac{2-\alpha}{2}}; \partial_{x_{1}}\left(\chi_{\epsilon, \tau}^{2}\right)\right]J^{s+\frac{\alpha-2}{2}}u\, dx\\
		&=-\frac{\alpha}{2}\int_{\mathbb{R}^{n}}J^{s}u\left[J^{\frac{2-\alpha}{2}}; \partial_{x_{1}}^{2}\left(\chi_{\epsilon, \tau}^{2}\right)\right]J^{s+\frac{\alpha-2}{2}}u\, dx\\
		&\quad- \frac{\alpha}{2}\int_{\mathbb{R}^{n}}J^{s}u\left[J^{\frac{2-\alpha}{2}}; \partial_{x_{1}}\left(\chi_{\epsilon, \tau}^{2}\right)\right]\partial_{x_{1}}J^{s+\frac{\alpha-2}{2}}u\, dx\\
		&=\Lambda_{3}(t)+\Lambda_{4}(t).
	\end{split}
\end{equation*}
We indicate how to estimate $\Lambda_{3}.$ We start by  writing 
\begin{equation*}
	\begin{split}
		&\Lambda_{3}(t)=\\
		&\frac{\alpha}{2}\int_{\mathbb{R}^{n}}J^{s}\left(u\chi_{\epsilon, \tau}+u\phi_{\epsilon, \tau}+u\psi_{\epsilon}\right)\left[J^{\frac{2-\alpha}{2}}; \partial_{x_{1}}^{2}\left(\chi_{\epsilon, \tau}^{2}\right)\right]J^{s+\frac{\alpha-2}{2}}\left(u\chi_{\epsilon, \tau}+u\phi_{\epsilon, \tau}+u\psi_{\epsilon}\right)dx\\
	\end{split}
\end{equation*}
Hence by Theorem \ref{continuity} and Lemma \ref{lem1} we obtain 
\begin{equation*}
	\begin{split}
		|	\Theta_{2,1,0,2,2}(t)|\lesssim \|J^{s}(u\chi_{\epsilon, \tau})\|_{L^{2}_{x}}^{2}+\|J^{s}(u\phi_{\epsilon, \tau})\|_{L^{2}_{x}}^{2}+\|u_{0}\|_{L^{2}_{x}}^{2}.
	\end{split}
\end{equation*}
Unlike the previous step  to estimate $\|J^{s}(u\chi_{\epsilon, \tau})\|_{L^{2}}$ we only require to apply Lemma \ref{zk37} and \eqref{final1}  we get 
\begin{equation*}
	\sup_{t\in(0,T)}\|J^{s}(u\chi_{\epsilon, \tau})\|_{L^{2}_{x}}^{2}<\infty.
\end{equation*}
Instead,  to estimate $\|J^{s}\left(u\phi_{\epsilon, \tau}\right)\|_{L^{2}_{T}L^{2}_{x}}$  we combine \ref{final1}  and the arguments described in \eqref{m1}-\eqref{m3} to obtain $$\|J^{s}(u\phi_{\epsilon, \tau})\|_{L^{2}_{T}L^{2}_{x}}<\infty.$$   

In summary, 
\begin{equation*}
	\int_{0}^{T}|	\Lambda_{3}(t)|\, dt<\infty.
\end{equation*}
The same arguments apply  to  estimate $\Lambda_{4}.$ Indeed,  
\begin{equation*}
	\int_{0}^{T}|\Lambda_{4}(t)|\,d t<\infty.
\end{equation*}
%\begin{flushleft}
%	{\sc  \underline{	Term $\Theta_{2,1,0,3}$}}
%\end{flushleft}
Under the conditions \eqref{cond11}- \eqref{cond12} we   there exist $\lambda>0$ such that 
\begin{equation*}
	\begin{split}
		&\lambda\left(\int_{\mathbb{R}^{n}}\left(J^{s+1}u\right)^{2}\partial_{x_{1}}\left(\chi_{\epsilon, \tau}^{2}\right)\, dx+\int_{\mathbb{R}^{n}}\left(J^{s-2}\partial_{x_{1}}u\right)^{2}\partial_{x_{1}}\left(\chi_{\epsilon, \tau}^{2}\right)\, dx\right)\\
		&\leq \Theta_{2,1,1,1}(t)+\Lambda_{3,1}(t)+ \Theta_{2,1,0,3}(t).
	\end{split}
\end{equation*}
%\begin{flushleft}
%	{\sc  \underline{	Term $\Theta_{2,2}$}}
%\end{flushleft}
The term that contains the kernel$\mathcal{K}_{\alpha}$ is estimated in the same  way as we did in \eqref{kernelruin}, and in such case 
\begin{equation*}
	\begin{split}
		\int_{0}^{T}|\Theta_{2,2}(t)|\, dt<\infty.
	\end{split}
\end{equation*}
The arguments of the proof  are  similar to the ones described in the proof of \eqref{claim1}. 
%		\begin{flushleft}
%		{\sc  \underline{Term $\Theta_{3}:$}}
%	\end{flushleft}
To handle the nonlinear part we use the decomposition
	\begin{equation*}
		\begin{split}
			\Theta_{3}(t)&=-\int_{\mathbb{R}^{n}}\chi_{\epsilon,\tau}J^{s+\frac{2-\alpha}{2}}u\, \left[J^{s+\frac{2-\alpha}{2}}; \chi_{\epsilon,\tau}\right]u\partial_{x_{1}}u\, dx	\\
			&\quad 	+\int_{\mathbb{R}^{n}}\chi_{\epsilon,\tau}J^{s+\frac{2-\alpha}{2}}u\left[J^{s+\frac{2-\alpha}{2}}; u\chi_{\epsilon,\tau}\right]\partial_{x_{1}}u\\
			&\quad -\frac{1}{2}\int_{\mathbb{R}^{n}}\left(J^{s+\frac{2-\alpha}{2}}u\right)^{2}\chi_{\epsilon,\tau}^{2}\, dx-\nu_{1}\int_{\mathbb{R}^{n}}u\left(J^{s+\frac{2-\alpha}{2}}u\right)^{2}\chi_{\epsilon,\tau}\chi_{\epsilon,\tau}'\, dx\\
		\end{split}
	\end{equation*}   
the reader can notice that this decomposition is just adapted to the  regularity of the step. Although, the process required to estimate the terms above is similar as the ones used in the $\Theta_{3}$ in the previous step, so for the sake of brevity we will omit the details.

In summary, we  find after applying Gronwall's inequality and integrating in time  the following: for any $\epsilon>0$ and $\tau\geq 5\epsilon,$
\begin{equation}\label{final2}
	\begin{split}
		&\sup_{0<t<T}\int_{\mathbb{R}^{n}}\left(J^{s+\frac{2-\alpha}{2}}u(x,t)\right)^{2}\chi_{\epsilon, \tau}^{2}(\nu\cdot x+\omega t )\, dx\\ &\qquad +\int_{0}^{T}\int_{\mathbb{R}^{n}}(J^{s+1}u(x,t))^{2}\left(\chi_{\epsilon, \tau}\chi_{\epsilon, \tau}'\right)(\nu\cdot x+\omega t)\, dx\, dt\leq c.
	\end{split}
\end{equation}
As part of the inductive process we suppose that $s\in (h,h+1), h\in\mathbb{N}
,$  with  the condition that for any 
$\epsilon>0$ and $\tau\geq 5\epsilon$ the following  estimate holds true:
\begin{equation}\label{final3}
	\begin{split}
		&\sup_{0<t<T}\int_{\mathbb{R}^{n}}\left(J^{s}u(x,t)\right)^{2}\chi_{\epsilon, \tau}^{2}(\nu\cdot x+\omega t )\, dx\\ &\qquad +\int_{0}^{T}\int_{\mathbb{R}^{n}}(J^{s+\frac{\alpha}{2}}u(x,t))^{2}\left(\chi_{\epsilon, \tau}\chi_{\epsilon, \tau}'\right)(\nu\cdot x+\omega t)\, dx\, dt\leq c.
	\end{split}
\end{equation}
As usual our starting point is the energy estimate \eqref{energy2}. 

As we did in the previous case we only  describe the  part that  will provide the smoothing effect since the remainder terms can be estimated in a standard way  with arguments  described previously.

First, we rewrite $\Theta_{2}$ from \eqref{energy2} as follows
 \begin{equation}\label{comm1.1.1}
	\begin{split}
		\Theta_{2}(t)&=\frac{1}{2}\int_{\mathbb{R}^{n}} J^{s+\frac{2-\alpha}{2}}u \left[J^{\alpha}\partial_{x_{1}}; \chi_{\epsilon,\tau}^{2}\right]J^{s+\frac{2-\alpha}{2}}u\,dx\\
		&\quad +\frac{1}{2}\int_{\mathbb{R}^{n}} J^{s+\frac{2-\alpha}{2}}u \left[\mathcal{K}_{\alpha}\partial_{x_{1}}; \chi_{\epsilon,\tau}^{2}\right]J^{s+\frac{2-\alpha}{2}}u\,dx\\
		&=\Theta_{2,1}(t)+\Theta_{2,2}(t).
	\end{split}
\end{equation}
We shall remind   that there exist  pseudo-differential operators  $p_{\alpha-j}(x,D)$  for  each  $\,j\in \{1,2\cdots,m\}$  and  some  $m\in \mathbb{N}$    satisfying
\begin{equation*}
	\begin{split}
		c_{\alpha}(x,D)=p_{\alpha}(x,D)+p_{\alpha-1}(x,D)+\dots+p_{\alpha-m}(x,D)+r_{\alpha-m-1}(x,D),
	\end{split}
\end{equation*}
where $p_{\alpha-j}\in \mathrm{OP}\mathbb{S}^{\alpha-j}$ and $r_{\alpha-m-1}\in  \mathrm{OP}\mathbb{S}^{\alpha-1-m}.$ %for $j\in\{1,2,\dots, m\}.$ 

We choose $m$ as being 
\begin{equation*}
	m=\left\lceil 2s+1-s_{n}\right\rceil.
\end{equation*}
Thus,
\begin{equation*}
	\begin{split}
		\Theta_{2,1}(t)&=\frac{1}{2}\sum_{j=0}^{m}\int_{\mathbb{R}^{n}}J^{s+\frac{2-\alpha}{2}}u p_{\alpha-j}(x,D)J^{s+\frac{2-\alpha}{2}}u\, dx\\
		&\quad +\frac{1}{2}\int_{\mathbb{R}^{n}}J^{s+\frac{2-\alpha}{2}}u r_{\alpha-1-m}(x,D)J^{s+\frac{2-\alpha}{2}}u\, dx\\
		&=\frac{1}{2}\sum_{j=0}^{m}\int_{\mathbb{R}^{n}}J^{s+\frac{2-\alpha}{2}}u p_{\alpha-j}(x,D)J^{s+\frac{2-\alpha}{2}}u\, dx\\
		&\quad +\frac{1}{2}\int_{\mathbb{R}^{n}}uJ^{s+\frac{2-\alpha}{2}} r_{\alpha-1-m}(x,D)J^{s+\frac{2-\alpha}{2}}u\, dx\\
		&=\sum_{j=0}^{m}\Theta_{2,1,j}(t)+\Theta_{2,1,m+1}(t).
	\end{split}
\end{equation*}
where 
\begin{equation*}
	\begin{split}
		p_{\alpha-j}(x,D)=\sum_{|\beta|=j}c_{\beta,j}\partial_{x}^{\beta}(\chi_{\epsilon,\tau}^{2})\Psi_{\beta,j}J^{\alpha-j},
	\end{split}
\end{equation*} 
for each $j\in\{1,2,\dots,m\},$  with $\Psi_{\beta,j}\in\mathrm{OP}\mathbb{S}^{0}.$

The key part to control the terms  in the  expression above relies into combine assumption \eqref{final3} and apply Lemma \ref{zk37}  properly to obtain    that 
$\epsilon>0$ and $\tau\geq 5\epsilon,$   the following  estimates holds true:
\begin{equation}\label{final3.1}
		\sup_{0<t<T}\int_{\mathbb{R}^{n}}\left(J^{r_{1}}u(x,t)\right)^{2}\chi_{\epsilon, \tau}^{2}(\nu\cdot x+\omega t )\, dx\leq c 
\end{equation}
for any $r_{1}\in \left(0,s\right],$
and 
\begin{equation}\label{final3.2}
	\begin{split}
		 \int_{0}^{T}\int_{\mathbb{R}^{n}}(J^{r_{2}}u(x,t))^{2}\left(\chi_{\epsilon, \tau}\chi_{\epsilon, \tau}'\right)(\nu\cdot x+\omega t)\, dx\, dt\leq c.
	\end{split}
\end{equation}
for all  $r_{2}\in \left(0, s+\frac{\alpha}{2}\right].$

The term $\Theta_{2,1,j}\,  j\geq 2,$  can be estimated  combining \eqref{final3.1}-\eqref{final3.2}  together  the arguments already described in \eqref{equiva}-\eqref{equiva2}.

Next, the  distinguished term in our decomposition is $\Theta{2,1,0}$   and it is given by 
\begin{equation*}
	\begin{split}
		\Theta_{2,1,0}(t)&=\frac{1}{2}\int_{\mathbb{R}^{n}}J^{s+\frac{2-\alpha}{2}}u J^{s+\frac{2+\alpha}{2}}u\partial_{x_{1}}\, \left(\chi_{\epsilon, \tau}^{2}\right)\, dx\\
		&\quad -\frac{\alpha}{2}\int_{\mathbb{R}^{n}}J^{s+\frac{2-\alpha}{2}}u\,J^{s+\frac{\alpha-2}{2}}\partial_{x_{1}}^{2}u\, \partial_{x_{1}}\left(\chi_{\epsilon, \tau}^{2}\right)\, dx\\
		&\quad -\frac{\alpha}{2}\sum_{|\beta|=1, \beta\neq \mathrm{e}_{1}}\int_{\mathbb{R}^{n}}J^{s+\frac{2-\alpha}{2}}u\,  J^{s+\frac{\alpha-2}{2}}\partial_{x_{1}}\partial_{x}^{\beta}u\, \partial_{x}^{\beta}\left(\chi_{\epsilon, \tau}^{2}\right)\, dx\\
		&=	\Theta_{2,1,0,1}(t)+	\Theta_{2,1,0,2}(t)+	\Theta_{2,1,0,3}(t).
	\end{split}
\end{equation*}
%\begin{flushleft}
%	{\sc  \underline{	Term $\Theta_{2,1,0,1}$}}
%\end{flushleft}
In the first place 
\begin{equation*}
	\begin{split}
		\Theta_{2,1,0,1}(t)&=\frac{1}{2}\int_{\mathbb{R}^{n}}\left(J^{s+1}u\right)^{2}\partial_{x_{1}}\left(\chi_{\epsilon, \tau}^{2}\right)\, dx+\frac{1}{2}\int_{\mathbb{R}^{n}}J^{s+1}u\left[J^{\frac{\alpha}{2}}; \partial_{x_{1}}(\chi_{\epsilon, \tau}^{2})\right]J^{s+\frac{2-\alpha}{2}}u\, dx\\
		&=	\Theta_{2,1,0,1,1}(t)+	\Theta_{2,1,0,1,2}(t).
	\end{split}
\end{equation*}
The term $	\Theta_{2,1,0,1,1}$  represents after integrating in time the pursuit smoothing effect.

For $\Theta_{2,1,0,1,2}$ we rewrite it as 
\begin{equation*}
	\begin{split}
		\Theta_{2,1,0,1,2}(t)&=\frac{1}{2}\int_{\mathbb{R}^{n}}J^{s}u\left[J^{1+\frac{\alpha}{2}}; \partial_{x_{1}}\left(\chi_{\epsilon, \tau}^{2}\right)\right] J^{s+1-\frac{\alpha}{2}}u\, dx\\
		&\quad -\frac{1}{2}\int_{\mathbb{R}^{n}}J^{s}u\left[J; \partial_{x_{1}}\left(\chi_{\epsilon, \tau}^{2}\right)\right] J^{s+1}u\, dx\\
		&=\Lambda_{1}(t)+\Lambda_{2}(t).
	\end{split}
\end{equation*} 
To handle $\Lambda_{1}$ we split the commutator expression as in  \eqref{commudecomp1..2} (after replacing $\frac{\alpha}{2}$ by $\frac{\alpha}{2}+1$), but fixing $m_{1}\in\mathbb{N}$  as being 
\begin{equation*}
	m_{1}=\left\lceil 2s+\frac{3\alpha}{2}-2-s_{n}\right\rceil.
\end{equation*}
Notice that from such decomposition we obtain terms of lower order  which are  controlled by using \eqref{final3.1}-\eqref{final3.2}.
%\begin{flushleft}
%	{\sc  \underline{	Term $\Theta_{2,1,0,2}$}}
%\end{flushleft}
\begin{equation*}
	\begin{split}
		\Theta_{2,1,0,2}(t)&= \frac{\alpha}{2}\int_{\mathbb{R}^{n}}\left(\partial_{x_{1}}J^{s}u\right)^{2}\, \partial_{x_{1}}\left(\chi_{\epsilon, \tau}^{2}\right)\, dx\\
		&\quad +\frac{\alpha}{2}\int_{\mathbb{R}^{n}}\partial_{x_{1}}J^{s}u\left[J^{\frac{2-\alpha}{2}}; \partial_{x_{1}}\left(\chi_{\epsilon, \tau}^{2}\right)\right]J^{s+\frac{\alpha-2}{2}}u\, dx\\
		&=\Theta_{2,1,0,2,1}(t)+\Theta_{2,1,0,2,2}(t).
	\end{split}
\end{equation*}
The term $\Theta_{2,1,0,2,1}$ corresponds to the smoothing effect of this step after integrating in time.

Next, we rewrite
\begin{equation*}
	\begin{split}
		\Theta_{2,1,0,2,2}(t)&=\frac{\alpha}{2}\int_{\mathbb{R}^{n}}\partial_{x_{1}}J^{s}u\left[J^{\frac{2-\alpha}{2}}; \partial_{x_{1}}\left(\chi_{\epsilon, \tau}^{2}\right)\right]J^{s+\frac{\alpha-2}{2}}u\, dx\\
		&=-\frac{\alpha}{2}\int_{\mathbb{R}^{n}}J^{s}u\left[J^{\frac{2-\alpha}{2}}; \partial_{x_{1}}^{2}\left(\chi_{\epsilon, \tau}^{2}\right)\right]J^{s+\frac{\alpha-2}{2}}u\, dx\\
		&\quad- \frac{\alpha}{2}\int_{\mathbb{R}^{n}}J^{s}u\left[J^{\frac{2-\alpha}{2}}; \partial_{x_{1}}\left(\chi_{\epsilon, \tau}^{2}\right)\right]\partial_{x_{1}}J^{s+\frac{\alpha-2}{2}}u\, dx\\
		&=\Lambda_{3}(t)+\Lambda_{4}(t).
	\end{split}
\end{equation*}
As we  did previously we only  indicate how to estimate $\Lambda_{3}$ since for $\Lambda_{4}$ the situation is analogous. We start by  writing 
\begin{equation*}
	\begin{split}
		&\Lambda_{3}(t)=\\
		&\frac{\alpha}{2}\int_{\mathbb{R}^{n}}J^{s}\left(u\chi_{\epsilon, \tau}+u\phi_{\epsilon, \tau}+u\psi_{\epsilon}\right)\left[J^{\frac{2-\alpha}{2}}; \partial_{x_{1}}^{2}\left(\chi_{\epsilon, \tau}^{2}\right)\right]J^{s+\frac{\alpha-2}{2}}\left(u\chi_{\epsilon, \tau}+u\phi_{\epsilon, \tau}+u\psi_{\epsilon}\right)dx\\
	\end{split}
\end{equation*}
Hence by Theorem \ref{continuity} and Lemma \ref{lem1} we obtain 
\begin{equation*}
	\begin{split}
		|	\Theta_{2,1,0,2,2}(t)|\lesssim \|J^{s}(u\chi_{\epsilon, \tau})\|_{L^{2}_{x}}^{2}+\|J^{s}(u\phi_{\epsilon, \tau})\|_{L^{2}_{x}}^{2}+\|u_{0}\|_{L^{2}_{x}}^{2}.
	\end{split}
\end{equation*}
Unlike the previous step  to estimate $\|J^{s}(u\chi_{\epsilon, \tau})\|_{L^{2}}$ we only require to apply Lemma \ref{zk37} and \eqref{final1}  we get 
\begin{equation*}
	\sup_{t\in(0,T)}\|J^{s}(u\chi_{\epsilon, \tau})\|_{L^{2}_{x}}^{2}<\infty.
\end{equation*}
Instead,  to estimate $\|J^{s}(u\phi_{\epsilon, \tau})\|_{L^{2}_{T}L^{2}_{x}}$  we combine \ref{final1}  and the arguments described in \eqref{m1}-\eqref{m3} to obtain $$\|J^{s}(u\phi_{\epsilon, \tau})\|_{L^{2}_{T}L^{2}_{x}}<\infty.$$   

In summary, 
\begin{equation*}
	\int_{0}^{T}|	\Lambda_{3}(t)|\, dt<\infty.
\end{equation*}
The same arguments apply  to  estimate $\Lambda_{4}.$ Indeed,  
\begin{equation*}
	\int_{0}^{T}|\Lambda_{4}(t)|\,d t<\infty.
\end{equation*}
In the case of $\Theta_{2,1,1}$ the arguments  described in \eqref{karine1}-\eqref{karine3} implies  together with \eqref{final3.1} and \eqref{final3.2}
 that 
 \begin{equation*}
 \int_{0}^{T}|\Theta_{2,1,1}(t)|\, dt <c
\end{equation*}
for some positive constant $c.$

Hence, we focus our attention on $\Theta_{2,1,m+1},$ 
   the  error term  containing $r_{m-\alpha-1}(x,D)$ satisfies in virtue of Theorem \ref{continuity}
\begin{equation*}
	\int_{0}^{T}|\Theta_{2,1,m+1}(t)|\, dt\lesssim T\|u_{0}\|_{L^{2}_{x}}\|u\|_{L^{\infty}_{T}H^{s_{n}+}_{x}}.
\end{equation*}
%\begin{flushleft}
%	{\sc  \underline{	Term $\Theta_{2,1,0,3}$}}
%\end{flushleft}
Under the conditions \eqref{cond11}- \eqref{cond12} we   there exist $\lambda>0,$ such that 
\begin{equation*}
	\begin{split}
		&\lambda\left(\int_{\mathbb{R}^{n}}\left(J^{s+1}u\right)^{2}\partial_{x_{1}}\left(\chi_{\epsilon, \tau}^{2}\right)\, dx+\int_{\mathbb{R}^{n}}\left(J^{s-2}\partial_{x_{1}}u\right)^{2}\partial_{x_{1}}\left(\chi_{\epsilon, \tau}^{2}\right)\, dx\right)\\
		&\leq \Theta_{2,1,1,1}(t)+\Lambda_{3,1}(t)+ \Theta_{2,1,0,3}(t).
	\end{split}
\end{equation*}
The analysis of this term remains unchained from the one in the previous steps.

Finally, we gather the estimates above from where we find that for any $\epsilon>0$ and $\tau \geq 5 \epsilon,$
\begin{equation}\label{final3}
	\begin{split}
		&\sup_{0<t<T}\int_{\mathbb{R}^{n}}\left(J^{s+\frac{2-\alpha}{2}}u(x,t)\right)^{2}\chi_{\epsilon, \tau}^{2}(\nu\cdot x+\omega t )\, dx\\ &\qquad +\int_{0}^{T}\int_{\mathbb{R}^{n}}(J^{s+1}u(x,t))^{2}\left(\chi_{\epsilon, \tau}\chi_{\epsilon, \tau}'\right)(\nu\cdot x+\omega t)\, dx\, dt\leq c.
	\end{split}
\end{equation}
This estimate finishes the inductive argument, and  we conclude the proof.
\end{proof}
The attentive  reader might  naturally wonder if it is possible to obtain a regularity propagation of regularity result as the one proved previously if the dispersion is weaker when compared  with that of the ZK equation. More precisely, if we consider  the equation
\begin{equation*}
	\partial_{t}u-\partial_{x_{1}}\left(-\Delta\right)^{\frac{\alpha}{2}}u+u\partial_{x_{1}}u=0, \qquad 0<\alpha<1.
\end{equation*}
This question was firstly addressed in \cite{AM2} in the one dimensional case (see  \cite{AMTHESIS} for a  more  detailed exposition). It was proved in\cite{AM2} that   even in the case that the dispersion is too weak, the propagation of regularity phenomena occurs. Although,  for higher dimensions this  question has not been addressed  before  since  it was unknown how to obtain  Kato's smoothing in these cases. 
\begin{rem}
	 The arguments  described   in Lemma \ref{main2} are  strong enough  and allows  us to describe 
\end{rem}

\section{Appendix A}\label{apendice1}
The following appendix intends to  provide a summary of the main results of Pseudo differential operators  we use in this work.
\begin{defn}
	Let $m\in \mathbb{R}.$  Let $\mathbb{S}^{m}(\mathbb{R}^{n}\times\mathbb{R}^{n})$ denote the set of functions $a\in C^{\infty}(\mathbb{R}^{n}\times \mathbb{R}^{n})$ such that 
	for all $\alpha$ and all $\beta$ multi-index
	\begin{equation*}
		\left|\partial_{x}^{\alpha}\partial_{\xi}^{\beta}a(x,\xi)\right|\lesssim_{\alpha,\beta}(1+|\xi|)^{m-|\beta|},\quad \mbox{for all}\quad x,\xi \in\mathbb{R}^{n}.
	\end{equation*}
	An element $a\in \mathbb{S}^{m}(\mathbb{R}^{n}\times\mathbb{R}^{n})$ is called a \emph{symbol of order $m.$}
\end{defn}
\begin{rem}
	For the sake of simplicity in the notation from here on we	will   suppress  the dependence of the space $\mathbb{R}^{n}$ when we make reference to  a symbol in  a particular class.
\end{rem}
\begin{rem}
	For $m\in\mathbb{R},$ the class  of symbols  $\mathbb{S}^{m}$  can be described  as 
	\begin{equation*}
		\mathbb{S}^{m}=\left\{a(x,\xi)\in C^{\infty}(\mathbb{R}^{n}\times\mathbb{R}^{n})\,|\, |a|_{\mathbb{S}^{m}}^{(j)}<\infty,\, j\in\mathbb{N}\right\},
	\end{equation*}
	where 
	\begin{equation*}
		|a|_{\mathbb{S}^{m}}^{(j)}:=\sup_{x,\xi\in\mathbb{R}^{n}}\left\{\left\|\langle \xi \rangle ^{|\alpha|-m}\partial_{\xi}^{\alpha}\partial_{x}^{\beta}a(\cdot,\cdot)\right\|_{L^{\infty}_{x,\xi}}  \Big| \alpha,\beta\in\mathbb{N}_{0}^{n},\, |\alpha+\beta|\leq j  \right\}.
	\end{equation*}
\end{rem}
\begin{defn}
	\emph{A pseudo-differential  operator} is a mapping $f\mapsto \Psi f$ given by 
	\begin{equation*}
		(\Psi f)(x)=\int_{\mathbb{R}^{n}}e^{2\pi\mathrm{i}x\cdot\xi}a(x,\xi)\widehat{f}(\xi)\,d\xi,
	\end{equation*}
	where $a(x, \xi)$ is the symbol of $\Psi.$
\end{defn}
\begin{rem}
	In order to emphasize the role of the symbol $a$ we will often write  $\Psi_{a}.$ Also, we  use the notation $a(x,D)$ to denote the operator $\Psi_{a}.$
\end{rem}
\begin{defn}
	If $a(x,\xi)\in \mathbb{S}^{m},$ the operator $\Psi_{a}$ is said to belong  to $\mathrm{OP\mathbb{S}^{m}.}$ More precisely,  if $\nu$ is  any symbol class and $a(x,\xi)\in \nu,$ we say that $\Psi_{a}\in \mathrm{OP}\mathbb{S}^{m}.$ 
\end{defn}
A quite remarkable property that pseudo-differential operators enjoy  is the existence of the adjoint operator, that is described below in terms of  its asymptotic decomposition. 
\begin{thm}
	Let $a\in \mathbb{S}^{m}.$ Then, there exist $a^{*}\in \mathbb{S}^{m}$ such that  $\Psi_{a}^{*}=\Psi_{a^{*}},$ and for all $N\geq 0,$
	\begin{equation*}
		a^{*}(x,\xi)-\sum_{|\alpha|<N}\frac{(2\pi i)^{-|\alpha|}}{\alpha!}\partial_{\xi}^{\alpha}\partial_{x}^{\alpha}\overline{a}(x,\xi)\in\mathbb{S}^{m-N}.
	\end{equation*}
\end{thm}
\begin{proof}
	See Stein \cite{stein3} chapter VI or Taylor \cite{MT1}.
\end{proof}

Additionally   the product $\Psi_{a}\Psi_{b}$ of two  operators  with symbols $a(x,\xi)$ and $b(x,\xi)$  respectively  is a pseudo-differential operator $\Psi_{c}$  with symbol  $c(x,\xi).$ More precisely, the description of the symbol $c$ is summarized in the following Theorem:
\begin{thm}
	Suppose $a$ and $b$ symbols belonging to $\mathbb{S}^{m}$ and  $ \mathbb{S}^{r}$
	respectively. Then, there is  a symbol $c$ in  $\mathbb{S}^{m+r}$ so that 
	\begin{equation*}
		\Psi_{c}=\Psi_{a}\circ\Psi_{b}.
	\end{equation*}
	Moreover, 
	\begin{equation*}
		c\sim\sum_{\alpha}\frac{(2\pi i)^{-|\alpha|}}{\alpha!}\partial_{\xi}^{\alpha}a\partial^{\alpha}_{x}b,
	\end{equation*}
	in the sense  that 
	\begin{equation*}
		c-\sum_{|\alpha|<N}\frac{(2\pi i)^{-|\alpha|}}{\alpha!}\partial_{\xi}^{\alpha}a\,\partial^{\alpha}_{x}b\in \mathbb{S}^{m+r-N}, \quad \mbox{for all integer} \quad N,\, N\geq 0. 
	\end{equation*}
\end{thm}
\begin{proof}
	For the proof see Stein \cite{stein3} chapter VI or Taylor \cite{MT1}.
\end{proof}
\begin{rem}
	Note that $c-ab\in  \mathbb{S}^{m+r-1}.$ Moreover,   each  symbol of the form   $\partial_{\xi}^{\alpha}a\,\partial_{x}^{\alpha}b$ lies  in the class $\mathbb{S}^{m+r-|\alpha|}.$
\end{rem}
A direct consequence of  the decomposition above  is that it allows to describe  explicitly   up to an error term,  operators such as  commutators between pseudo- differential operators as is described below:
\begin{prop}\label{prop1}
	For $a\in \mathbb{S}^{m}$ and $b\in \mathbb{S}^{r}$  we define the commutator $\left[\Psi_{a};\Psi_{b}\right]$ by  
	\begin{equation*}
		\left[\Psi_{a};\Psi_{b}\right]=\Psi_{a}\Psi_{b}-\Psi_{b}
		\Psi_{a}.
	\end{equation*}
	Then, the operator  ${\displaystyle  \left[\Psi_{a};\Psi_{b}\right]\in \mathrm{OP}\mathbb{S}^{m+r-1},}$ has by principal   symbol the Poisson bracket, i.e, 
	\begin{equation*}
		\sum_{|\alpha|=1}^{n}\frac{1}{2\pi i}\left(\partial_{\xi}^{\alpha}a\,\partial_{x}^{\alpha}b- \partial_{x}^{\alpha}a\,\partial_{\xi}^{\alpha}b\right)\,\, \mathrm{mod}\,\, \mathbb{S}^{m+r-2}.
	\end{equation*} 
\end{prop}
Also,  certain  class the pseudo-differential operators enjoy of some continuity properties  as the described below.
%\begin{thm}\label{th1}
%	Suppose that $a$ is symbol with  $a\in \mathbb{S}^{0}.$ Then,  the operator $\Psi_{a}$ given by 
%	\begin{equation*}
%		(\Psi_{a} f)(x)=\int_{\mathbb{R}^{n}}e^{2\pi\mathrm{i}x\cdot\xi}a(x,\xi)\widehat{f}(\xi)\,d\xi,
%	\end{equation*}
%	initially defined on $\mathcal{S}(\mathbb{R}^{n}), $ extends  to a bounded  operator from $L^{2}(\mathbb{R}^{n}) $ to itself.
%\end{thm}
\begin{proof}
	The proof can be consulted in Stein \cite{stein3} Chapter VI, Theorem 1.
\end{proof}
An interesting an useful  continuity result  in the Sobolev spaces  is decribed below.
%\begin{thm}
%	Let $m\in\mathbb{R},a\in \mathbb{S}^{m},$ and  $s\in\mathbb{R}$. Then, the operator $\Psi_{a}$ extends to a bounded  linear operator  from $H^{s+m}(\mathbb{R}^{n})$ to $H^{s}(\mathbb{R}^{n}).$ Moreover, there exists $j=j(n;m,s)\in\mathbb{N}$ and $c=c(n;m;s)>0$ such that 
%	\begin{equation*}
%		\left\|\Psi_{a}f\right\|_{H^{s}_{x}}\leq c|a|_{\mathbb{S}^{m}}^{(j)}\|f\|_{H^{s+m}_{x}}
%	\end{equation*}
%\end{thm}
%\begin{proof}
%	The proof can be consulted in Kumano-go \cite{Kumano}.
%\end{proof}
\begin{thm}\label{continuity}
	Let $m\in\mathbb{R},a\in \mathbb{S}^{m},$ and  $s\in\mathbb{R}$. Then, the operator $\Psi_{a}$ extends to a bounded  linear operator  from $H^{s+m}(\mathbb{R}^{n})$ to $H^{s}(\mathbb{R}^{n}).$ Moreover, there exists $j=j(n;m,s)\in\mathbb{N}$ and $c=c(n;m;s)>0$ such that 
	\begin{equation*}
		\left\|\Psi_{a}f\right\|_{H^{s}_{x}}\leq c|a|_{\mathbb{S}^{m}}^{(j)}\|f\|_{H^{s+m}_{x}}
	\end{equation*}
\end{thm}
\begin{proof}
See Kumano-go \cite{Kumno} or Stein \cite{stein3}, Chapter VI.
\end{proof}
An alternative formula for the Bessel kernel.
\begin{lem}\label{b1}
	 Let $0<\delta<n+1,$ and  $f$ be a tempered distribution. Then $J^{-\delta}f=\mathcal{B}_{\delta}*f,$ where 
	 \begin{equation*}
	 	\mathcal{B}_{\delta}(y)=\frac{1}{(2\pi)^{\frac{n-1}{2}}2^{\frac{\delta}{2}}\Gamma\left(\frac{\delta}{2}\right)\Gamma\left(\frac{n-\delta+1}{2}\right)}e^{-|y|}\int_{0}^{\infty}e^{-|y|s}\left(s+\frac{s^{2}}{2}\right)^{\frac{n-\delta-1}{2}}
\,ds	
 \end{equation*}
\begin{proof}
	See Calderon \& Zygmund \cite{CZ}, Lemma 4.1.
\end{proof}
\end{lem}
\begin{lem}\label{b2}
	The function $\Theta_{\delta}$  is non-negative and it satisfies  the following properties:
	\begin{itemize}
		\item[(a)] For $0<\delta<n,$ 
		\begin{equation*}
			\mathcal{B}_{\delta}(x)\lesssim_{\delta} e^{-|x|}\left(1+|x|^{\delta-n}\right);
		\end{equation*}
		\item[(b)] for $\delta=n$
		\begin{equation*}
			\mathcal{B}_{\delta}(x)\lesssim  e^{-|x|}\left(1+\log^{+}\left(\frac{1}{|x|}\right)\right);
		\end{equation*}
		\item[(c)] for $\beta$ multi-index with $|\beta|>0$  and $0<\delta<n+1$
		\begin{equation*}
			\left|\partial_{x}^{\beta}\mathcal{B}_{\delta}(x)\right|\leq c_{\beta,\delta} e^{-|x|}\left(1+|x|^{-n+\delta-|\beta|}\right), \quad x\neq 0.
		\end{equation*}
	\end{itemize}
\end{lem}
\begin{proof}
	We refer to  Calderon \& Zygmund \cite{CZ}, Lemma 4.2.
\end{proof}
Also, the behavior  of the Bessel potentials in the following cases is  necessary for our  arguments.
\begin{lem}\label{Lemmasimpt}
	Let $\delta>0$. The function $\mathcal{B}_{\delta}$ satisfies the following   estimates:
	\begin{itemize}
		\item[(i)] For $\delta<n$ and $|y|\rightarrow 0,$
		\begin{equation}\label{asimp1}
			\mathcal{B}_{\delta}(y)\approx\left( \frac{\pi^{\frac{n}{2}}\Gamma\left(\frac{n-\delta}{2}\right)}{2^{\delta-n}\Gamma\left(\frac{\delta}{2}\right)}|2\pi y|^{\tau-n}\right).
		\end{equation}
	\item[(ii)] For $\delta=n$ and $|y|\rightarrow 0,$
	\begin{equation}\label{asimp2}
	\mathcal{B}_{\delta}(y)\approx\frac{\pi
	^{n/2}}{\Gamma\left(\frac{n}{2}\right)}\log\left(\frac{1}{|2\pi y|}\right).
\end{equation}
\item[(iii)] For $\delta>n$ and $|y|\rightarrow 0,$
\begin{equation}\label{simp3}
	\mathcal{B}_{\delta}(y)\approx\frac{\pi^{\frac{n}{2}}\Gamma\left(\frac{\delta-n}{2}\right)}{\Gamma\left(\frac{\delta}{2}\right)}.
\end{equation}
\item[(iv)] For $\delta>0$ and $|y|\rightarrow \infty,$
\begin{equation}\label{asimp4}
	\mathcal{B}_{\delta}(y)\approx\frac{(2\pi)^{\frac{n}{2}}}{2^{\frac{\alpha-1}{2}}\pi^{-\frac{1}{2}}\Gamma\left(\frac{\delta}{2}\right)}|2\pi|^{\frac{\alpha-n-1}{2}}e^{-|2\pi y|}.
\end{equation}
	\end{itemize}
\end{lem}
\begin{proof}
	For the proof see Aronszajn \& Smith \cite{ARO}.
\end{proof}
\section{Appendix B}
In this section we present some localization tools that are  quite useful to describe the  regularity phenomena we are working on.
\begin{lem}\label{lem1}
	Let $\Psi_{a}\in\mathrm{OP\mathbb{S}^{r}}.$ Let  $ \alpha=\left(\alpha_{1},\alpha_{2},\dots,\alpha_{n}\right)$ be a multi-index with $|\alpha|\geq 0.$  If $f\in L^{2}(\mathbb{R}^{n})$ and $g\in L^{p}(\mathbb{R}^{n}),\, p\in [2,\infty]$  with 
	\begin{equation}\label{e16}
		\dist\left(\supp(f),\supp(g)\right)\geq \delta>0,
	\end{equation}
	then, 
	\begin{equation*}
		\left\|g\partial_{x}^{\alpha}\Psi_{a}f\right\|_{L^{2}}\lesssim \|g\|_{L^{p}}\|f\|_{L^{2}},
	\end{equation*}
	where $\partial_{x}^{\alpha}:= \partial_{x_{1}}^{\alpha_{1}}\dots\partial_{x_{n}}^{\alpha_{n}}.$ 
\end{lem}
\begin{proof}
	See Mendez \cite{AMZK}.
\end{proof}
The next result  can be proved by using the ideas  from the proof of   Lemma above. Nevertheless, for the reader convenience we describe the main  details.
\begin{cor}\label{separated}
	Let $f,g$ be functions such that 
	\begin{equation}\label{support}
		\dist\left(\supp (f),\supp(g)\right)=\delta>0.
	\end{equation}	
	Then,  the operator 
	\begin{equation*}
		\left(\mathcal{T}_{\psi}f\right)^{\widehat{}}(\xi):=\psi(\xi)\widehat{f}(\xi),\, f\in\mathcal{S}(\mathbb{R}^{n})
	\end{equation*}
	where $\psi$ is defined as in \eqref{a1}, that is,
	\begin{equation*}
		\psi(\xi)=\sum_{j=1}^{\infty} {\alpha/2 \choose j}\langle 2\pi\xi\rangle ^{2-2j},\quad \xi\in\mathbb{R}^{n}.
	\end{equation*}
	If $f\in L^{p}(\mathbb{R}^{n})$ and $g\in L^{2}(\mathbb{R}^{n}),$ then
	\begin{equation*}
		\left\|\left[\mathcal{T}_{\psi}; f\right]g\right\|_{L^{2}}\lesssim \|f\|_{L^{p}}\|g\|_{L^{2}},
	\end{equation*}
	for $p\in[2,\infty].$
\end{cor}
\begin{proof}
	According to  condition \eqref{support} we have 
	\begin{equation*} 
		\left(\left[\mathcal{T}_{\psi}; f\right]g\right)(x)
		%&=f(x)\left(\mathcal{T}_{\psi}g\right)(x)\\
		%&=\sum_{j=2}^{\infty} {\alpha/2 \choose j}f(x)\left(\mathcal{B}_{2j-2}*g\right)(x)\\
		=\sum_{j=2}^{\infty}{\alpha/2 \choose j}f(x)\int_{\{|x-y|>\delta\}}g(y)\mathcal{B}_{2j-2}(x-y)\,dy.
	\end{equation*}
	In virtue of Lemma \ref{Lemmasimpt}  is clear that  for  $j>1,$
	\begin{equation*}
		\begin{split}
			&\left|	f(x)\int_{\{|x-y|>\delta\}}g(y)\mathcal{B}_{2j-2}(x-y)\,dy\right|\\
			&\lesssim_{n}\frac{1}{2^{j}(j-2)!}\int_{\{|x-y|>\delta\}}|f(x)||g(y)| |2\pi(x-y)|^{\frac{2j-3-n}{2}}e^{-2\pi|x-y|}\,dy.\\
		\end{split}
	\end{equation*}
	Thus, by Young's inequality 
	\begin{equation*}
		\begin{split}
			\left\|	\left[\mathcal{T}_{\psi}; f\right]g\right\|_{L^{2}}&\lesssim_{n}\|f\|_{L^{p}}\|g\|_{L^{2}}\left(\sum_{j=2}^{\infty}\frac{\Gamma\left(j+\frac{n-3}{2}\right)}{(j-2)!j^{1+\alpha}}\right)\\
			%	&\lesssim_{\alpha,n}\|f\|_{L^{\infty}}\|g\|_{L^{p}}\left(\sum_{j=2}^{\infty}\frac{j^{\frac{n}{2}-1-\alpha}}{2^{j}}\right)\\
			&\lesssim_{n,\alpha}\|f\|_{L^{p}}\|g\|_{L^{2}},
		\end{split}
	\end{equation*}
	for  $p\in[2,\infty].$ 
\end{proof}
The following formulas where firstly obtained  by Linares, Kenig, Ponce and Vega  in the one dimensional case  in their study about propagation of regularity fo solutions of the KdV equation. Nevertheless, these results where later extended  to  dimension $n,n\geq 2$   in the work of  \cite{AMZK}     for  solutions of the  Zakharov-Kuznetsov equation.
\begin{lem}[Localization formulas]\label{A}
	Let $f\in L^{2}(\mathbb{R}^{n}).$ Let   $\nu=(\nu_{1},\nu_{2},\dots,\nu_{n})\in  \mathbb{R}^{n}$ a non-null vector  such that $\nu_{j}\geq 0,\, j=1,2,\dots,n.$  Let $\epsilon>0,$   we consider  the function $\varphi_{\nu,\epsilon}\in C^{\infty}(\mathbb{R}^{n})$   to satisfy: $0\leq\varphi_{\nu,\epsilon} \leq 1,$  
	\begin{equation*}
		\varphi_{\nu,\epsilon}(x)=
		\begin{cases}
			0\quad \mbox{if}\quad & x\in \mathcal{H}_{\left\{
				\nu,\frac{\epsilon}{2}\right\}}^{c}\\
			1\quad \mbox{if}\quad & x\in\mathcal{H}_{\{\nu,\epsilon\}}
		\end{cases}	
	\end{equation*}
	and the following increasing property:
	for every multi-index $\alpha$ with $|\alpha|=1$
	\begin{equation*} 
		\partial^{\alpha}_{x}\varphi_{\nu,\epsilon}(x)\geq 0,\quad x\in\mathbb{R}^{n}. 
	\end{equation*}
	\begin{itemize}
		\item[(I)] If  $m\in \mathbb{Z}^{+}$ and $\varphi_{\nu,\epsilon}J^{m}f\in L^{2}(\mathbb{R}^{n}),$  then for  all $\epsilon'>2\epsilon$  and all multi-index $\alpha$ with   $0 \leq|\alpha|\leq m,$ the derivatives of  $f$ satisfy
		\begin{equation*}
			\varphi_{\nu,\epsilon'}\partial^{\alpha}_{x}f\in L^{2}(\mathbb{R}^{n}).
		\end{equation*} 
		
		\item[(II)]If $m\in \mathbb{Z}^{+}$ and $\varphi_{\nu,\epsilon}\,\partial^{\alpha}_{x}f\in L^{2}(\mathbb{R}^{n})$ for all multi-index $\alpha$ with $ 0\leq |\alpha|\leq m,$  then for all $\epsilon'>2\epsilon$
		\begin{equation*}
			\varphi_{\nu,\epsilon'}J^{m}f\in L^{2}(\mathbb{R}^{n}).
		\end{equation*}
		\item[(III)] If $s>0,$  and $J^{s}(\varphi_{\nu,\epsilon}f)\in L^{2}(\mathbb{R}^{n}),$ then for any  $\epsilon'>2\epsilon$
		\begin{equation*}
			\varphi_{\nu,\epsilon'}\,J^{s}f\in L^{2}(\mathbb{R}^{n}).
		\end{equation*}
		\item[(IV)] If $s>0,$  and  $\varphi_{\nu,\epsilon}J^{s}f\in L^{2}(\mathbb{R}^{n}),$ then for any $\epsilon'>2\epsilon$
		\begin{equation*}
			J^{s}\left(\varphi_{\nu,\epsilon'}f\right)\in L^{2}(\mathbb{R}^{n}).
		\end{equation*}
	\end{itemize}
\end{lem}
\begin{proof}
	See \cite{AMZK}.
\end{proof}
A more general version  that   we are going to use in this work.
\begin{lem}\label{lemm}
	Let $f\in L^{2}(\mathbb{R}^{n}).$ If   $\theta_{1}, \theta_{2}\in C^{\infty}(\mathbb{R}^{n})$  are functions such that:  $0\leq \theta_{1},\theta_{2}\leq 1,$   their  respective supports satisfy
	\begin{equation*}
		\dist\left(\supp\left(1-\theta_{1}\right), \supp\left(\theta_{2}\right)\right)\geq \delta,
	\end{equation*}
	for some positive number  $\delta,$ and for all multi-index $\beta,$   the functions $\partial_{x}^{\beta}\theta_{1},\partial_{x}^{\beta}\theta_{2}\in L^{\infty}(\mathbb{R}^{n}).$
	
	Then, the following identity  holds:
	\begin{itemize}
		\item[(I)] If  $m\in \mathbb{Z}^{+}$ and $\theta_{1}J^{m}f\in L^{2}(\mathbb{R}^{n}),$  then for  all multi-index $\alpha$ with   $0 \leq|\alpha|\leq m,$ the derivatives of  $f$ satisfy
		\begin{equation*}
			\theta_{2}\partial^{\alpha}_{x}f\in L^{2}(\mathbb{R}^{n}).
		\end{equation*} 
		
		\item[(II)]If $m\in \mathbb{Z}^{+}$ and $\theta_{1}\,\partial^{\alpha}_{x}f\in L^{2}(\mathbb{R}^{n})$ for all multi-index $\alpha$ with $ 0\leq |\alpha|\leq m,$  then 
		\begin{equation*}
			\theta_{2}J^{m}f\in L^{2}(\mathbb{R}^{n}).
		\end{equation*}
		\item[(III)] If $s>0,$  and $J^{s}(\theta_{1}f)\in L^{2}(\mathbb{R}^{n}),$ then 
		\begin{equation*}
			\theta_{2}\,J^{s}f\in L^{2}(\mathbb{R}^{n}).
		\end{equation*}
		\item[(IV)] If $s>0,$  and  $\theta_{1}J^{s}f\in L^{2}(\mathbb{R}^{n}),$ then 
		\begin{equation*}
			J^{s}\left(\theta_{2}f\right)\in L^{2}(\mathbb{R}^{n}).
		\end{equation*}
	\end{itemize}
\end{lem}
\begin{proof}
	See  \cite{AMZK}.
\end{proof}
\begin{lem}\label{zk19}
	Let $\Psi_{a}\in \mathrm{OP\mathbb{S}^{0}}.$ Let $\theta_{1},\theta_{2}:\mathbb{R}^{n}\longrightarrow\mathbb{R}$ be smooth functions such that 
	\begin{equation*}
		\dist\left(\supp\left(1- \theta_{1}\right),\supp \theta_{2}\right)>\delta,
	\end{equation*}
for some $\delta>0.$

	 Assume that 
	$ f\in H^{s}(\mathbb{R}^{n}),\, s<0.$   If   $\theta_{1}f\in L^{2}(\mathbb{R}^{n}),$ then 
	\begin{equation*}
		\theta_{2}\Psi_{a} f\in L^{2}(\mathbb{R}^{n}).
	\end{equation*}
\end{lem}
\begin{proof}
	For the one dimensional case  see \cite{KLPV} and  the extension to the $n-$dimensional case   see  \cite{AMZK}.
\end{proof}
\begin{lem}\label{zk37}
	Let $f\in L^{2}(\mathbb{R}^{n})$  and  $\nu=(\nu_{1},\nu_{2},\dots,\nu_{n})\in \mathbb{R}^{n}$ such that $\nu_{1}>0,$ for $j=1,2,\dots,n.$  Also assume   that   
	\begin{equation*}
		J^{s}f\in L^{2}\left(\mathcal{H}_{\{\alpha,\nu\}}\right),\quad s>0.
	\end{equation*}
	Then, for any $\epsilon>0$ and any $r\in (0,s]$
	\begin{equation*}
		J^{r}f\in L^{2}\left(\mathcal{H}_{\{\alpha+\epsilon,\nu\}}\right).
	\end{equation*}
\end{lem}
\begin{proof}
	For the one dimensional case  see \cite{KLPV} and  the extension to the $n-$dimensional case   see  \cite{AMZK}.
\end{proof}
\begin{thm}\label{KPDESI}
	Let $s>0$ and  $f,g\in\mathcal{S}(\mathbb{R}^{n}).$ Then,
	\begin{equation*}
		\left\|\left[J^{s};g\right]f\right\|_{L^{2}}\lesssim \|J^{s-1}f\|_{L^{2}}\|\nabla g\|_{L^{\infty}}+\|J^{s}g\|_{L^{2}}\|f\|_{L^{\infty}},
	\end{equation*}
where the implicit constant does not depends  on $f$ nor $g.$
\end{thm}
\begin{proof}
	For the proof see the appendix in Kato and  Ponce \cite{KATOP2}.
\end{proof}
Also, the  following Leibniz rule  for the operator $J^{s}$ is quite  useful in our arguments
\begin{thm}\label{leibniz}
	Let $s>\frac{n}{2}$ and $f,g\in \mathcal{S}(\mathbb{R}^{n}),$ then
	\begin{equation*}
		\|J^{s}(f\cdot g)\|_{L^{2}}\lesssim \|J^{s}f\|_{L^{2}}\|g\|_{L^{\infty}}+\|J^{s}g\|_{L^{2}}\|f\|_{L^{\infty}},
	\end{equation*}
where the implicit constant does not depends  on $f$ nor $g.$
\end{thm}
\begin{proof}
See the appendix in Kato and  Ponce \cite{KATOP2}.
\end{proof}
\section{Acknowledgment}
 I shall thank   R. Freire for reading an early version of this document and its helpful suggestions. I also would like to thank  Prof. F. Linares for pointing out several references and the suggestions  in the presentation of this work. I also thanks O. Riaño for calling my attention to this problem.


\begin{thebibliography}{9}
	\bibitem{ABFS}
	{\sc L. Abdelouhab, J. Bona, M. Felland, J.-C. Saut,}\textit{  Nonlocal models for  nonlinear, dispersive waves,} Phys. D, 40 (3): 360-392,1989.
	
	\bibitem{abelh}
	{\sc H. Abels,}\textit{ Pseudodifferential and singular integral operators. An introduction with application,}. De Gruyter Graduate Lectures. De Gruyter, Berlin, 2012. x+222 pp. ISBN: 978-3-11-025030-5.
	
	\bibitem{astep}
	{\sc L. A. Abramyan, Y. A. Stepanyants, V. I. Shrira,}\textit{ Multidimensional solitons in shear flows of the boundary-layer type,}   Soviet Physics Doklady, 37(12):575-578, 1992.
	
	\bibitem{ARO}
	{\sc  N. Aronszajn, K. T. Smith,}\emph{ Theory of Bessel potentials I,}  Ann. Inst. Fourier. 11. 385–475.
	
	\bibitem{Bonasmith}
	{\sc J. L. Bona, R. Smith,}\textit{ The initial-value problem for the Korteweg–de Vries equation,}
		Philosophical Transactions of the Royal Society of London. Series A, Mathematical and
		Physical Sciences, 278:555-601, 1975.
	\bibitem{BL}
	{\sc J. Bourgain, D. Li,}\textit{\, On an endpoint Kato-Ponce inequality,} Differential Integral Equations 27. (11/12) 1037-1072.
	
	\bibitem{CZ}
	{\sc A.P. Calder\'{o}n, A. Zygmundn,}\textit{ Local properties  of solutions  of elliptic partial differential equations,} Studia Mathematica,  20 (1961) 2. 181-225. 
	
%	\bibitem{CMPS}
%	{\sc R. C\^{o}te, C. Mu\~{n}oz, D. Pilod, G. Simpson,} \textit{ Asymptotic stability  of high-dimensional Zakharov-Kuznetsov solitons,} Arc. Rational  Mech. Anal. 220 (2016) 639-710.
	
	\bibitem{CONSAUT}
	{\sc P. Constantin, P,  J.-C. Saut, }\textit{ Local smoothing properties of dispersive equations,} J. Amer. Math. Soc. 1 (1988), no. 2, 413-439.
	
	
	\bibitem{Dyachenko}
	{\sc A. I. D'yachenko, E.A. Kuznetsov,}\textit{ Two-dimensional wave collapse in the boundary layer,} The nonlinear Schr\"{o}dinger equation (Chernogolovka, 1994). Phys. D 87 (1995), no. 1-4, 301-313. 
	
	\bibitem{FAMI1}
	{\sc A. Faminskii,}\textit{ The Cauchy problem  for the Zakharov-Kuznetsov equation,} Differential Equations, 31:1002-1012, 1995.
	
	\bibitem{FMR}
	{\sc R.C. Freire, A.J.  Mendez, O. Riaño, }\textit{ On some regularity properties for the dispersive generalized Benjamin-Ono-Zakharov-Kuznetsov equation,}  J. Differential Equations 322 (2022), 135-179. 
	
	\bibitem{Ginibrev1}
	{\sc J. Ginibre, G.Velo, }\textit{ Commutator expansions and smoothing properties of generalized Benjamin-Ono equations,} Ann. Inst. H. Poincar\'{e} Phys. Th\'{e}or. 51 (1989), no. 2, 221-229.
	
	\bibitem{Gnibrev2}
	{\sc  J. Ginibre, G. Velo,}\textit{ Smoothing properties and existence of solutions for the generalized Benjamin-Ono equation,}  J. Differential Equations 93 (1991), no. 1, 150-212.
	
	\bibitem{GRUHER}
	{\sc A. Gr\"{u}nrock, S. Herr,}\textit{ The Fourier restriction norm method  for the Zakharov-Kuznetsov equation,} Discrete Contin. Dyn. Syst. 34 (2014), no. 5, 2061-2068. 
		
		\bibitem{herkino}
		{\sc S. Herr, S. Kinoshita,}\textit{ The Zakharov-Kuznetsov equation in high dimensions: small initial data of critical regularity,}  J. Evol. Equ. 21 (2021), no. 2, 2105-2121.
		
		\bibitem{HLKW}
	{\sc J. Hickman, F. Linares,  O. Riaño, R. Keith M, J. Wright,}\textit{ On a higher dimensional version of the Benjamin-Ono equation. } SIAM J. Math. Anal. 51 (2019), no. 6, 4544–4569. 
	
	\bibitem{IORIO}
	{\sc R. J. Iorio,}\textit{ On the  Cauchy problem  for the Benjamin-Ono equation,} Comm. in Partial Differential Equations, 11(10): 1031-1081, 1986.
	
	\bibitem{ILP1}
	{\sc P. Isaza, F. Linares, G. Ponce,}\textit{ On the propagation of regularity and decay  of solutions  to the $k-$generalized Korteweg-de Vries  equation,} Comm. Partial Differential Equations 40 (2015), pp 1336-1364.
	
	\bibitem{ILP2}
	{\sc P. Isaza, F. Linares, G. Ponce,}\textit{ On the propagation of regularities in solutions of the Benjamin-Ono equation,} J. Funct. Anal. 270 (2016), no. 3, 976-1000.
	
	\bibitem{ILP3}
	{\sc P. Isaza, F. Linares, G. Ponce,}\textit{ On the propagation of regularity of solutions of the Kadomtsev-Petviashvili equation,} SIAM J. Math. Anal. 48 (2016), no. 2, 1006-1024.
	\bibitem{lls}
	{\sc D. Lannes,  F. Linares, J.-C. Saut,}\textit{ The Cauchy problem for the Euler-Poisson system and derivation of the Zakharov-Kuznetsov equation,}  Studies in phase space analysis with applications to PDEs, 181-213, Progr. Nonlinear Differential Equations Appl., 84, Birkh\"{a}auser/Springer, New York, 2013.
%	
%	\bibitem{LINPO1}
%	{\sc F. Linares, G. Ponce,} \textit{} Personal communication.
	 \bibitem{SLP}
	{\sc  F. Linares, G. Ponce, D. L. Smith,} \textit{On the regularity of solutions to a class of nonlinear dispersive equations,}
	Math. Ann. 369 (2017), no. 1-2, 797-837.
	
	\bibitem{KATO1}
	{\sc T. Kato,}\textit{ On the  Cauchy problem for the (generalized) Korteweg-de Vries  equations,} Advances in  Mathematics  Supplementary  Studies , Stud. Math. 8 (1983) 93-128.
	
	\bibitem{KATOP2}
	{\sc T. Kato, G. Ponce,}\textit{ Commutator  estimates  and the Euler  and Navier-Stokes  equations,} Comm Pure Appl. Math  41 (1988) 891-907.
	
	\bibitem{KLPV}
	{\sc C.E. Kenig, F. Linares, G. Ponce, L. Vega,}
	\textit{ On the regularity of solutions to the k-generalized Korteweg–de Vries equation,} Proc. Amer. Math. Soc. 146 (2018), no. 9, 3759-3766. 
	
	\bibitem{KPVOS}
	{\sc C.E. Kenig, G. Ponce, L. Vega,}\textit{ Oscillatory integrals and regularity of dispersive equations,} Indiana Univ. Math. J. 40 (1991), no. 1, 33-69.
	\bibitem{KINO}
	{\sc S. Kinoshita,}\textit{ Global well-posedness for the Cauchy problem of the Zakharov-Kuznetsov equation in 2D,} Ann. Inst. H. Poincar\'{e} C Anal. Non Lin\'{e}aire 38 (2021), no. 2, 451–505.
	
	\bibitem{KDV}
	{\sc D.J. Korteweg, G. de Vries,}\textit{ On the change of  form of long waves advancing  in a rectangular canal, and  on anew type of long stationary  waves,} Philos. Mag. 39 (1895), 422-443.
	
	\bibitem{KF}
	{\sc S.N. Kruzhkov, A.V. Faminski\u{\i},}\textit{Generalized solutions of the Cauchy problem for the Korteweg-de Vries equation,} (Russian) Mat. Sb. (N.S.) 120(162) (1983), no. 3, 396-425.
	
	\bibitem{Kumno}
	{\sc H. Kumano-go,}\textit{ Pseudodifferential operators,} Translated from the Japanese by the author, Rémi Vaillancourt and Michihiro Nagase. MIT Press, Cambridge, Mass.-London, 1981. xviii+455 pp. ISBN: 0-262-11080-6.
	
	\bibitem{LIPAS}
	{\sc F. Linares, A. Pastor,} \textit{Well-posedness for the two-dimensional modified Zakharov-Kuznetsov equation,}  SIAM J. Math. Anal. 41 (2009), no. 4, 1323-1339.
	\bibitem{Lipolibro}
	{\sc  F. Linares, G. Ponce,}\textit{ Introduction to nonlinear dispersive equations,} Second edition. Universitext. Springer, New York, 2015. xiv+301 pp. ISBN: 978-1-4939-2180-5; 978-1-4939-2181-2.
	\bibitem{LPZK}
	{\sc Linares. F, Ponce. G,}\textit{ On special regularity properties of solutions of the Zakharov-Kuznetsov equation,} Commun. Pure Appl. Anal. 17 (2018), no. 4, 1561-1572. 
	
	\bibitem{LPPROPA}
	{\sc F. Linares, G. Ponce,}\textit{ Propagation of regularity  and decay of solutions  to nonlinear dispersive equations,} (2022), \url{https://w3.impa.br/%7Elinares/preprints/Linares-Ponce-RMC.pdf}. 
	
	\bibitem{LPS}
	{\sc F.Linares, G. Ponce, Gustavo, D.L.  Smith,}\textit{ On the regularity of solutions to a class of nonlinear dispersive equations,} Math. Ann. 369 (2017), no. 1-2, 797-837.
%	\bibitem{Maris}
%	{\sc M. Mari\c{s},} \textit{ On the existence, regularity and decay of solitary waves to a generalized Benjamin–
%		Ono equation,} Nonlinear Analysis, 51(6):1073-1085, 2002.
	
	\bibitem{MM}
	{\sc   S.Melkonian, S. Maslowe,}\textit{ Two-dimensional amplitude evolution equations for nonlinear dispersive waves on thin films,}  Phys. D 34 (1989), no. 1-2, 255-269.
	
	\bibitem{AM1}
	{\sc A. J. Mendez,}\textit{ On the propagation of regularity for solutions of the dispersion generalized Benjamin-Ono equation,} Anal. PDE 13 (2020), no. 8, 2399-2440.

	\bibitem{AM2}
	{\sc A. Mendez,}\textit{ On the propagation of regularity for solutions of the fractional Korteweg–de Vries equation,} J. Differential Equations 269 (2020), no. 11, 9051-9089.  
	
	\bibitem{AMZK}
	{\sc A.J. Mendez,}\textit{ On the propagation of regularity for solutions of the Zakharov-Kuznetsov  equation,}arXiv: 2008.11252v1. (2020).
	
	\bibitem{MMPP}
	{\sc  A.J. Mendez, C. Mu\~{n}oz, F.  Poblete, J.C.  Pozo, }\textit{ On local energy decay for large solutions of the Zakharov-Kuznetsov equation,} Comm. Partial Differential Equations 46 (2021), no. 8, 1440-1487.
		
	\bibitem{AMTHESIS}
	{\sc A. J. Mendez,}\textit{ On the propagation of regularity in some non-local dispersive models,} PhD thesis  at IMPA-(2019), \url{https://impa.br/wp-content/uploads/2019/08/tese_dout_Argenis-Jose-Mendez-Garcia.pdf}.
			\bibitem{Molipilo}
			{\sc L. Molinet, D. Pilod,}\textit{ Bilinear Strichartz estimates for the Zakharov-Kuznetsov equation and applications,}  Ann. Inst. H. Poincar\'{e} C Anal. Non Lin\'{e}aire 32 (2015), no. 2, 347-371.
			\bibitem{ailton}
			{\sc A.C. Nascimento,} \textit{ On special regularity properties of solutions of the Benjamin-Ono-Zakharov-Kuznetsov (BO-ZK) equation,}  Commun. Pure Appl. Anal. 19 (2020), no. 9, 4285-4325. 
			

	\bibitem{RRY}
	{\sc O. Ria\~{n}o, S. Roudenko, K. Yang,}\textit{ Higher dimensional generalization of the Benjamin-Ono equation: 2D case,}  
		Stud. Appl. Math. 148 (2022), no. 2, 498-542.
	
	\bibitem{oscarmari}
	{\sc O. G.  Ria\~{n}o,}\textit{ The IVP for a higher  dimensional  version  of the Benjamin-Ono equation in weighted Sobolev spaces,} J.Funct. Anal., 279 (8): 108707, 2020.
	
%	\bibitem{SACONS}
%	{\sc P.Constantin,J.-C. Saut,} \textit{ Local smoothing properties of dispersive equations,}  J. Amer. Math. Soc. 1 (1988), no. 2, 413-439.
	
	\bibitem{Scrira pelynosvky}
{\sc D. E. Pelinovsky, V. I. Shrira,}\textit{ Collapse transformation for self-focusing solitary waves in
	boundary-layer type shear flows,} Physics Letters A, 206(3):195-202, 1995.

\bibitem{PSTEP}
{\sc D.E. Pelinovsky,  Y.A. Stepanyants, Y.S. Kivshar,}\textit{ Self-focusing of plane dark solitons in nonlinear defocusing media,} Phys. Rev. E (3) 51 (1995), no. 5, part B, 5016-5026.

\bibitem{GP}
{\sc G. Ponce,}\textit{ Smoothing properties of solutions to the Benjamin-Ono equation, } Analysis and partial differential equations, 667-679, Lecture Notes in Pure and Appl. Math., 122, Dekker, New York, 1990.

\bibitem{RIBAUDVENT}
{\sc F. Ribaud, S. Vento,}\textit{ Well-posedness results for the three-dimensional Zakharov-Kuznetsov equation,}  SIAM J. Math. Anal. 44 (2012), no. 4, 2289-2304.

	\bibitem{RS}
	{\sc R. Schippa,}\textit{ Local and global well-posedness for dispersion generalized Benjamin-Ono equations on the circle,}
	 Nonlinear Anal. 196 (2020), 111777, 38 pp.
	 
	 \bibitem{SS}
	 {\sc Segata, Jun-Ichi; Smith, Derek L,}\textit{ Propagation of regularity and persistence of decay for fifth order dispersive models,} J. Dynam. Differential Equations 29 (2017), no. 2, 701–736.
	\bibitem{Schira}
	{\sc V. Shrira,}\textit{ On the subsurface waves in the oceanic upper mixed layer,}  Dokl. Akad. Nauk SSSR, 308(3):732-736,
		1989.
	
	\bibitem{SJOLin}
	{\sc  P.Sj\"{o}lin,}\textit{ Regularity of solutions to the Schr\"{o}dinger equation,}  Duke Math. J. 55 (1987), no. 3, 699-715.
	
	\bibitem{stein1}
	{\sc E.M. Stein,}\textit{ Singular Integrals and Differentiability Properties  of Functions,} Princeton  University Press, Princeton, 1970. 
	
%	\bibitem{stein2}
%	{\sc E.M. Stein, }\textit{ Harmonic analysis real-variable methods, orthogonality, and oscillatory integrals,} Princeton University Press, Princeton, N.J. 1993.
	 
	 \bibitem{stein3}
	{\sc E.M. Stein, }\textit{ Harmonic analysis: real-variable methods, orthogonality, and oscillatory integrals,} Princeton Mathematical Series, vol. 43, Princeton University Press, Princeton, NJ, 1993. With the assistance of Timothy S. Murphy; Monographs in Harmonic Analysis, III. MR 1232192.

\bibitem{MT1}
{\sc  M. E. Taylor,}\textit{ Pseudodifferential operators. Princeton Mathematical Series,} No. 34. Princeton University Press, Princeton, N.J., 1981. xi+452 pp. ISBN: 0-691-08282-0.

%\bibitem{MUDI}
%{\sc  M.M. Tom,}\textit{ Smoothing properties of some weak solutions of the Benjamin-Ono equation,}  Differential Integral Equations 3 (1990), no. 4, 683-694.
%\bibitem{Voroshrira}
%{\sc V. V. Voronovich,  V. I. Shrira,} \textit{ Internal wave–shear flow resonance and wave breaking in the subsurface layer,} In Nonlinear instability analysis, volume II of Advances in fluid mechanics, 133-177, WIT Press, 2001.

\bibitem{VEGA}
{\sc L. Vega,}\textit{ Schr\"{o}dinger equations: pointwise convergence to the initial  data,} Proc. A.M.S. 102 (1988), 874-878.
\bibitem{ZAKHARIV}
{\sc V. Zakharov,  E. Kuznetsov,} \textit{ On three dimensional solitons,}  J. Exp. Theor. Phys., 39:285-286, 1974.
\end{thebibliography}
\end{document}